\documentclass[a4paper,12pt, twoside]{article}
\usepackage{amsfonts}
\usepackage{CJK}
\usepackage{amsthm,amsmath,amssymb}
\usepackage{enumerate,bm}
\usepackage{graphicx}
\usepackage{subfigure}
\usepackage{tikz}
\usepackage{epstopdf}
\usepackage{layout}
\usepackage{mathrsfs}
\usepackage{epsfig}
\usepackage{indentfirst,latexsym,bm}
\usepackage{pifont}
\usepackage{txfonts}
\usepackage{xcolor}
\usepackage{fancyhdr}
\usepackage[colorlinks,driverfallback=dvipdfmx,linkcolor=black,anchorcolor=black,citecolor=black,plainpages=false,pdfpagelabels]{hyperref}
\usepackage{bookmark}

\newtheorem{theorem}{Theorem}[section]
\newtheorem{lemma}[theorem]{Lemma}
\newtheorem{remark}[theorem]{Remark}

\def\beqbel{\begin{equation}\label}
\def\eq{\end{equation}}
\def\beqq{\begin{equation*}}
\def\eqq{\end{equation*}}
\def\f{\frac}
\def\vp{\varphi}
\def\p{\partial}

\def\ds{\displaystyle}
\def\ve{\epsilon}

\numberwithin{equation}{section}

\begin{document}

\title{Formation and construction of a shock wave for 1-D
$n\times n$ strictly hyperbolic conservation laws with small smooth initial data \footnote{The research of Ding Min was supported by the National Natural Science Foundation of China under Grant Nos.12371226, 11701435,
the Natural Science Foundation of Hubei Province (2021CFB452), and the Fundamental Research Funds for the Central Universities.
The research of  Yin Huicheng was supported by the National Natural Science Foundation of China (No.12331007)
and the National key research and development program (No.2020YFA0713803).}}
\author{Ding Min$^{1}$, \quad Yin Huicheng$^{2}$
\footnote{Email addresses: Ding Min (minding@whut.edu.cn),\quad Yin Huicheng (huicheng@nju.edu.cn,\; 05407@njnu.edu.cn)}.}
\date{}

\maketitle
\centerline{1. Department of Mathematics, School of Mathematics and Statistics, Wuhan University of}
\centerline{Technology, Wuhan, 430070, China.}
\vspace{0.1cm}
\centerline{2. School of Mathematical Sciences and Institute of Mathematical Sciences,}
\centerline{Nanjing Normal University, Nanjing, 210023, China.}
\vspace{0.1cm}

\begin{abstract}
Under the genuinely nonlinear assumption for 1-D $n\times n$ strictly  hyperbolic conservation laws,
we investigate the geometric blowup of smooth solutions and the development
of singularities when the small initial data fulfill the generic
nondegenerate condition.
At first, near the unique blowup point we give a precise description on the space-time blowup rate of the smooth solution
and meanwhile derive  the cusp singularity structure of characteristic envelope.
These results are established through extending the smooth solution of the completely nonlinear blowup
system across the blowup time. Subsequently, by utilizing
a new form on the resulting 1-D strictly hyperbolic system with $(n-1)$ good components and one bad component,
together with the choice of an efficient iterative scheme and some involved analyses,
a weak entropy shock wave starting from the blowup point is constructed. As a byproduct,
our result can be applied to the shock formation and construction for the
2-D supersonic steady  compressible full Euler equations ($4\times 4$ system),
1-D MHD equations ($5\times 5$ system), 1-D elastic wave equations ($6\times 6$ system)
and 1-D full ideal compressible MHD equations ($7\times 7$ system).
\end{abstract}


\qquad \quad \quad generic nondegenerate condition, shock formation


\tableofcontents

\section{Introduction}

As is well known, no matter how smooth and how small the initial data are, the classical solutions
will generally form singularities in finite time for 1-D strictly quasilinear hyperbolic
conservation laws with genuinely nonlinear structures (see \cite{A1}, \cite{Horm}, \cite{LP-1}, \cite{LT} and  \cite{Xin-1}).
Therefore, it is important to understand the physical process of singularity development from the smooth solutions
and the evolution of singularities starting from the blowup points.

Consider the following Cauchy problem for 1-D  $n\times n$ strictly hyperbolic conservation law
\begin{equation}\label{eq:1.1}
\begin{cases}
\partial_t u+\partial_x f(u)=0,\\
u(x, 0)=\epsilon u_0(x),
\end{cases}
\end{equation}
where $t\ge 0$, $x\in\mathbb{R}$, $\epsilon>0$ is a sufficiently small parameter, $u=(u_1, \cdots, u_n)^{\top}\in\mathbb{R}^n$,
$f(u)=(f_1(u), \cdots, f_n(u))^{\top}\in C^{\infty}$,
and
\begin{equation}\label{eq:1.2}
u_0(x)=(u^{1}_0(x), \cdots, u^{n}_0(x))^{\top}\in C^{\infty}_0(\mathbb{R}), \quad \text{supp }u_0(x)\subset [a, b],
\end{equation}
here $a$ and $b$ with $a<b$ are constants.

For $C^1$ solution $u$,  the system in $\eqref{eq:1.1}_1$ can be rewritten as
\begin{align}\label{Y-1}
\partial_t u+F(u)\partial_x u=0,
\end{align}
where $F(u)=\p_u f(u)$ is an $n\times n$ matrix. By the strict hyperbolicity of \eqref{Y-1},
$\det(\lambda \mathbb{I}_n-F(u))=0$ has $n$ distinct real eigenvalues, denoted by
\begin{equation}\label{eq:1.3}
\lambda_1(u)<\lambda_2(u)<\cdots<\lambda_n(u),
\end{equation}
where $\mathbb{I}_n$ is the $n\times n$ identity matrix. The corresponding left and  right eigenvectors of $F(u)$
are $l_1(u)$, $\cdots$, $l_n(u)$, and  $r_1(u)$, $\cdots$, $r_n(u)$, respectively. If
\[
\nabla_u \lambda_j(u)\cdot r_j(u)\neq 0 \quad \text{ for all } u,
\]
then \eqref{Y-1} is called genuinely nonlinear with respect to $\lambda_j$. Otherwise, if
\[
\nabla_u \lambda_j(u)\cdot r_j(u)\equiv 0\quad  \text{ for all } u,
\]
then \eqref{Y-1} is called linearly degenerate for $\lambda_j$.

When $\eqref{Y-1}$ is genuinely nonlinear for all eigenvalues $\lambda_{j}(u)$ ($1\le j\le n$), and
the initial data $u(x,0)=\psi(x)\in C^2$ satisfy that $\text{supp}\psi\subseteq [a, b]$ holds
and $(b-a)\sup_{x\in\mathbb{R}}|\psi''(x)|$ is sufficiently small (but $\psi\not\equiv 0$),
F. John \cite{John} introduced the wave decomposition method to prove that the first order derivatives of $u$ blow up in finite time.
T.P. Liu \cite{Liu} generalized the result of \cite{John} to the system $\eqref{Y-1}$ with the features
that some eigenvalues are genuinely nonlinear and others are linearly degenerate.
L. H\"{o}rmander \cite{Horm1, Horm} improved the results in \cite{John} and \cite{Liu} such that a sharp estimate for the lifespan
$T_\epsilon$ of smooth solutions is established.
As shown in \cite{Horm1, Horm},
the lifespan $T_\epsilon$ of $\eqref{eq:1.1}$ satisfies
\begin{equation}\label{eq:1.5}
\lim_{\epsilon\to 0^+}\epsilon T_{\epsilon}=-\ds\f{1}{\ds\min_{1\le j\le n}M_j},
\end{equation}
where $M_j=\min_{x}\ds\f{\nabla_u\lambda_j(0)r_j(0)}{l_j(0)r_j(0)}l_j(0)u_0'(x)$.
Without loss of generality, it is assumed that for some fixed $i$ ($1\le i\le n$),
\begin{equation}\label{HC-2}
M_i<\ds\min_{j\not=i}M_j,
\end{equation}
which also means that \eqref{Y-1} is genuinely nonlinear with respect to $\lambda_i(u)$.

In addition, employing the method of geometric optics, S. Alinhac \cite{A2} reconsidered the result in \cite{Horm1},
and gave a more precise description on $T_\epsilon$ through the asymptotic expansion form of $\epsilon$.
Recently, motivated by \cite{B-2}, through taking the efficient  decomposition of $u$ along the different
characteristic directions, the involved analysis on the characteristics with large variations,
the suitable introduction of the modulated coordinates together with the global weighted energy estimates
and the characteristics method,
Li-Xu-Yin \cite{LXY} have established the geometric blowup mechanism of smooth solutions to \eqref{Y-1}
for a class of large variational initial data $u(x,0)$.  The
geometric blowup by the terminology in \cite{A1} means that up to the blowup time $T_*$, the solution $u\in C([0, T_*)\times\Bbb R)\cap
L^{\infty}([0, T_*]\times\Bbb R)$ but $\|\nabla_{x,t}u(\cdot, t)\|_{L^{\infty}}\to\infty$  holds as $t\to T_{*}-$.

In the present paper, we wonder to know under which condition on smooth initial data, a shock can be generated,
and how it develops from the blowup point for general 1-D $n\times n$ hyperbolic conservation laws ($n\geq 3$).
For the 1-D or M-D scalar conservation laws, Yin-Zhu \cite{YinZ, YinZ-1} have solved the
problem of shock formation and construction under the various assumptions of initial data. For 1-D scalar
convex conservation law, Chen-Zhang \cite{CZ} showed the shock formation of solutions for piecewise smooth initial data
with finite discontinuities. For the $p$-system of gases dynamics,
under the assumptions that one Riemann invariant is a constant and the initial data satisfy the related generic
nondegenerate condition,
M.P. Lebaud \cite{Lebaud} constructed a shock solution from the blowup point. This result was extended to the more
general case of $p$-system in \cite{ChenD, Kong}, where both the  Riemann invariants are not constants and
it is additionally assumed that only one family of characteristics is  squeezed, while the other characteristics family does not
squeeze at the same point. Here it is pointed out that the existence of Riemann invariants plays a crucial role in the
analysis of \cite{ChenD, Kong, Lebaud}  since the $p$-system can be diagonalized in this situation.
For the $3\times 3$ case of \eqref{eq:1.1}, Chen-Yin-Xin \cite{CXY} introduced a suitable invertible transformation $u=u(w)$
to find a new unknown function $w=(w_1,w_2, w_3)$
with two good components $(w_1,w_3)$ and one bad component $w_2$, and constructed a $2$-shock starting from the blowup point
under the
generic nondegenerate condition of initial data and by the extension result of smooth solutions to
the resulting blowup system across the blowup time in \cite{A2}.
For the 3-D full compressible Euler equations
with spherical symmetric structure, a symmetric shock solution after the blowup time is constructed in \cite{Yin}
(also see the independent work of \cite{C-L}). In addition, Buckmaster-Drivas-Shkoller-Vicol  \cite{BDSV-4} studied the simultaneous development
and cusps for the 2-D compressible Euler system with azimuthal symmetric smooth data.

Benefiting from the analysis of the blowup mechanism of smooth solutions
for the $3\times 3$ strictly hyperbolic system \eqref{Y-1} in \cite{A2}, and the construction of
a shock for the 1-D $3\times 3$ hyperbolic conservation law $\eqref{eq:1.1}_1$ in \cite{CXY},
we shall study the shock formation and construction of \eqref{eq:1.1}
with $n\ge 3$. For this purpose, at first, we establish the geometric blowup mechanism of \eqref{eq:1.1}
and extend the smooth solution of the resulting blowup system (see \eqref{eq:2.1} below)
across the lifespan $T_{\epsilon}$.
Note that in order to treat the $3\times 3$ case in \eqref{eq:1.1},
S. Alinhac \cite{A2} utilized
the geometric optics method and the special properties of $3\times 3$ system that the solution $u$
is $1$-simple on the left side of 2-characteristics $\Gamma_a^2$ through $(a, 0)$,
and $3$-simple on the right side of 2-characteristics $\Gamma_b^2$ through $(b, 0)$, respectively.
Namely, $u_2=u_3=0$ on $\Gamma_a^2$ and $u_1=u_2=0$ on $\Gamma_b^2$ hold.
However, for the $n\times n$ system \eqref{Y-1} with $n\geq 4$, the smooth boundary values
of $u$ on the i-characteristics $\Gamma_a^i$ through $(a, 0)$ and $\Gamma_b^i$ through $(b, 0)$ are unknown
($2\le i\le n-1$), moreover, they can  be usually determined in the time interval $[0, T_{\ve}+\delta_0]$
with $\delta_0>0$ being small so that the determined domains for the points at  $\Gamma_a^i$ and $\Gamma_b^i$
do not include the blowup point at $T_{\ve}$. This means that it is difficult for us to
introduce the slow time variable $\tau=\ve t$ and utilize the geometric optics method to deal with the related blowup system
for $\tau\in [\tau_{\ve}, \tau_{\ve}+1]$ with $\tau_{\ve}=\f{T_{\ve}}{\ve}$ as in  \cite{A2}.
Our strategy is to solve the blowup system of $\eqref{eq:1.1}_1$ directly
by deriving the precise smallness property of boundary values on $\Gamma_a^i$ and $\Gamma_b^i$
along their tangential directions separately (see \eqref{eq:2.98} below) and through some careful
observations on the nonlinear structure of blowup system.
Based on the extension property of smooth solution to
the blowup system across $T_{\ve}$  and the cusp property of characteristic envelope,
through choosing a new form of $\eqref{eq:1.1}_1$ such that its solution is more singular along
one direction than other left directions, and taking the corresponding iterative scheme,
we can construct the weak entropy shock solution issuing from the blowup point,
meanwhile, the detailed descriptions on the location of the shock as well as the estimates of the solution near the blowup point
are also given. Although the main argument procedures for the uniform boundedness and convergence of the iterative scheme of
approximate shock solutions are analogous to those in \cite{CXY} and \cite{Yin},
we still give all the details due to the general forms of $n\times n$ cases together with more precise and complete
computations.

With the aid of Lemma \ref{Y-4} below and \eqref{eq:1.5}, it follows from direct computation that
$M_j=\min_{x}N_j(x)$ with $N_j(x)=\partial_{w_j}\lambda_j(0)(w_0^j)'(x)$ holds,
where $w_j, w_0^j(x)$ are defined in Lemma \ref{Y-4} later.
Under assumption \eqref{HC-2}, the following generic nondegenerate condition is imposed:

\quad {\it There exists a unique point $x_0$ such that}
\begin{equation}\label{Eq:2.1}
\mathcal{N}_i=N_i(x_0)=\min_x N_i(x), \quad N'_i(x_0)=0, \quad N''_i(x_0)>0.
\end{equation}

The main result in this paper can be stated as follows.

\begin{theorem}\label{DY-1}
Provided that the generic nondegenerate condition \eqref{Eq:2.1} holds
and $\eqref{eq:1.1}_1$ is genuinely nonlinear with respect to the  $i$-th eigenvalue
$\lambda_i(u)$ ($1\leq i\leq n$),
there exists a unique solution $u(x,t)\in C(\mathbb{R}\times[0, T_{\epsilon}])\cap C^1(\mathbb{R}\times [0, T_{\epsilon}))$ to \eqref{eq:1.1}
which produces the geometric blowup at the unique point $(x_\epsilon, T_\epsilon)$. Moreover, problem \eqref{eq:1.1}  admits a weak entropy solution with an $i-$shock curve $x=\phi(t)\in C^1([T_{\epsilon}, T_\epsilon+\delta_0])$
starting from $(x_\epsilon, T_\epsilon)$ for some small positive constant $\delta_0$, which
satisfies
\begin{itemize}
\item[(i)] near $(x_\epsilon, T_\epsilon)$ and for $t\in [T_\epsilon, T_\epsilon+\delta_0]$,
\begin{align}\label{CYY-2}
\phi(t)=&x_{\epsilon}+\lambda_i(u(x_{\epsilon}, T_{\epsilon}))(t-T_\epsilon)+O(1)(t-T_\epsilon)^2.
\end{align}
\item[(ii)] $u(x, t)\in C^1\Big(\mathbb{R}\times(T_\epsilon, T_{\epsilon}+\delta_0)\Big)\backslash \{x=\phi(t)\}$ and
\begin{align}\label{CYY-3}
u(x, t)=&u(x_{\epsilon}, T_\epsilon)+O(1)\Big((t-T_\epsilon)^3+\big(x-x_\epsilon-\lambda_i(u(x_{\epsilon}, T_{\epsilon}))(t-T_\epsilon)\big)^2\Big)^{\frac16},
\end{align}
where $O(1)$ represents a generic bounded quantity  independent of $\epsilon$.
\end{itemize}
\end{theorem}

\begin{remark}\label{CYY-5}
By the completely analogous proof procedure, Theorem \ref{DY-1} for the 1-D strictly hyperbolic system
can be extended into the case of  1-D symmetric hyperbolic system.
\end{remark}

\begin{remark}
About the geometric blowup of the smooth solution $u(x,t)$ to \eqref{eq:1.1}
at $(x_\epsilon, T_\epsilon)$ in Theorem  \ref{DY-1},
the more precise descriptions will be given in Theorem \ref{T:2.3}  together with
Theorem  \ref{T:2.1} below.
\end{remark}

\begin{remark}
We point out that the bound $O(1)$ in \eqref{CYY-2} and \eqref{CYY-3} can be improved to $\ve O(1)$
by checking the proof of Theorem \ref{DY-1} carefully since the amplitude of the solution $u$
is still of $\ve O(1)$ in $[T_{\epsilon}, T_\epsilon+\delta_0]$.
However, for brevity in this paper, we omit the small factor $\ve$ here.
\end{remark}

\begin{remark}
In recent years, the studies on the shock formation of smooth solutions to the
multidimensional hyperbolic conservation laws or the second order quasilinear wave equations
have made much progress (see \cite{B-2}, \cite{C1}, \cite{C2}, \cite{CM},
\cite{0-Speck}, \cite{Lu1}, \cite{MY} and \cite{S2}), which illustrate that the
formation of the multidimensional shock is due to the compression of
the characteristic surfaces. However, the related constructions of a multidimensional
shock wave after the blowup of smooth solutions are not obtained.
\end{remark}

The paper is organized as follows. In \S $2$, we study the geometric blowup mechanism
and extend the smooth solution of the blowup system across $T_{\ve}$ for problem \eqref{eq:1.1}.
In order to solve the  blowup system, some suitable boundary conditions and  boundary values are derived
by basic observations.
In addition, the precise descriptions on the formation and
construction of a shock wave are given. In \S $3$, close to the blowup point,
the crucial cusp properties and estimates on the pre-shock are obtained. In \S $4$,
by introducing a transformation to fix the free shock curve
and  taking a suitable iterative scheme to construct the approximate
shock wave solutions which satisfy the Rankine-Hugoniot conditions and the Lax's geometric entropy conditions,
we can trace the location of the approximate shock and get the estimates of approximate solutions.
In \S $5$, the convergence of the approximate shock solutions is shown. Subsequently, the main conclusions
in Theorem \ref{T:2.14} and Theorem \ref{DY-1} are proved. In \S $6$, as applications of Theorem \ref{DY-1},
we will give the related illustrations of shock formation for the 2-D supersonic steady
compressible full Euler equations ($4\times 4$ system),
1-D MHD equations ($5\times 5$ system), 1-D elastic wave equations ($6\times 6$ system)
and 1-D full ideal compressible MHD equations ($7\times 7$ system).

\section{Geometric blowup of the hyperbolic system}

\subsection{Simplification of \eqref{Y-1} and the resulting blowup system}

At first, motivated by \cite{CXY} for the $3\times 3$ case, we now
give a generalized simplification for the general $n\times n$ case ($n\ge 3$) of \eqref{Y-1} as follows.

\begin{lemma}\label{Y-4}
Assume that \eqref{eq:1.3} holds  and \eqref{Y-1} is genuinely nonlinear with respect to the eigenvalue
$\lambda_i(u)$ (fixed number $i$ with $1\le i\le n$). Then there exists an invertible transformation in the neighbourhood of the origin:
$ u\longrightarrow w(u)$ with $w(0)=0$ such that \eqref{Y-1} can be equivalently reduced into
\begin{equation}\label{eq:2.88}
\begin{cases}
\partial_t w+A(w)\partial_xw=0,\\[5pt]
w(x, 0)=\ve w_0(x)+\ve^2w_1(x,\ve),
\end{cases}
\end{equation}
where $w_0(x)=(w^{1}_0(x), \cdots, w^{n}_0(x))^{\top}$ and $w_1(x,\ve)=(w_1^1(x,\ve), \cdots, w_1^n(x,\ve))^{\top}$ are $C^{\infty}$
with respect to their arguments and compactly supported in $[a, b]$
for the variable $x$. In addition, the $n\times n$ matrix $\displaystyle A(w)=(\partial_u w)F(u(w))(\partial_u w)^{-1}:=(a_{ij})_{n\times n}$
admits the following properties
\begin{itemize}
\item[$(1)$] the eigenvalues of $A(w)$ are $\lambda_1(u(w)), \cdots, \lambda_n(u(w))$, which are sometimes denoted by
$\lambda_1(w), \cdots, \lambda_n(w)$ respectively;
\item[$(2)$] for $j\neq i$, $a_{ji}=0$, and $a_{ii}=\lambda_i(u(w))$;
\item[$(3)$] the $i$-th right eigenvector of $A(w)$ is $r_i(w):=(\partial_u w)r_i(u(w))$,
which is parallel to the unit vector $(0, \cdots, 1, 0,\cdots, 0)^{\top}$;
\item[$(4)$] $A(0)=\text{diag}\Big(\lambda_1(0), \lambda_2(0), \cdots, \lambda_n(0)\Big)$.
\end{itemize}
\end{lemma}

\begin{proof}
From the definition  $7.3.1$ in \cite{Dafermos}, we know that there exist $(n-1)$ Riemann invariants $q_j(u)$ $(j\neq i)$ whose
 gradients are linearly independent and satisfy that for  $|u|\ll 1$
\begin{equation}\label{Eq:2.102}
\nabla_u q_j(u)\cdot r_i(u)=0, \qquad j\neq i.
\end{equation}

Let $\{\zeta_j\}_{j\neq i}$ be $(n-1)$ linearly independent column constant vectors orthogonal to $r_i(0)$. Inspired by \eqref{Eq:2.102},
set
\begin{equation}\label{Eq:2.103}
q_j(u)=\zeta_j\cdot u+\tilde{q}_j(u),
\end{equation}
where $\{\tilde{q}_j(u)\}_{j\neq i}$ satisfy
\begin{equation}\label{Eq:2.104}
\begin{cases}
\nabla_u\tilde{q}_j(u)\cdot r_i(u)=-\zeta_j\cdot(r_i(u)-r_i(0)),\\
\tilde{q}_j(0)=0, \quad j\neq i.
\end{cases}
\end{equation}

From the standard theory of the first order scalar quasilinear partial differential equations,
problem \eqref{Eq:2.104} is solved for $|u|\ll 1$ and $\tilde{q}_j(u)=O(|u|^2)$ holds.

Introduce a transformation: $u\to\tilde{u}$ as
\begin{equation}\label{Eq:2.105}
\tilde{u}_j=q_j(u) \quad \text{for $j\neq i$}, \qquad \tilde{u}_i=r^{\top}_i(0)\cdot u.
\end{equation}
Then
\[
\det(\partial_u\tilde{u})|_{u=0}=\det(\zeta_1, \cdots, \zeta_{i-1},\ r_i(0),\ \zeta_{i+1}, \cdots, \zeta_n)^{\top}\neq 0,
\]
and the mapping $u\to\tilde{u}(u)$ is invertible when $|u|\ll 1$.

Under the transformation \eqref{Eq:2.105}, the system \eqref{Y-1} can be reduced into
\begin{equation}\label{Eq:2.106}
\partial_t\tilde{u}+\tilde{A}(\tilde{u})\partial_x\tilde{u}=0,
\end{equation}
where $\tilde{A}(\tilde{u})=(\partial_u \tilde{u})F(u)(\partial_u \tilde{u})^{-1}
=\Big(\tilde{a}_{ij}(\tilde{u})\Big)_{n\times n}$. By direct calculations, it is known that $\tilde{A}(\tilde{u})$
has $n$ distinct eigenvalues
$\Big\{\lambda_k(\tilde{u})\Big\}^n_{k=1}$ with $\lambda_k(\tilde{u})=\lambda_k(u(\tilde{u}))$ and the corresponding right eigenvectors are $\Big\{(\partial_u\tilde{u})r_k(u)\Big\}^n_{k=1}$. Moreover, it holds that
\begin{equation*}
(\partial_u\tilde{u})r_i(u)=r^{\top}_i(0)\cdot r_i(u(\tilde{u}))\boldsymbol{e_i}\neq 0,
\end{equation*}
which implies
\[
\tilde{a}_{ji}(\tilde{u})=0 \;\; \text{for $j\neq i$}, \quad \tilde{a}_{ii}(\tilde{u})=\lambda_i(\tilde{u}).
\]

Let $\tilde{A}_{n-1}(\tilde{u})=(\tilde{a}_{ij})_{(n-1)\times(n-1)}$ be the $(n-1)$-th order square matrix, formed by getting
rid of the $i$-th row and $i$-th column of the matrix $\tilde{A}(\tilde{u})$. Then  $\tilde{A}_{n-1}(\tilde{u})$ has $(n-1)$ distinct
eigenvalues
$$\lambda_1(\tilde{u})<\cdots<\lambda_{i-1}(\tilde{u})<\lambda_{i+1}(\tilde{u})<\cdots<\lambda_n(\tilde{u}).$$

Therefore, there exists an invertible constant  square matrix $B_{n-1}=(b_{ij})_{(n-1)\times(n-1)}$ such that
\[
B_{n-1}\tilde{A}_{n-1}(0)B^{-1}_{n-1}=\text{diag}(\lambda_1(0), \cdots, \lambda_{i-1}(0), \lambda_{i+1}(0), \cdots, \lambda_n(0)).
\]

Let
\begin{equation}\label{Eq:2.108}
 \left(
 \begin{array}{c}
w_1\\[3pt]
\vdots\\
 \hat{w}_i\\[3pt]
 \vdots\\
w_n
 \end{array}
 \right)=B_{n-1} \left(\begin{array}{c}
\tilde{u}_1\\[3pt]
\vdots\\
 \hat{\tilde{u}}_i\\[3pt]
 \vdots\\
\tilde{u}_n
 \end{array}
 \right),
\end{equation}
where $(w_1, \cdots, \hat{w}_i, \cdots, w_n)^{\top}$ and $(\tilde{u}_1, \cdots, \hat{\tilde{u}}_i, \cdots, \tilde{u}_n)^{\top}$ represent the related $(n-1)$ dimensional vectors without the components $w_i$
and $\tilde{u}_i$, respectively.
It follows from \eqref{Eq:2.106} and \eqref{Eq:2.108} that
\begin{equation}\label{Eq:2.108-0}
\partial_t \left(
 \begin{array}{c}
w_1\\[3pt]
\vdots\\
 \hat{w}_i\\[3pt]
 \vdots\\
w_n
 \end{array}
 \right)+\bar{B}_{n-1}(w_1, \cdots, \tilde{u}_i,\cdots, w_n) \partial_x\left(
 \begin{array}{c}
w_1\\[3pt]
\vdots\\
 \hat{w}_i\\[3pt]
 \vdots\\
w_n
 \end{array}
 \right)=0,
\end{equation}
where
\[
\begin{split}
&\bar{B}_{n-1}(w_1, \cdots, \tilde{u}_i, \cdots, w_n)\\
=&B_{n-1}\tilde{A}_{n-1}(w)B^{-1}_{n-1}
:=
\left(
\begin{matrix}
\bar{a}_{11}&\cdots&\bar{a}_{1(i-1)}&\bar{a}_{1(i+1)}&\cdots&\bar{a}_{1n}\\[3pt]
\vdots&\vdots&\vdots&\vdots&\vdots&\vdots\\
\bar{a}_{(i-1)1}&\cdots&\bar{a}_{(i-1)(i-1)}&\cdots&\cdots&\bar{a}_{(i-1)n}\\[3pt]
\bar{a}_{(i+1)1}&\cdots&\cdots&\cdots&\cdots&\bar{a}_{(i+1)n}\\[3pt]
\vdots&\vdots&\vdots&\vdots&\vdots&\vdots\\
\bar{a}_{n1}&\cdots&\bar{a}_{n(i-1)}&\cdots&\cdots&\bar{a}_{nn}
\end{matrix}
\right)_{n-1}
\end{split}
\]
and
\[
\bar{B}_{n-1}(0)=\text{diag}
(\lambda_1(0), \cdots, \lambda_{i-1}(0), \lambda_{i+1}(0), \cdots, \lambda_n(0)).
\]
Along with  $\eqref{Eq:2.106}_i$ and \eqref{Eq:2.108}, it yields that
\begin{equation}\label{Eq:2.108*}
\partial_t \left(
 \begin{array}{c}
w_1\\[3pt]
\vdots\\
 \tilde{u}_i\\[3pt]
 \vdots\\
w_n
 \end{array}
 \right)+\bar{B}_n(w_1, \cdots, \tilde{u}_i,\cdots, w_n) \partial_x\left(
 \begin{array}{c}
w_1\\[3pt]
\vdots\\
 \tilde{u}_i\\[3pt]
 \vdots\\
w_n
 \end{array}
 \right)=0,
\end{equation}
where the $i$-th column of the square matrix $\bar{B}_n$ is $(0,  \cdots, 0, \tilde{a}_{ii}, 0, \cdots,0)^{\top}$,
and $\bar{B}_{n-1}$ is just the square  matrix by removing the $i$-th row and $i$-th column from $\bar{B}_n$.

Let
\begin{equation}\label{Eq:2.100}
\tilde{u}_i=w_i+\sum_{j\neq i}k_jw_j,
\end{equation}
where $k_j$ are some constants determined later. In this case, equation $\eqref{Eq:2.108*}_i$ can be rewritten as
\[
\begin{split}
\partial_tw_i+\tilde{a}_{ii}\partial_xw_i+\sum_{j\neq i}\Big(\bar{a}_{ij}+k_j(\tilde{a}_{ii}-\bar{a}_{jj})
-\sum_{l\neq i, j}k_l\bar{a}_{lj}\Big)\partial_xw_j=0.
\end{split}
\]
By $\tilde{a}_{ii}(0)=\lambda_i(0)$, we set the following equalities
\begin{equation}\label{Eq:2.111}
k_j(\lambda_i(0)-\bar{a}_{jj}(0))
-\sum_{l\neq i, j}k_l\bar{a}_{lj}(0)=-\bar{a}_{ij}(0) \qquad \text{for\; $j\neq i$}.
\end{equation}
Due to
\[
\begin{split}
&\det\left(
\begin{matrix}
\lambda_i(0)-\bar{a}_{11}(0)&-\bar{a}_{21}(0)&\cdots&-\bar{a}_{(i-1)1}(0)&-\bar{a}_{(i+1)1}(0)&\cdots&-\bar{a}_{n1}(0)\\[3pt]
-\bar{a}_{12}(0)&\lambda_i(0)-\bar{a}_{22}(0)&\cdots&-\bar{a}_{(i-1)2}(0)&-\bar{a}_{(i+1)2}(0)&\cdots&-\bar{a}_{n2}(0)\\[3pt]
\vdots\\
-\bar{a}_{1n}(0)&-\bar{a}_{2n}(0)&\cdots&-\bar{a}_{(i-1)n}(0)&-\bar{a}_{(i+1)n}(0)&\cdots&\lambda_i(0)-\bar{a}_{nn}(0)
\end{matrix}
\right)_{n-1}\\
=&\Big(\lambda_i(0)-\lambda_1(0)\Big)\cdots\Big(\lambda_i(0)-\lambda_{i-1}(0)\Big)\Big(\lambda_i(0)-\lambda_{i+1}(0)\Big)\cdots
\Big(\lambda_i(0)-\lambda_n(0)\Big)
\neq 0,
\end{split}
\]
then $k_j$ ($j\not =i$) can be uniquely solved from \eqref{Eq:2.111}.

From \eqref{Eq:2.103}, \eqref{Eq:2.105}, \eqref{Eq:2.108} and \eqref{Eq:2.100}, $A(w)$ with properties (1)-(4)
can be obtained. Then Lemma \ref{Y-4} is proved.
\end{proof}

Denote $x=\varphi_j(y, t)$ by the
$j$-th characteristics of  \eqref{eq:2.88} passing through the point $(y, 0)$.
Set $D_{j,t_0}=\{(x,t): \vp_j(a,t)\le x\le \vp_j(b,t), t_0\le t\le T_{\epsilon}\}$ for some large fixed $t_0>0$
(see Figure $1$).
In this case, the domains $D_{j,t_0}$ for different $j$ ($1\le j\le n$) are disjoint.
On the other hand, it follows from Chapter 4 of \cite{Horm} or \cite{John} that the blowup points at the blowup time only appear
in $D_{i,t_0}$ under the assumption \eqref{HC-2}.

For simplicity, we still denote $x=\varphi(y, t)$ as the
 $i$-th characteristics of  \eqref{eq:2.88} passing through the point $(y, t_0)$, which means
 \begin{equation}\label{HC-1}
\begin{cases}
 \partial_t \varphi(y, t)=\lambda_i(w(\varphi(y, t), t)), \\
 \varphi(y, t_0)=y.
 \end{cases}
\end{equation}
Define $a_0=\vp_i(a, t_0)$ and $b_0=\vp_i(b, t_0)$.
\vskip 0.2 true cm

\begin{figure}[ht]
\begin{center}
\begin{tikzpicture}[scale=1.1]
\draw [thick][->] (-2.5, -1)--(8.9, -1);
\draw [ultra thick](1.5,-1)to [out=100, in=-50](0.3, 2.2);
\draw [ultra thick] (3, -1) to [out=100, in=-60](2, 2.2);
\draw [red][ultra thick](0.3, 2.2)to [out=130, in=-50](-1.2, 4.5);
\draw [red][ultra thick] (2, 2.2) to [out=120, in=-50](0.6, 4.5);
\draw [ultra thick](1.5,-1)to [out=80, in=-120](2.8, 2.2);
\draw [ultra thick] (3, -1) to [out=80, in=-120](4.5, 2.2);
\draw [blue][ultra thick](2.8,2.2)to [out=60, in=-130](4.3, 4.5);
\draw [blue][ultra thick] (4.5, 2.2) to [out=60, in=-130](6, 4.5);
\draw [ultra thick](1.5,-1)to [out=60, in=200](5.5, 2.2);
\draw [ultra thick] (3, -1) to [out=60, in=210](7.5, 2.2);
\draw [green][ultra thick](5.5, 2.2)to [out=20, in=-130](7.2, 4.5);
\draw [green][ultra thick] (7.5, 2.2) to [out=30, in=-130](8.7, 4.5);
\draw [ultra thick](1.5,-1)to [out=120, in=-30](-1.5, 2.2);
\draw [ultra thick] (3, -1) to [out=120, in=-30](-0.2, 2.2);
\draw [purple][ultra thick](-1.5,2.2)to [out=150, in=-38](-2.7, 3.6);
\draw [purple][ultra thick] (-0.2, 2.2) to [out=150, in=-38](-2.3, 4.5);
\draw [thick](-2.8, 2.2)--(8.9, 2.2);
\draw [thick](-2.8, 4.5)--(8.9, 4.5);
\node at (9.5, -1) {$t=0$};
\node at (1.5, -1.3) {$(a, 0)$};
\node at (3, -1.3) {$(b, 0)$};
\node at (0.5, 3.5) {$D_{i,t_0}$};
\node at (4.2, 3.2) {$D_{j,t_0}$};
\node at (7.2, 3.2) {$D_{l,t_0}$};
\node at (-1.6, 3) {$D_{k,t_0}$};
\node at (0.6, 2.5) {$(a_0, t_0)$};
\node at (1.6, 2) {$(b_0, t_0)$};
\node at (9.5, 2.2) {$t=t_0$};
\node at (9.5, 4.5) {$t=T_{\epsilon}$};
\node [below] at (3.5, -1.8){Figure $1$. Domains $D_{i,t_0}$ and $D_{j,t_0} (j\not=i)$};
\end{tikzpicture}
\end{center}
\end{figure}

Set $v(y,t)=w(\varphi(y, t), t)$. Then it follows from \eqref{eq:2.88} and \eqref{HC-1} together with direct computation
that
\begin{equation}\label{eq:2.1}
\begin{cases}
\partial_t\varphi(y, t)=\lambda_i(v),\\[3pt]
l_i(v)\partial_t v=0,\\[5pt]
l_j(v)\Big(\partial_t v\partial_y\varphi(y, t)+(\lambda_j-\lambda_i)(v(y, t))\partial_y v\Big)=0,\qquad  j\neq i,
\end{cases}
\end{equation}
which is called the blowup system corresponding to the $i$-th eigenvalue $\lambda_i(w)$
by the terminology in \cite{A1, A2}. Note that \eqref{eq:2.1} is a completely nonlinear
evolution system of $(\varphi, v)$, which is degenerate at the points satisfying $\partial_y\varphi(y, t)=0$.

We now state a result on the extension of smooth solution $(\varphi,v)$ of  \eqref{eq:2.1} across $T_{\ve}$ when  the
 initial data are given by
\begin{equation}\label{Y-5}
 \varphi(y, t_0)=y,\quad v(y,t_0)=w(y,t_0).
\end{equation}

\begin{theorem}\label{T:2.1}
Assume that \eqref{eq:1.3} holds  and \eqref{Y-1} is genuinely nonlinear with respect to $\lambda_i(u)$.
Under assumption \eqref{Eq:2.1},
there exist a small constant $\delta_0$ and a unique smooth solution $(\varphi, v)$ to \eqref{eq:2.1}-\eqref{Y-5}
in the domain  $D=\{(y, t):  a_0\le y\le  b_0, t_0\leq t\le T_\epsilon+\delta_0\}$.
Moreover, the following estimates hold that
for $0\le |\alpha|\le 3$ and $(y,t)\in D$,
\begin{equation}\label{Eq:2.3}
|\partial_{y, t}^{\alpha}\varphi(y, t)|\leq C_{\alpha}, \quad
|\partial^{\alpha}_{y, t}v(y, t)|\leq C_{\alpha}\epsilon,
\end{equation}
where $C_{\alpha}$ stands for the generic positive constant independent of $\epsilon$.

\end{theorem}

\begin{remark}\label{RK-1}
Note that for the i-characteristics $\Gamma_a^i$ and $\Gamma_b^i$ through $(a, 0)$ and $(b, 0)$ separately,
when $\delta_0>0$ is small, the determined domains for the points at  $\Gamma_a^i$ and $\Gamma_b^i$
do not include the blowup point at $T_{\ve}$. Then by Chapter 4 of \cite{Horm} or \cite{John}, we know
that the smooth solution of \eqref{eq:1.1}
exists in $\{(y,t): y\in\Bbb R, 0\le t\le T_{\epsilon}+\delta_0\}\setminus D$.
Moreover,  when  $(y,t)\in D\cap\{(y,t): y\in\Bbb R, t_0\le t\le T_{\epsilon}\}$,
$\partial_{y,t}^{\alpha} v_j(y,t)=O(1)\ve^2$ hold for $j\not=i$ and $|\alpha|>0$, while  $v(y,t)=O(1)\ve$.
In addition, for convenience of writing,
$\delta_0=1$ is assumed in Theorem \ref{T:2.1} from now on.	
\end{remark}

\begin{remark}
If there exist two numbers $\mathcal{N}_{i_0}$ and $\mathcal{N}_{j_0}$ with $\mathcal{N}_{i_0}=\mathcal{N}_{j_0}$
($i_0\not=j_0$, $1\le i_0, j_0\le n$) in \eqref{Y-2}, and the first blowup point
appears in $D_{i_0,t_0}$, then under the condition \eqref{Eq:2.1} for $N_{i_0}(x)$,
Theorem \ref{T:2.1} holds analogously due to $D_{i_0,t_0}\cap D_{j_0,t_0}=\emptyset$.
\end{remark}

\begin{remark}
In this paper, assume $n\geq 3$ and $2\leq i\leq n-1$. In fact, for $i=1$ or $i=n$, it is much simpler
to show Theorem \ref{T:2.1} since the solution $u$ is $1$-simple on the left side of the 2-characteristics,
and is $n$-simple on the right side of the $(n-1)$-characteristics.
\end{remark}


\subsection{Reformulation of blowup system \eqref{eq:2.1}}

Introduce a quantity in the domain $D=\{(y, t):  a_0\le y\le  b_0, t_0\leq t\le T_\epsilon+1\}$ as
\[
v(y, t)=\epsilon D_{\epsilon}\omega
\]
with the matrix $D_\epsilon=\text{diag}(\epsilon, \cdots,1, \cdots, \epsilon)$ (the number $1$ is at the $(i,i)-$position).

We shall investigate the blowup system \eqref{eq:2.1} reformulated by $\omega$. It follows from Remark \ref{RK-1} that one can only expect the uniform boundedness of $D_{\epsilon}\omega$
(rather than $\omega$) and $\partial_{x,t}^{\alpha}\omega$
with $|\alpha|>0$ if $n\ge 4$, which is different from the case of $n=3$ in \cite{A2}
(where $\partial_{x,t}^{\alpha}\omega$ for all $|\alpha|\ge 0$ are uniformly bounded).

In addition, set
\[
\tilde{l}_j(\epsilon D_{\epsilon}\omega)=l_j(\epsilon D_\epsilon \omega)D_{\epsilon},
\quad \tilde{r}_j=
D^{-1}_\epsilon r_j(\epsilon D_{\epsilon}\omega), \quad j=1, \cdots, n.
\]
Note that $\tilde{l}_j(\epsilon D_{\epsilon}\omega)=\epsilon(l_{j1},..., 0, ..., l_{jn})$ for $j\neq i$ contains the small factor $\epsilon$
due to the simplification in Lemma \ref{Y-4}.

Then the blowup system \eqref{eq:2.1} can be reduced into
\begin{equation}\label{eq:2.2}
\begin{cases}
\partial_t\varphi(y, t)=\lambda_i(\epsilon D_\epsilon \omega),\\[3pt]
\tilde{l}_i(\epsilon D_{\epsilon}\omega)\partial_t\omega=0,\\[5pt]
\tilde{l}_j(\epsilon D_\epsilon \omega)\Big(K\partial_t\omega+(\lambda_j-\lambda_i)(\epsilon D_\epsilon \omega)\partial_y\omega\Big)=0,
\qquad  j\neq i,
\end{cases}
\end{equation}
where $K:=\partial_y\varphi(y, t)$.
Motivated by \cite{A2}, set
\begin{equation}\label{Eq:2.4}
h_i=\tilde{l}_i\partial_y\omega, \quad h_j=\tilde{l}_j\partial_t\omega \;\;\text{for  $j\neq i$}.
\end{equation}
This leads to
\begin{equation}\label{eq:2.4}
\partial_t \omega=\sum_{j\neq i}h_j\tilde{r}_j, \quad
\partial_y\omega=-\sum_{j\neq i}\frac{K h_j}{\lambda_j-\lambda_i}\tilde{r}_j+h_i\tilde{r}_i.
\end{equation}
Taking the first order derivative of $\eqref{eq:2.2}_1$ with respect to $y$, we have
\begin{equation}\label{YD-1}
\partial_t K=\epsilon\nabla \lambda_i D_{\epsilon}\partial_y\omega
=\epsilon L^{(1)}(\epsilon D_\epsilon\omega)\tilde{h}K+\epsilon\nabla \lambda_i
\cdot r_ih_i,
\end{equation}
where and below $L^{(k)}(\epsilon D_\epsilon \omega)$ represents the row vector depending on $\epsilon D_{\epsilon}\omega$,
and $\tilde{h}=(h_1, \cdots,$  $0, \cdots, h_n)^{\top}$ stands for the resulting vector that replaces the
component $h_i$ in $h$ with $0$.

\vspace{0.3cm}
Next, we derive the equation of $h_j$ for $j=1, \cdots, n$. Differentiating $\eqref{eq:2.2}_2$ with respect to $y$ yields
\begin{equation}\label{eq:2.5}
\epsilon\Big(\nabla\tilde{l}_i(\epsilon D_{\epsilon}\omega)D_\epsilon\partial_y\omega\Big)^{\top}\partial_t\omega
+\tilde{l}_i\partial^2_{ty}\omega=0,
\end{equation}
where
\begin{equation}\label{eq:2.6}
\tilde{l}_i\partial^2_{ty}\omega=\partial_t h_i-\epsilon (\nabla\tilde{l}_i(\epsilon D_\epsilon\omega)D_{\epsilon}\partial_t\omega)^{\top}\partial_y\omega.
\end{equation}

By \eqref{eq:2.4} and \eqref{eq:2.5}-\eqref{eq:2.6}, one arrives at
\begin{equation*}
\begin{split}
\partial_t h_i=&\epsilon K\sum_{j\neq i, k\neq i}
h_jh_k\frac{\lambda_k-\lambda_j}{(\lambda_j-\lambda_i)(\lambda_k-\lambda_i)}
(\nabla\tilde{l}_i r_j)^{\top}\tilde{r}_k+
\epsilon h_i\sum_{j\neq i}h_j\Big((\nabla\tilde{l}_i r_j)^{\top}\tilde{r}_i-
(\nabla\tilde{l}_i r_i)^{\top}\tilde{r}_j\Big)\\
=&\epsilon\tilde{h}^{\top}Q^{(1)}(\epsilon D_{\epsilon}\omega)\tilde{h}K+\epsilon L^{(2)}(\epsilon D_{\epsilon}\omega)\tilde{h}h_i,
\end{split}
\end{equation*}
where and below $Q^{(k)}(\epsilon D_{\epsilon}\omega)$ represents a symmetric $n \times n$ matrix.

\vspace{0.2cm}
On the other hand, taking the first order derivative of $\eqref{eq:2.2}_3$ with respect to $t$ yields that
\[
\begin{split}
\tilde{l}_j\Big(K\partial_t+(\lambda_j-&\lambda_i)\partial_y\Big)\partial_t\omega+\tilde{l}_j\Big(\partial_t K\partial_t \omega+\epsilon\nabla(\lambda_j-\lambda_i)D_{\epsilon}
\partial_t\omega\partial_y \omega\Big)\\
&+\epsilon(\nabla\tilde{l}_jD_{\epsilon}\partial_t\omega)^{\top}
(K\partial_t\omega+(\lambda_j-\lambda_i)\partial_y\omega)=0,
\end{split}
\]
which drives
\begin{equation*}
\begin{split}
&\Big(K\partial_t+(\lambda_j-\lambda_i)\partial_y\Big)h_j
+\epsilon K\Big(\tilde{h}^{\top}Q^{(2)}(\epsilon D_{\epsilon}\omega)\tilde{h}+L^{(3)}(\epsilon D_{\epsilon}\omega)\tilde{h} h_j\Big)\\
&\qquad\quad+\epsilon h_i\Big(r^{\top}_iQ^{(3)}(\epsilon D_{\epsilon}\omega)\tilde{h}+
\nabla \lambda_i(\epsilon D_{\epsilon}\omega)r_i(\epsilon D_{\epsilon}\omega)h_j\Big)=0.
\end{split}
\end{equation*}
Therefore, the blowup system \eqref{eq:2.2} in domain $D$ can be reformulated as
\begin{equation}\label{Y-3}
\begin{cases}
\partial_t K=\epsilon L^{(1)}(\epsilon D_\epsilon\omega)\tilde{h}K+\epsilon\nabla \lambda_i
\cdot r_i(\epsilon D_\epsilon \omega) h_i,\\[5pt]
\partial_t h_i
=\epsilon\tilde{h}^{\top}Q^{(1)}(\epsilon D_{\epsilon}\omega)\tilde{h}K+\epsilon L^{(2)}(\epsilon D_{\epsilon}\omega)\tilde{h}h_i,\\[5pt]
\Big(K\partial_t+(\lambda_j-\lambda_i)\partial_y\Big)h_j
+\epsilon K\Big(\tilde{h}^{\top}Q^{(2)}(\epsilon D_{\epsilon}\omega)\tilde{h}+L^{(3)}(\epsilon D_{\epsilon}\omega)\tilde{h} h_j\Big)\\[3pt]
\qquad\qquad\quad+\epsilon h_i\Big(r^{\top}_iQ^{(3)}(\epsilon D_{\epsilon}\omega)\tilde{h}
+\nabla \lambda_i(\epsilon D_{\epsilon}\omega)r_i(\epsilon D_{\epsilon}\omega)h_j\Big)=0,\\[5pt]
\tilde{l}_i(\epsilon D_{\epsilon}\omega)\partial_t\omega=0,\\[5pt]
\tilde{l}_j(\epsilon D_{\epsilon}\omega)\Big(K\partial_t\omega+(\lambda_j-\lambda_i)(\epsilon D_{\epsilon}\omega)
\partial_y\omega\Big)=0, \qquad j\neq i.
\end{cases}
\end{equation}
From \eqref{Y-3}, it is known that $K$ and $h_i$ can be solved by the direct integration with respect to $t$ through their own initial data, while
$h_j$ ($j\not=i$) are determined by their initial data and  suitable boundary values on $y=a_0$ or $y=b_0$ (the signs of $K$ and $\lambda_j-\lambda_i$
on the boundaries play a key role). Additionally, we have to overcome the difficulties arisen  from
the boundedness of $D_\epsilon\omega$ rather than $\omega$ since the uniform bounds or smallness orders
of $(K, h, D_\epsilon\omega)$ and their derivatives require to be derived. In this process, such basic smallness results
of the tangent derivatives  $\partial_th_i|_{y=a_0}=O(1)\epsilon, \partial_th_i|_{y=b_0}=O(1)\epsilon$ and $\partial_t K=O(1)\epsilon$
are crucial.

\subsection{Boundary values  of blowup system \eqref{eq:2.1} on  the $i$-th characteristics}

In this subsection, we mainly study the estimates of $h_k$ and $\tilde{l}_k w$  on
the boundaries $y=a_0$ or $y=b_0$ in domain $D$ for the integer $k\neq i$.

By $\text{supp }u_0(x)\subset [a, b]$,  for the solution $u$ of problem \eqref{eq:1.1} one has
\[
\text{supp }u(x, t)\subset\Big\{(x, t): a+\lambda_1(0)t \leq x\leq b+\lambda_n(0)t, t\ge 0\Big\}.
\]
The $i$-th characteristics $\Gamma^i: x=\vp(y, t)$ has been defined in \eqref{HC-1}.
Set
\[
a(t)=\vp(a_0, t), \quad b(t)=\vp(b_0, t),
\]
and the domain $R_{i,t_0}$ (see Figure $2$) can be described as
\[
R_{i,t_0}:=\{(x, t): a(t)\leq x\leq b(t), t_0\leq t\leq t_1\},
\]
where and below $t_1:=T_{\epsilon}+1$ is defined.
\begin{figure}[ht]
\begin{center}
\begin{tikzpicture}[scale=1.3]
\draw [thick][->] (0.5, -1)--(7, -1);
\draw [red][ultra thick](1.5,-1)to [out=90, in=-90](2.5, 1.5);
\draw [red][ultra thick] (5, -1) to [out=90, in=-90](6, 1.5);
\node at (7.2,-1.15) {$x$};
\node at (-0.1, -1) {$t=t_0$};
\node at (1.5, -1.3) {$(a_0, t_0)$};
\node at (5, -1.3) {$(b_0, t_0)$};
\node at (3.5, 0.3) {$R_{i,t_0}$};
\node at (3, 1.8) {$\Gamma_a^i : x=a(t)$};
\node at (7, 1.8) {$\Gamma_b^i: x=b(t)$};
\node [below] at (3.6, -1.5){Figure $2$. Domain $R_{i,t_0}$};
\end{tikzpicture}
\end{center}
\end{figure}

To solve \eqref{Y-3} and further derive the behavior of solution in $R_{i,t_0}$,
we need to study the appropriate boundary
conditions of $h_k$ and $\tilde{l}_k\omega_k$ on $x=a(t)$ and $x=b(t)$ for $k\neq i$.

Under the transformation $w=\epsilon D_\epsilon \omega$, $\eqref{eq:2.88}_1$ can be reduced into $\partial_x \omega=-D^{-1}_\epsilon A^{-1} D_\epsilon\partial_t \omega$. In addition, without loss of generality,  $\lambda_k\not=0$ for all $1\le k\le n$ are assumed
(otherwise, one can achieve this by a simple spatial translation). By direct calculations, for $j\neq i$, it holds that
\begin{equation}\label{eq:2.88*}
\begin{split}
L_j h_j:=&\partial_t h_j+\lambda_j\partial_x h_j\\
=&\epsilon(\nabla\tilde{l}_j D_\epsilon(\partial_t\omega+\lambda_j\partial_x\omega))^{\top}\partial_t\omega+\tilde{l}_j\partial^2_t\omega+
\lambda_j\tilde{l}_j\partial^2_{xt}\omega\\
=&\epsilon\sum_{k\neq i, l\neq i}h_kh_l(\lambda_j\lambda^{-1}_l-\lambda_j\lambda^{-1}_k)(\nabla\tilde{l}_jr_k)^{\top}\tilde{r}_l
+\epsilon\nabla\lambda_j\sum_{k\neq i, l\neq i}h_kh_l\lambda^{-1}_l\delta_{jl}r_k\\
=&\epsilon(Q_0h^2_j+Q_1h_j+Q_2)=\epsilon\sum_{k\neq i, l\neq i}\gamma_{jkl}h_kh_l,\\[5pt]
L_j p_j:=&\partial_t(\tilde{l}_j\omega)
+\lambda_j\partial_x(\tilde{l}_j\omega)\\
=&\epsilon\sum_{k\neq i, 1\leq l\leq n}(1-\lambda_j\lambda^{-1}_k)(\nabla\tilde{l}_j r_k)^{\top}h_kp_l\tilde{r}_l
=\epsilon\sum_{k\neq i, 1\leq l\leq n}\zeta_{jkl}h_k p_l,
\end{split}
\end{equation}
where $p_j=\tilde{l}_j\omega$, $Q_0:=\gamma_{jjj}=\frac{1}{\lambda_j}\nabla\lambda_j\cdot r_j$,\
$Q_1:=\displaystyle\sum_{k\neq i, j}\gamma_{jjk}$,\ $\displaystyle Q_2:=\sum_{k\neq i, j \atop l\neq i, j}\gamma_{jkl}h_kh_l$, and
\[
\begin{split}
\gamma_{jkl}=\lambda_j(\lambda^{-1}_l-\lambda^{-1}_k)(\nabla\tilde{l}_jr_k)^{\top}\tilde{r}_l+\delta_{jl}\lambda^{-1}_l\nabla\lambda_j r_k,
\;\;  \gamma_{jkk}|_{k\neq j}=0,\;\;
\zeta_{jkl}=(1-\lambda_j\lambda^{-1}_k)(\nabla\tilde{l}_j r_k)^{\top}\tilde{r}_l.
\end{split}
\]
Note that
\[
R_{i,t_0}\cap R_{j,t_0}=\emptyset \quad \text{for $i\neq j$}.
\]

We claim that there exist two positive constants $A_0$ and $A_1$ independent of $\epsilon$ such that
for small $\epsilon>0$, $(x, t)\notin R_{j,t_0}$ and $t_0\leq t\le t_1$, one has
\begin{equation}\label{Eq:2.89}
|h_j(x, t)|<A_0\epsilon, \quad |p_j(x, t)|<A_1.
\end{equation}

To prove this claim, first of all, we will show
\begin{equation}\label{eq:2.97*}
\int_{\Gamma^k\cap\{t_0\le t\le t_1\}}|h_j|\text{d}s\leq C\epsilon, \quad j\neq k,
\end{equation}
where and below $C>0$ is a generic constant independent of $\epsilon$, and $\Gamma^k$
is the $k-$characteristics.

\vspace{0.2cm}
Inspired by Chapter 4 of \cite{Horm}, let
\[
\begin{split}
L(t)&=\sum_{j\neq i}\int^{+\infty}_{-\infty} |h_j(x, t)|\text{d}x, \quad Q(t)=\epsilon\sum_{j, k\neq i, j>k}
\iint_{x<z} |h_j(x, t)h_k(z, t)|\text{d}x\text{d}z,\\
R(t)&=\sum_{j>k}\int^{+\infty}_{-\infty}|h_j(x, t)h_k(x, t)|\text{d}x.
\end{split}
\]
In general, $L(t)$ is not decreasing with respect to $t$. However, a suitable linear combination of
$L(t)$ and $Q(t)$ can be shown to decrease under the help of $R(t)$.
We now have

\begin{lemma}\label{YH-1}
There exist some positive constants $C_0$ and $C_1$  such that when $L(t_0)\leq \frac{C_0}{3C_1}$, one has
\begin{equation*}
L(t)\leq 2L(t_0),\;\; \int^{t_1}_{t_0} R(t)\text{d}t\leq \frac{3\epsilon}{2C_0}L^2(t_0).
\end{equation*}
\end{lemma}
\begin{proof}
Note that
\[
\begin{split}
&\text{d}(h_j(\text{d}x-\lambda_j(w)\text{d}t))
=\Big(\partial_t h_j+\lambda_j\partial_x h_j
-\epsilon \nabla\lambda_j\sum_{k\neq i}h_jh_k\lambda^{-1}_k r_k\Big)\text{d}t\wedge \text{d}x\\
=&\epsilon\sum_{k\neq i,l\neq i}\Gamma_{jkl}h_kh_l\text{d}t\wedge \text{d}x,
\end{split}
\]
where $\Gamma_{jkl}:=\gamma_{jkl}-\frac{1}{2}\nabla\lambda_j(\lambda^{-1}_k\delta_{jl}r_k+\lambda^{-1}_l\delta_{jk}r_l)$,
$\Gamma_{jkk}|_{j\neq k}=0$, and $\delta_{jl}$, $\delta_{jk}$ are Kronecker symbols.

Therefore,
\begin{equation}\label{eq:2.91}
\begin{split}
\text{d}(|h_j|(\text{d}x-\lambda_j(w)\text{d}t))=&\Big(\partial_t|h_j|+\partial_x(|h_j|\lambda_j(w))\Big)\text{d}t\wedge\text{d}x
=\epsilon\text{sgn } h_j\sum_{k, l\neq i}
\Gamma_{jkl}h_kh_l\text{d}t\wedge \text{d}x.
\end{split}
\end{equation}
Applying Stokes' formula to \eqref{eq:2.91} in the interval $(-\infty, x]\times [t_0, t_1]$ yields
\[
\begin{split}
&\int^x_{-\infty}|h_j(x, t_1)|\text{d}x-\int^x_{-\infty}|h_j(x, t_0)|\text{d}x+
\int^{t_1}_{t_0}\lambda_j(w(x, t))|h_j(x, t)|\text{d}t\\
=&\int^x_{-\infty}\int^{t_1}_{t_0}\Big(\partial_t |h_j(x, t)|
+\partial_x(\lambda_j(w(x, t))|h_j(x, t)|)\Big)\text{d}t\text{d}x\\
=&\epsilon\int^x_{-\infty}\int^{t_1}_{t_0}\text{sgn}h_j\sum_{k, l\neq i}\Gamma_{jkl}h_kh_l\text{d}t\text{d}x.
\end{split}
\]
Hence,
\begin{equation}\label{eq:2.92}
\int^x_{-\infty}\partial_t|h_j(x, t)|\text{d}x=-\lambda_j(w(x, t))|h_j(x, t)|+
\epsilon\int^x_{-\infty}\text{sgn}h_j\sum_{k, j\neq i}\Gamma_{jkl}h_kh_l(x, t)\text{d}x.
\end{equation}
Similarly,
\begin{equation}\label{eq:2.93}
\int^{+\infty}_x\partial_t|h_j(x,t)|\text{d}x=\lambda_j(w(x, t))|h_j(x, t)|+
\epsilon\int^{+\infty}_x\text{sgn}h_j\sum_{k, j\neq i}\Gamma_{jkl}h_kh_l(x, t)\text{d}x.
\end{equation}
Based on \eqref{eq:2.92} and \eqref{eq:2.93}, there exist two positive constant $C_0$ and $C_1$ such that
\begin{equation}\label{eq:2.94}
\begin{split}
Q'(t)=&\epsilon\sum_{j>k}\Big(\int^{\infty}_{-\infty}\text{d}z\int^z_{-\infty}\partial_t|h_j(x, t)||h_k(z, t)|\text{d}x+\int^{\infty}_{-\infty}
\text{d}x\int^{\infty}_x\partial_t|h_k(z, t)||h_j(x, t)|\text{d}z\Big)\\
=&\epsilon\sum_{j>k}\int^{\infty}_{-\infty}(\lambda_k(w)-\lambda_j(w))|h_j(x, t)||h_k(x, t)|\text{d}x\\
&+\epsilon^2\sum_{j>k}\int^{\infty}_{-\infty}\int^z_{-\infty}\text{sgn}h_j\sum_{\mu, \nu\neq i}\Gamma_{j\mu\nu}h_{\mu}h_{\nu}(x, t)|h_k(z, t)|\text{d}x\text{d}z\\
&+\epsilon^2\sum_{j>k}\int^{\infty}_{-\infty}\int^{\infty}_{x}\text{sgn}h_k\sum_{\mu, \nu\neq i}\Gamma_{k\mu\nu}h_{\mu}h_{\nu}(z, t)|h_j(x, t)|\text{d}z\text{d}x\\
\leq& -C_0 \epsilon R(t)+C_1 \epsilon^2 R(t)L(t),
\end{split}
\end{equation}
where $\lambda_j(u)-\lambda_k(u)\geq C_0$ for $j>k$, and
\begin{equation}\label{eq:2.94*}
L'(t)=\epsilon\sum_{j\neq i, \; k\neq i\atop l\neq i}\int\text{sgn }h_j\Gamma_{jkl}h_k h_l\text{d}x\leq C_1\epsilon R(t).
\end{equation}

Assume that $L(t)\leq \frac{2C_0}{3C_1}$ holds for $t_0\leq t\leq t_1$. By continuous induction,
we need to show that
for $t<t_1$, there exists a positive constant $M_0$ with $M_0<\frac{2C_0}{3C_1}$ such that
$L(t)<M_0$.

Along with \eqref{eq:2.94} and \eqref{eq:2.94*}, one has
\[
\begin{split}
\Big(3C_1Q(t)+C_0L(t)\Big)'=&3C_1Q'(t)+C_0L'(t)\leq 3\epsilon C_1R(t)\Big(-C_0+C_1\epsilon L(t)\Big)+\epsilon C_1C_0R(t)\\
=&\epsilon (3C_1\epsilon L(t)-2C_0)C_1R(t)\leq 0.
\end{split}
\]
This implies
\[
3C_1Q(t)+C_0L(t)\leq 3C_1Q(t_0)+C_0L(t_0)\leq \frac{3C_1}{2}\epsilon L^2(t_0)+C_0L(t_0)\leq \frac{3C_0}{2}L(t_0)
\]
and
\[
L(t)\leq \frac{3}{2}L(t_0)<2L(t_0)<\frac{2C_0}{3C_1}.
\]
Meanwhile, it follows from a direct computation that
\[
Q'(t)\leq \epsilon R(t)(C_1\epsilon L(t)-C_0)\leq \epsilon R(t)(2C_1\epsilon L(t_0)-C_0)\leq -\frac{C_0}{3}R(t).
\]
Then
\[
Q(t)-Q(t_0)\leq -\frac{C_0}{3}\int^{t_1}_{t_0} R(s)\text{d}s
\]
and
\[
\int^{t_1}_{t_0} R(s)\text{d}s\leq \frac{3}{C_0}Q(t_0)\leq \frac{3}{2C_0}\epsilon L^2(t_0).
\]

Thus, we complete the proof of Lemma \ref{YH-1}.
\end{proof}

Based on Lemma \ref{YH-1}, we will show \eqref{eq:2.97*} and estimate the integral of  $h_j(j\neq i)$ along the different characteristics
families, which is crucial for evaluating the boundedness of $h_j$ on the $i$-th characteristics.

\begin{lemma}\label{YH-2}
Under the assumption in Lemma \ref{YH-1}, it holds that
\begin{equation*}
\int_{\Gamma^k\cap\{t_0\le t\le t_1\}}|h_j|\text{d}s\leq \frac{4\epsilon}{3C_1}, \quad k\neq j,
\end{equation*}
where the positive constant $C_1$ has been given in Lemma \ref{YH-1}.
\end{lemma}

\begin{proof}
Denote by $\mathcal{D}$ the domain bounded by $\Gamma^k$, $\Gamma^j$,
the straight lines $t=t_0$ and $t=t_1$ (see Figure $3$). Applying Stokes' formula to \eqref{eq:2.91} yields
\begin{equation}\label{eq:2.96}
\int_{\partial\mathcal{D}}|h_j|( \text{d}x-\lambda_j(w)\text{d}t)=\epsilon\iint_{\mathcal{D}}\text{sgn }h_j
\sum_{k, l\neq i}\Gamma_{jkl}h_kh_l\text{d}t\text{d}x.
\end{equation}

\begin{figure}[ht]
\begin{center}
\begin{tikzpicture}[scale=1.4]
\draw [thick][->] (0, 0)--(3, 0);
\draw [thick] [->] (3, 0) to [out=90, in=-90](5, 2);
\draw[thick][->] (5, 2)--(1, 2);
\draw[thick] [->] (1, 2) to [out=200, in=60](0, 0);
\node at (0.5, 2.2) {$t=t_1$};
\node at (1, 0.3) {$t=t_0$};
\node at (2, 1.2) {$\mathcal{D}$};
\node at (0.1, 1.2) {$\Gamma^j$};
\node at (4.8, 1) {$\Gamma^k$};
\node [below] at (3, -0.3){Figure $3$. Domain $\mathcal{D}$};
\end{tikzpicture}
\end{center}
\end{figure}

Along $\Gamma^j: |h_j|(\text{d}x-\lambda_j(w)\text{d}t)=0$, and along $\Gamma^k$:
\[
|h_j||\text{d}x-\lambda_j(w)\text{d}t|=|h_j||\lambda_k-\lambda_j||\text{d}t|\geq C_0|h_j||\text{d}t|,
\]
where $C_0>0$ is given  in Lemma \ref{YH-1}.

Therefore, it follows from \eqref{eq:2.96} that
\[
\begin{split}
C_0\int_{\Gamma^k\cap\{t_0\le t\le t_1\}}|h_j|\text{d}t\leq& \epsilon\Big(L(t_0)+L(t_1)+C_1\int^{t_1}_{t_0}R(s)\text{d}s\Big)\\
<&3\epsilon L(t_0)+\frac{3C_1}{2C_0}\epsilon^2 L^2(t_0)<4\epsilon L(t_0)<\f{4C_0\epsilon}{3C_1}.
\end{split}
\]
Thus, the proof of Lemma \ref{YH-2} is finished.
\end{proof}

In the sequel, we prove the claim \eqref{Eq:2.89} by the continuous induction argument.
For any point $P_0(a(t), t)$ lying on $\Gamma^i_{a_0}$, $P_0\notin R_{j,t_0}$ for
$j\not=i$ (see Figure $4$),
by integrating
\eqref{eq:2.88*} along $\Gamma^j (j>i)$, one has
\begin{equation}\label{eq:Y-1}
\begin{split}
h_j(x(t), t)|_{P_0}=&\epsilon\sum_{k, l\neq i}\int^{t}_0\gamma_{jkl} h_k(x(s), s)h_l(x(s), s)\Big|_{\Gamma^j}\text{d}s,\\
p_j(x(t), t)|_{P_0}=&\epsilon\sum_{k\neq i, 1\leq l\leq n}\int^{t}_0\zeta_{jkl} h_k(x(s), s)p_l(x(s), s)\Big|_{\Gamma^j}\text{d}s.
\end{split}
\end{equation}

\begin{figure}[ht]
\begin{center}
\begin{tikzpicture}[scale=1.4]
\draw [thick][->] (-0.8, -1)--(7, -1);
\draw[red] [ultra thick](1.5,-1)to [out=90, in=-90](2.5, 1.5);
\draw[red] [ultra thick] (5, -1) to [out=90, in=-90](6, 1.5);
\draw[blue] [thin] (1.5, -1) to [out=50, in=-100](4.8, 1.5);
\draw[blue] [thin] (5, -1) to [out=50, in=-100](7.5, 1.5);
\draw [dashed][thick] (-0.5, -1) to [out=50, in=-130](2, 0.3);
\draw [thick] (2, 0.3) to [out=35, in=-100](3.8, 1.5);
\node at (7.2,-1.15) {$x$};
\node at (-1.3, -1) {$t=t_0$};
\node at (1.5, -1.3) {$(a_0, t_0)$};
\node at (5, -1.3) {$(b_0, t_0)$};
\node at (4.5, -0.1) {$R_{i,t_0}$};
\node at (1.7, 0.3) {$P_0$};
\node at (2.5, 1.7) {$\Gamma^i_{a_0}$};
\node at (3.3, 1.3) {$\Gamma^j$};
\node at (6, 1.7) {$\Gamma^i_{b_0}$};
\node [below] at (3.6, -1.6){Figure $4$. The picture on the related characteristics};
\end{tikzpicture}
\end{center}
\end{figure}

Assume that the estimates in \eqref{Eq:2.89} are valid for $t_0\leq t<t_1$.
We next show that \eqref{Eq:2.89} remains true for $t=t_1$. The proof procedure will be
divided into several cases so that the terms on the right hand of
\eqref{eq:Y-1} for $h_j|_{P_0}$ and $p_j|_{P_0}$ can be treated respectively.

\begin{itemize}
\item[Case $1$.] When $k=l=j$, in terms of the expression of $\gamma_{jjj}$ and $\zeta_{jjj}=0$, one has
\[
|\epsilon\int^{t_1}_{t_0}\gamma_{jjj}h^2_j\text{d}s|\leq C\epsilon\int_{\Gamma^j\cap\{t_0<t<t_1\}\cap R^c_{j,t_0}} h^2_j\text{d}s+C\epsilon\int_{\Gamma^j\cap\{t_0<t<t_1\}\cap R_{j,t_0}} h^2_j\text{d}s\leq CA^2_0\epsilon^2.
\]

\item[Case $2$.] When $k=l\neq j$,\; due to $\gamma_{jkl}=0$, we arrive at
\[
\begin{split}
&\epsilon|\int^{t_1}_{t_0}\zeta_{jkl}h_k p_l\text{d}s|
\leq C\epsilon\int_{\Gamma^j\cap\{t_0<t<t_1\}\cap R^c_{k,t_0}}|h_kp_k|\text{d}s+
C\epsilon\int_{\Gamma^j\cap\{t_0<t<t_1\}\cap R_{k,t_0}}|h_kp_k|\text{d}s\\
\leq& CA_0A_1\epsilon^2(t_1-t_0)+C\epsilon\int_{\Gamma^j\cap\{t_0<t<t_1\}\cap R_{k,t_0}}|h_kp_k|\text{d}s\\
\leq& CA_0A_1\epsilon^2(t_1-t_0)+C\epsilon\int_{\Gamma^j\cap\{t_0<t<t_1\}}|h_k|\text{d}s\\
\leq& CA_0A_1\epsilon^2(t_1-t_0)+C\epsilon^2,
\end{split}
\]
where $\Gamma^j\cap\{t_0<t<t_1\}\subset R^c_{j,t_0}$, $R_{k,t_0}\subset R^c_{j,t_0}$,
$|p_k(x,t)|\leq C$ for $k\neq i$ and $(x, t)\in R^c_{i, t_0}$ (see the proof procedure of Lemma 4.3.2 in \cite{Horm1}).
When $\epsilon>0$ is small, one can let
\[
CA_0A_1\epsilon^2(t_1-t_0)<\frac{A_1}{2}.
\]

\item[Case $3$.] When $k\neq l$, it holds that
\[
\begin{split}
&\epsilon|\int^{t_1}_{t_0}\gamma_{jkl}h_kh_l\text{d}s|
\leq C\epsilon\int_{\Gamma^j\cap\{t_0<t<t_1\}\cap (R_{k,t_0}\cup R_{l,t_0})^c}|h_kh_l|\text{d}s+
C\epsilon\int_{\Gamma^j\cap\{t_0<t<t_1\}\cap (R_{k,t_0}\cup R_{l,t_0})}|h_kh_l|\text{d}s\\
\leq& CA^2_0\epsilon^3(t_1-t_0)+C\epsilon\int_{\Gamma^j\cap\{t_0<t<t_1\}\cap R_{k,t_0}}|h_kh_l|\text{d}s
+C\epsilon\int_{\Gamma^j\cap\{t_0<t<t_1\}\cap R_{l,t_0}}|h_kh_l|\text{d}s,
\end{split}
\]
where the assumption of $|h_j(x, t)|\leq A_0\epsilon$ for $(x, t)\notin R_{j,t_0}$ is assumed.
\begin{itemize}
\item[Case $3.1$.] $k=j$, $l\neq j$: $(x, t)\notin R_{j,t_0}$, $\zeta_{jkl}=0$, then
\[
\begin{split}
&\epsilon\int_{\Gamma^j\cap\{t_0<t<t_1\}\cap R_{k,t_0}}|h_kh_l|\text{d}s
+\epsilon\int_{\Gamma^j\cap\{t_0<t<t_1\}\cap R_{l,t_0}}|h_kh_l|\text{d}s
=\epsilon\int_{\Gamma^j\cap\{t_0<t<t_1\}\cap R_{l,t_0}}|h_jh_l|\text{d}s\\
&\quad \leq A_0\epsilon^2\int_{\Gamma^j\cap\{t_0<t<t_1\}}|h_l|\text{d}s\leq CA_0\epsilon^3,
\end{split}
\]
where $|h_j(x, t)|\leq A_0\epsilon$ for $(x, t)\notin R_{j, t_0}$, and
$\int_{\Gamma^j\cap\{t_0<t<t_1\}}|h_l|\text{d}s\leq C\epsilon$ for $l\neq j$.

\item[Case $3.2$.] $k\neq j$, $l=j$: $(x, t)\notin R_{j,t_0}$, $\Gamma^j\cap R_{l,t_0}\subset R^c_{j,t_0}\cap R_{j,t_0}=\emptyset$,
then
\[
\begin{split}
&\epsilon\int_{\Gamma^j\cap\{t_0<t<t_1\}\cap R_{l,t_0}}|h_kh_l|\text{d}s=0,\\
&\epsilon\int_{\Gamma^j\cap\{t_0<t<t_1\}\cap R_{k,t_0}}|h_kh_l|\text{d}s
\leq CA_0\epsilon^2\int_{\Gamma^j\cap\{t_0<t<t_1\}\cap R_{k,t_0}}|h_k|\text{d}s\leq
CA_0\epsilon^2\int_{\Gamma^j\cap\{t_0<t<t_1\}}|h_k|\text{d}s, \\
&\epsilon\int_{\Gamma^j\cap\{t_0<t<t_1\}\cap R_{k,t_0}}|h_kh_l|\text{d}s
+\epsilon\int_{\Gamma^j\cap\{t_0<t<t_1\}\cap R_{l,t_0}}|h_kh_l|\text{d}s\leq CA_0\epsilon^3.
\end{split}
\]
\end{itemize}
Meanwhile,
\[
\begin{split}
&\epsilon\int_{\Gamma^j\cap\{t_0<t<t_1\}}|\zeta_{jkl}h_kp_l|\text{d}s
=\epsilon\int_{\Gamma^j\cap\{t_0<t<t_1\}\cap (R_{k,t_0}\cup R_{j,t_0})^c}|\zeta_{jkl}h_kp_j|\text{d}s\\
&\qquad+\epsilon\int_{\Gamma^j\cap\{t_0<t<t_1\}\cap R_{k,t_0}}|\zeta_{jkl}h_kp_j|\text{d}s
+\epsilon\int_{\Gamma^j\cap\{t_0<t<t_1\}\cap R_{l,t_0}}|\zeta_{jkl}h_kp_j|
\text{d}s\\
\leq &CA_0A_1\epsilon^2(t_1-t_0)+CA_1\epsilon\int_{\Gamma^j\cap\{t_0<t<t_1\}\cap R^c_{j,t_0}}|h_k|\text{d}s\\
\leq &CA_0A_1\epsilon^2(t_1-t_0)+CA_1\epsilon\int_{\Gamma^j\cap\{t_0<t<t_1\}}|h_k|\text{d}s\leq C\epsilon,
\end{split}
\]
where $R_{k,t_0}\cap R_{j,t_0}=\emptyset$, $\Gamma^j\cap R_{j,t_0}=\emptyset$, $R_{k,t_0}\subset R^c_{j,t_0}$ for $t_0<t<t_1$,
and
\[
\int_{\Gamma^j\cap\{t_0<t<t_1\}\cap R_{j,t_0}}\zeta_{jkl}h_k p_j\text{d}s=0.
\]
\begin{itemize}
\item[Case $3.3$.] $k\neq j$, $l\neq j$, $k\neq l$: $(x, t)\notin R_{j,t_0}$, $\Gamma^j\cap R_{k,t_0}\subset \Gamma^j\cap R_{l,t_0}^c$, $\Gamma^j\cap R_{l,t_0}\subset \Gamma^j\cap R^c_{k,t_0}$, then
\[
\begin{split}
&\epsilon\int_{\Gamma^j\cap\{t_0<t<t_1\}\cap R_{l,t_0}}|h_kh_l|\text{d}s\leq
CA_0\epsilon^2\int_{\Gamma^j\cap\{t_0<t<t_1\}}|h_l|\text{d}s, \quad l\neq j,\\
&\epsilon\int_{\Gamma^j\cap\{t_0<t<t_1\}\cap R_{k,t_0}}|h_kh_l|\text{d}s\leq
CA_0\epsilon^2\int_{\Gamma^j\cap\{t_0<t<t_1\}}|h_k|\text{d}s, \quad k\neq j, \\
&\epsilon\int_{\Gamma^j\cap\{t_0<t<t_1\}\cap R_{k,t_0}}|h_kh_l|\text{d}s
+\epsilon\int_{\Gamma^j\cap\{t_0<t<t_1\}\cap R_{l,t_0}}|h_kh_l|\text{d}s\leq CA_0\epsilon^3.
\end{split}
\]
\end{itemize}
Meanwhile,
\[
\begin{split}
&\epsilon\int_{\Gamma^j\cap\{t_0<t<t_1\}}|\zeta_{jkl}h_kp_l|\text{d}s
=\epsilon\int_{\Gamma^j\cap\{t_0<t<t_1\}\cap (R_{k,t_0}\cup R_{l,t_0})^c}|\zeta_{jkl}h_kp_l|\text{d}s\\
&\qquad +
\epsilon\int_{\Gamma^j\cap\{t_0<t<t_1\}\cap R_{k,t_0}}|\zeta_{jkl}h_kp_l|\text{d}s
+\epsilon\int_{\Gamma^j\cap\{t_0<t<t_1\}\cap R_{l,t_0}}|\zeta_{jkl}h_kp_l|
\text{d}s\\
\leq &CA_0A_1\epsilon^2(t_1-t_0)+CA_1\epsilon\int_{\Gamma^j\cap\{t_0<t<t_1\}\cap R_{k,t_0}}|h_k|\text{d}s+
C\epsilon\int_{\Gamma^j\cap\{t_0<t<t_1\}\cap R_{l,t_0}}|h_k|\text{d}s\\
\leq &CA_0A_1\epsilon^2(t_1-t_0)+C\epsilon\int_{\Gamma^j\cap\{t_0<t<t_1\}}|h_k|\text{d}s\leq C\epsilon.
\end{split}
\]
In particular, when $k\neq l$, $k\neq j$ and $l=i$, we have $|\zeta_{jki}|\leq C\epsilon$ and
\[
\begin{split}
&\epsilon\int_{\Gamma^j\cap\{t_0<t<t_1\}}|\zeta_{jkl}h_kp_i|\text{d}s
=\epsilon\int_{\Gamma^j\cap\{t_0<t<t_1\}\cap (R_{k,t_0}\cup R_{i,t_0})^c}|\zeta_{jki}h_kp_i|\text{d}s\\
&\qquad +\epsilon\int_{\Gamma^j\cap\{t_0<t<t_1\}\cap R_{k,t_0}}|\zeta_{jki}h_kp_i|\text{d}s
\quad+\epsilon\int_{\Gamma^j\cap\{t_0<t<t_1\}\cap R_{i,t_0}}|\zeta_{jki}h_kp_i|
\text{d}s\\
\leq &CA_0A_1\epsilon^3(t_1-t_0)+CA_1\epsilon^2\int_{\Gamma^j\cap\{t_0<t<t_1\}\cap R_{k,t_0}}|h_k|\text{d}s\\
\leq &CA_0A_1\epsilon^3(t_1-t_0)+C\epsilon^2\int_{\Gamma^j\cap\{t_0<t<t_1\}}|h_k|\text{d}s\leq C\epsilon^2.
\end{split}
\]
\end{itemize}
Combining with all the estimates in Cases $1$-$3.3$, we can confirm the claim \eqref{Eq:2.89}.

On the other hand, under the transformation $w(u)=\epsilon D_{\epsilon}\omega$, the $i$-th equation of \eqref{eq:2.88}
can be written as
\begin{equation}\label{Eq:2.40}
l_{ii}(\epsilon D_{\epsilon}\omega)(\partial_t\omega_i+\lambda_i\partial_x\omega_i)+\epsilon\sum_{j\neq i}l_{ij}(\epsilon D_{\epsilon}\omega)(\partial_t\omega_j+\lambda_i\partial_x\omega_j)=0,
\end{equation}
where $l_{ii}(\epsilon D_{\epsilon}\omega)\neq 0$. Then the equation \eqref{Eq:2.40} is reduced into
\[
\frac{\text{d}\omega_i}{\text{d}_it}:=\partial_{t}\omega_i+\lambda_i\partial_x\omega_i=\epsilon\sum_{j\neq i}p_{ij}(\epsilon D_{\epsilon}\omega)\frac{\text{d}\omega_j}{\text{d}_i t}.
\]
Integrating this along $x=a(t)$ in the interval $[t_0, t]$ ($t\le t_1$) yields
\[
\begin{split}
\omega_i(a(t), t)=&\omega_i(a_0, t_0)+\epsilon\sum_{j\neq i}p_{ij}(\epsilon D_{\epsilon}\omega)\omega_j(a(t), t)\\
&-\epsilon^3\sum_{j, k\neq i}\int^{t}_{t_0}\partial_{\omega_k}p_{ij}(\epsilon D_{\epsilon}\omega)(p_{ik}(\epsilon D_{\epsilon}\omega)+1)\Big(\partial_t\omega_k+\lambda_i\partial_x\omega_k\Big)\omega_j(a(s), s)\text{d}s.
\end{split}
\]
Due to $p_{ij}(0)\Big|_{j\neq i}=0$, then for $t_0\le t\le t_1$,
$\omega_i(a(t), t)=O(1)$.

Furthermore, it follows from $\eqref{Y-3}_3$ and $\eqref{Y-3}_2$  that on $x=a(t)$ with $t_0\le t\le t_1$ and for $j=i+1, \cdots, n$,
we have
\[
\begin{split}
&\partial_t h_j=O(1)\epsilon, \;\; \partial^2_{t}h_j=O(1)\epsilon, \;\; \partial^3_{t}h_j=O(1)\epsilon, \;\;\tilde{l}_j\partial^2_{t}\omega=O(1)\epsilon, \;\; \tilde{l}_j\partial^3_{t}\omega=O(1)\epsilon.
\end{split}
\]

Similarly, one can show that on $x=b(t)$ with $t_0\le t\le t_1$, the following estimates hold
for  $ j=1, \cdots, i-1$,
\[
\epsilon\omega_j(b(t),t)=O(1),\;\; h_j(b(t),t)=O(1)(\epsilon), \;\;  \omega_i(b(t), t)=O(1),
\;\;(\partial_{t}+\lambda_i\partial_x)\omega_i(b(t), t)=O(1)(\epsilon)
\]
and
\[
\begin{split}
&\partial_t h_j=O(1)\epsilon, \;\; \partial^2_{t}h_j=O(1)\epsilon, \;\; \partial^3_{t}h_j=O(1)\epsilon,\;\;\tilde{l}_j\partial^2_{t}\omega=O(1)\epsilon, \;\; \tilde{l}_j\partial^3_{t}\omega=O(1)\epsilon.
\end{split}
\]

Returning to the coordinate $(y,t)$, we next study the solvability of the following initial-boundary value problem
for the blowup system \eqref{Y-3} in time interval $[t_0, t_1]$
\begin{equation}\label{eq:2.98}
\begin{cases}
\eqref{Y-3}, \qquad\qquad\qquad\qquad\qquad\quad\qquad\quad \text{ in } D,\\[3pt]
K=1, \; \omega_j=\omega^0_j(y), \;h_j=h^0_j(y), \qquad\qquad\;\;\text{for }t=t_0,\; j=1,\cdots, n, \\[3pt]
\epsilon\omega_j=\epsilon\omega^*_j=O(1),\; h_j=h^*_j(t)=O(1)\epsilon, \quad\text{ on } y=a_0,\;  j=i+1, \cdots, n, \\[3pt]
\epsilon\omega_j=\epsilon\omega^*_j=O(1),\; h_j=h^*_j(t)=O(1)\epsilon, \quad\text{ on } y=b_0,\; j=1, \cdots, i-1,
\end{cases}
\end{equation}
where $\p_y^k\omega^0_j(y)=O(1)$ and $\p_y^kh^0_j(y)=O(1)\epsilon$ for $j\not=i$ and $k\ge1$,
$\omega^0_i(y)=O(1)$ and $h^0_i(y)=O(1)$, $\p_t^k\omega^*_j(t)|_{\text{$y=a_0$ or $y=b_0$}}=O(1)$ ($j\not=i$)
for $k\ge1$. Moreover, the compatibility conditions of all orders hold on the corners $(a_0, t_0)$
and $(b_0, t_0)$ for the initial-boundary values of \eqref{eq:2.98}.

\subsection{Solvability of the blowup system  \eqref{eq:2.1} and proof of Theorem \ref{T:2.1}}

In this subsection,  we show the existence of smooth solution $(\omega, K, h)$ to problem \eqref{eq:2.98} for  $t\in [t_0, t_1]$
and complete the proof of Theorem \ref{T:2.1}.
Note that due to the degeneracy of $K$ near the blowup time $T_{\ve}$,
we can sometimes think $y$ and $t$ as the new ``time variable" and ``space variable", respectively.

First of all, we start to construct the first approximate solution $(\varphi^{(0)}, K^{(0)}, \omega^{(0)})$ of \eqref{eq:2.2} such that the nondegenerate condition holds at some point.

Let $\tau:=\epsilon t$.
Then it follows from $\eqref{eq:2.2}$ that the corresponding solution for $\epsilon=0$ is
\begin{equation}\label{YH-3}
\begin{split}
&\text{$\bar\omega_i(x, \tau)=w^{i}_0(y)$ with $x=\bar\varphi(y, \tau)=y+\partial_{w_i}\lambda_i(0)w^{i}_0(y)\tau$},
\quad \text{$\bar\omega_j(y, \tau)=0$ for $j\not=i$}, \\
&\bar K(y, \tau)=1+\partial_{w_i}\lambda_i(0)\Big(w^{i}_0(y)\Big)'\tau.\\
\end{split}
\end{equation}

Choose a cut-off function $\chi(s)\in C^{\infty}(\mathbb{R})$ such that
\[
\chi(s)=1 \;\text{ for } s\leq \frac{1}{2};\quad \chi(s)=0 \;\text{ for }s\geq \f34.
\]

Since there exists a local smooth solution $(\varphi^{\epsilon}, K^{\epsilon}, \omega^{\epsilon})$ to the blowup system \eqref{eq:2.2} for
$t_0\leq t<T_\epsilon$, we then glue the local smooth solution $(\varphi^{\epsilon}, K^{\epsilon}, \omega^{\epsilon})$  and
\eqref{YH-3} to get the first approximate solution $(\varphi^{(0)}, K^{(0)}, \omega^{(0)})$ as
\begin{equation}\label{YHC-3}
\begin{cases}
\varphi^{(0)}(y, t)=\chi(\frac{t}{T_\epsilon})\varphi^{\epsilon}(y, t)+(1-\chi(\frac{t}{T_\epsilon}))\bar{\varphi}(y, t),\\[3pt]
K^{(0)}(y, t)=\chi(\frac{t}{T_\epsilon})K^{\epsilon}(y, t)+(1-\chi(\frac{t}{T_\epsilon}))\bar{K}(y, t),\\[5pt]
\omega^{(0)}(y, t)=\chi(\frac{t}{T_\epsilon})\omega^{\epsilon}(y, t)+(1-\chi(\frac{t}{T_\epsilon}))\bar{\omega}(y, t).
\end{cases}
\end{equation}

\vspace{0.2cm}
Let $\omega^{(m+1)}$, $h^{(m+1)}$ and $K^{(m+1)}$ for $m\in\Bbb N$ satisfy the following linearized system of \eqref{Y-3}
\begin{align}
&\partial_t K^{(m+1)}=\epsilon L^{(1)}(\epsilon D_{\epsilon}\omega^{(m)})\tilde{h}^{(m)}K^{(m)}+\epsilon\nabla\lambda_i\cdot r_i(\epsilon D_\epsilon\omega^{(m)})h^{(m)}_i, \label{eq:2.12}\\[3pt]
&\partial_th^{(m+1)}_i=\epsilon(\tilde{h}^{(m)})^{\top}Q^{(1)}(\epsilon D_{\epsilon}\omega^{(m)})\tilde{h}^{(m)}K^{(m)}
+\epsilon L^{(2)}(\epsilon D_{\epsilon}\omega^{(m)}) \tilde{h}^{(m)} h^{(m)}_i,\label{eq:2.13}\\[3pt]
&\big(K^{(m)}\partial_t+(\lambda_j-\lambda_i)(\epsilon D_\epsilon\omega^{(m)})\partial_y\big)h^{(m+1)}_j+\epsilon K^{(m)}\Big((\tilde{h}^{(m)})^{\top}Q^{(2)}(\epsilon D_\epsilon \omega^{(m)})\tilde{h}^{(m+1)}\notag+L^{(3)}(\epsilon D_{\epsilon}\omega^{(m)})\tilde{h}^{(m)}h^{(m+1)}_j\Big)\\
&\quad+\epsilon h^{(m)}_i\Big(r^{\top}_i(\epsilon D_\epsilon \omega^{(m)})Q^{(3)}(\epsilon D_\epsilon \omega^{(m)})\tilde{h}^{(m+1)}
+\nabla\lambda_i\cdot r_i(\epsilon D_{\epsilon}\omega^{(m)})h^{(m+1)}_j\Big)=0,\label{eq:2.14}\\[3pt]
&\tilde{l}_i(\epsilon D_{\epsilon}\omega^{(m)})\partial_t\omega^{(m+1)}=0,\label{eq:2.15}\\[3pt]
&\tilde{l}_j(\epsilon D_{\epsilon}\omega^{(m)})\Big(K^{(m)}\partial_t
+(\lambda_j-\lambda_i)(\epsilon D_{\epsilon}\omega^{(m)})\partial_y\Big)
\omega^{(m+1)}=0, \quad j\neq i\label{eq:2.16}
\end{align}
with the initial-boundary values
\begin{equation*}
\begin{cases}
\epsilon\omega^{(m+1)}_j=\epsilon\omega_j^*(t),\; h^{(m+1)}_j=h_j^*(t), \qquad\qquad\quad\quad \text{ on } y=a_0,\; j=i+1, \cdots, n,\\[3pt]
\epsilon\omega^{(m+1)}_j=\epsilon\omega_j^*(t),\; h^{(m+1)}_j=h_j^*(t), \qquad\qquad\quad\quad\text{ on } y=b_0,\;  j=1, \cdots, i-1,\\[3pt]
K^{(m+1)}=1, \; \omega_j^{(m+1)}=\omega^0_j(y), \;h^{(m+1)}_j=h^0_j(y), \quad\quad\text{for }\; t=t_0,\; j=1,\cdots, n.
\end{cases}
\end{equation*}
It is worth mentioning  that  the iterative scheme \eqref{eq:2.14} is delicately chosen. If $h^{(m+1)}$ is replaced
by  $h^{(m)}$ in the second and third terms of \eqref{eq:2.14}, then it is  difficult for us to directly get the uniform boundedness
of  $h_j^{(m+1)}$ ($j\not=i$) due to the appearance of $O(|h^{(m)}|^2)$.

\vspace{0.1cm}
Next, we establish the boundedness and convergence of $(K^{(m+1)}, h^{(m+1)}, \omega^{(m+1)})$. In this process,
we have to pay  special attentions whether  the small factor $\ve$ appears in each related term or not.
The proof is divided into  the following five steps.

\vspace{0.3cm}
\underline{\bf Step 1. Estimates of $K^{(m+1)}$, $h^{(m+1)}_i, \tilde{h}^{(m+1)}$ \text{ and } $\omega^{(m+1)}$}

\begin{lemma}\label{L:2.9}
It holds that
\begin{equation}\label{eq:2.17*}
|K^{(m)}|\leq \mathcal{K}_0, \;\; |h^{(m)}_i|\leq \mathcal{H}_i, \;\; |\tilde{h}^{(m)}|\leq \tilde{\mathcal{H}}=C\epsilon, \;\; ,\text{$|\epsilon\omega^{(m)}_j|\leq \mathcal{W}$ for $j\not=i$}, \;\; |\omega^{(m)}_i|\leq \mathcal{W},
\end{equation}
where $\mathcal{K}_0$, $\mathcal{H}_i$, $\tilde{\mathcal{H}}$ and $\mathcal{W}$ are some positive constants to be
determined later (see \eqref{eq:2.52-Y1} below).
\end{lemma}
\begin{proof}
The proof will be carried out by continuous induction. At first, \eqref{eq:2.17*} holds true for $m=0$
by \eqref{YHC-3}. Assume that \eqref{eq:2.17*} holds for $m$,  it is required to
show the validity for $m+1$.

\vspace{0.1cm}
Integrating \eqref{eq:2.12}-\eqref{eq:2.13} for $t\in [t_0, t_1]$
yields  that
\begin{equation}\label{eq:2.17}
|K^{(m+1)}|\leq 1+ C_1(\epsilon \mathcal{W})\tilde{\mathcal{H}}\mathcal{K}_0+C_2(\epsilon \mathcal{W})\mathcal{H}_i,
\end{equation}
\begin{equation}\label{eq:2.18}
|h^{(m+1)}_i|\leq \mathcal{H}_{i0}+C_3(\epsilon \mathcal{W})\tilde{\mathcal{H}}^2\mathcal{K}_0+C_4(\epsilon \mathcal{W})\tilde{\mathcal{H}}\mathcal{H}_i,
\end{equation}
where $\mathcal{H}_{i0}:=\displaystyle\max_{y\in [a_0, b_0]}|h^{0}_i(y)|$, and from now on, $C_k(\epsilon\mathcal{W})$ denotes a
smooth function depending on $\epsilon\mathcal{W}$.

\vspace{0.2cm}
Before evaluating $h^{(m+1)}_j(j\neq i)$, we need to figure out the trend of the first order operator
$K^{(m)}\partial_t+(\lambda_j-\lambda_i)(\epsilon D_\epsilon\omega^{(m)})\p_y$ on the boundaries $y=a_0$ and $y=b_0$,
especially, check the signs of $K^{(m)}$ on the boundaries. From \eqref{YD-1}, one has
\begin{equation}\label{Eq:2.43*}
\begin{split}
\partial_t K^{(m)}(y, t)=&\epsilon^2\sum_{j\neq i}\partial_{w_j}\lambda_i(\epsilon D_{\epsilon}\omega^{(m-1)})\partial_y\omega^{(m-1)}_j
+\epsilon\partial_{w_i}\lambda_i(\epsilon D_{\epsilon}\omega^{(m-1)})\partial_y\omega^{(m-1)}_i.
\end{split}
\end{equation}
Note that
\[
\epsilon\omega^{(m-1)}_j(a_0, t)=O(1),\quad h^{(m-1)}_j(a_0, t)=O(1)\epsilon\quad \text{ for }\;i+1\leq j\leq n,
\]
and
\[
\epsilon\omega^{(m-1)}_j(b_0, t)=O(1), \quad h^{(m-1)}_j(b_0, t)=O(1)\epsilon \quad \text{ for } 1\leq j\leq i-1.
\]
Along with \eqref{eq:2.4} and \eqref{Eq:2.40}, we have
\begin{equation*}
\begin{split}
&\partial_t \omega^{(m-1)}_j(a_0, t)=O(1), \;\;\;\partial_y \omega^{(m-1)}_j(a_0, t)=O(1)\;\text{ for } j\neq i,\\[3pt]
&\partial_t \omega^{(m-1)}_i(a_0, t)=O(1)\epsilon,\;\;\; \partial_y\omega_i^{(m-1)}(a_0, t)=O(1)\epsilon.
\end{split}
\end{equation*}
Thus
\begin{equation}\label{eq:2.104*}
|\partial_tK^{(m)}(a_0, t)|\leq C\epsilon^2,
\end{equation}
which yields that for $t\in [t_0, t_1]$ and small $\epsilon>0$,
\[
K^{(m)}(a_0, t)\geq 1-C\epsilon^2(t_1-t_0)>\frac{1}{2}.
\]
Similarly,
\[
K^{(m)}(b_0, t)\geq 1-C\epsilon^2(t_1-t_0)>\frac{1}{2}.
\]
Meanwhile, for $t\in [t_0, t_1]$ and small $\epsilon>0$,
\begin{equation*}
K^{(m)}(a_0, t)<\frac{3}{2}, \quad K^{(m)}(b_0, t)<\frac{3}{2}.
\end{equation*}

Therefore, we can apply the characteristics method to \eqref{eq:2.14}, and derive that
\begin{equation}\label{Eq:2.55}
\begin{split}
\sum_{j\neq i}|h^{(m+1)}_j(y, t)|\leq& \tilde{\mathcal{H}}_0+\Big(\epsilon C_5(\epsilon \mathcal{W})\tilde{\mathcal{H}}
+\epsilon C_6(\epsilon \mathcal{W})\mathcal{H}_i\Big)\sum_{k\neq i}|h^{(m+1)}_k|,
\qquad j\neq i,
\end{split}
\end{equation}
where
\[
\tilde{\mathcal{H}}_0=\max\Big(\max_{y\in [a_0, b_0]}\sum_{j\neq i}|h^{0}_j(y)|, \max_{t\in [t_0, t_1]}\sum^n_{j=i+1}|h_j^{(m+1)}(a_0, t)|, \max_{t\in [t_0, t_1]}\sum^{i-1}_{j=1}|h_j^{(m+1)}(b_0, t)|\Big).
\]
Then it yields that for sufficiently small $\epsilon>0$,
$\displaystyle\sum_{j\neq i}|h^{(m+1)}_j|\leq C\tilde{\mathcal{H}}_0\leq C\epsilon$.

On the other hand, if we set
\[
p_j^{(m+1)}=\tilde{l}_j(\epsilon D_{\epsilon}\omega^{(m)})\omega^{(m+1)}, j=1, \cdots, n,\ \text{ and }\;
p^{(m+1)}=(p^{(m+1)}_1, \cdots,  p^{(m+1)}_n)^{\top},
\]
then
\[
\omega^{(m+1)}=\sum^n_{k=1}p^{(m+1)}_k\tilde{r}_k(\epsilon D_\epsilon \omega^{(m)}).
\]
Therefore, it follows from \eqref{eq:2.15} and \eqref{eq:2.16} that
\begin{align}
&\partial_t p^{(m+1)}_i+\epsilon (\tilde{h}^{(m)})^{\top}M_1(\epsilon D_\epsilon\omega^{(m)})p^{(m+1)}=0,
\label{eq:2.23}\\[3pt]
&\Big(K^{(m)}\partial_t+(\lambda_j-\lambda_i)(\epsilon D_\epsilon\omega^{(m)})\partial_y\Big)p^{(m+1)}_j+
\epsilon K^{(m)}(\tilde{h}^{(m)})^{\top}M_2(\epsilon D_\epsilon \omega^{(m)})p^{(m+1)}\notag \\
&\qquad\quad+\epsilon h^{(m)}_i(r^{(m)}_i)^{\top}
M_3(\epsilon D_\epsilon\omega^{(m)})p^{(m+1)}=0,\qquad\quad j\neq i,\label{eq:2.24}
\end{align}
where and below $M_k(\epsilon D_\epsilon \omega^{(m)})$ stands for the $n\times n$ function matrix
depending on $\epsilon D_{\epsilon}\omega^{(m)}$.

Integrating \eqref{eq:2.23} with respect to $t$  on the interval $[t_0, t]$ yields
\begin{equation}\label{eq:2.25}
|p^{(m+1)}_i|\leq |p^{(m+1)}_i(y, t_0)|+C_7(\epsilon \mathcal{W})\tilde{\mathcal{H}}|p^{(m+1)}|.
\end{equation}

Along with \eqref{eq:2.25}, we apply the characteristics method to \eqref{eq:2.24}, and obtain that
\[
\max_{y\in [a_0,b_0]}|p^{(m+1)}|\leq \mathcal{P}_0+
\epsilon\Big(C_7(\epsilon \mathcal{W})+C_8(\epsilon\mathcal{W})\tilde{\mathcal{H}}\mathcal{K}_0+
C_9(\epsilon\mathcal{W})\mathcal{H}_i\Big)\max_{y\in [a_0,b_0]}|p^{(m+1)}|,
\]
where
\[
\mathcal{P}_0:=\displaystyle\max\Big(\max_{y\in [a_0, b_0]}\sum^n_{j=1}|p^{(m+1)}_j(y, t_0)|,
\max_{t\in [t_0, t_1]}\sum^n_{j=i+1}|p^{(m+1)}_j(a_0, t)|, \max_{t\in [t_0, t_1]}\sum^{i-1}_{j=1}|p^{(m+1)}_j(b_0, t)|\Big).
\]

This derives
\begin{equation}\label{eq:2.28}
|p^{(m+1)}|\leq C\mathcal{P}_0.
\end{equation}

Let $\mathcal{W}=C\mathcal{P}_0$, then it satisfies that in $D$,
\[
\epsilon|\omega^{(m+1)}_j(y, t)|\leq \mathcal{W}\quad\text{for $j\not=i$}, \quad |\omega^{(m+1)}_i(y, t)|\leq \mathcal{W}
\]
and then $|v^{(m+1)}(y, t)|=|\epsilon D_\epsilon\omega^{(m+1)}|\leq \epsilon\mathcal{W}$.
Based on \eqref{eq:2.17}, \eqref{eq:2.18}, and \eqref{eq:2.28}, when $\epsilon>0$
is small, one can choose
\begin{equation}\label{eq:2.52-Y1}
\mathcal{K}_0=1+C\mathcal{H}_{i0},\;\; \mathcal{H}_i=C\mathcal{H}_{i0}, \;\;\tilde{\mathcal{H}}=C\epsilon, \;\; \mathcal{W}=C\mathcal{P}_0,
\end{equation}
such that \eqref{eq:2.17*} still holds for $m+1$, where $C>1$ is a positive constant independent of $\epsilon$.

\end{proof}

\vspace{0.2cm}
\underline{\bf Step 2. Estimates of $\nabla K^{(m+1)}, \nabla h^{(m+1)}_i, \nabla \tilde{h}^{(m+1)}$ \text{ and } $\nabla \omega^{(m+1)}$}

\vspace{0.2cm}
Let
\begin{equation}\label{eq:2.30*}
\bar{h}^{(m+1)}_i=\tilde{l}_i(\epsilon D_{\epsilon}\omega^{(m)})\partial_y\omega^{(m+1)}, \quad \bar{h}^{(m+1)}_j
=\tilde{l}_j(\epsilon D_\epsilon \omega^{(m)})\partial_t\omega^{(m+1)} \;\; \text{for $j\neq i$}.
\end{equation}
In addition, one has
\begin{equation}\label{eq:2.54-Y2}
\begin{split}
\partial_y\omega^{(m+1)}(y, t)=&-\sum_{j\neq i}\frac{K^{(m)}\bar{h}^{(m+1)}_j}
{(\lambda_j-\lambda_i)(\epsilon D_{\epsilon}\omega^{(m)})}\tilde{r}_j(\epsilon D_\epsilon \omega^{(m)})
+\bar{h}^{(m+1)}_i\tilde{r}_i(\epsilon D_\epsilon \omega^{(m)}), \\
\partial_t\omega^{(m+1)}(y, t)=&\sum_{j\neq i}\bar{h}^{(m+1)}_j\tilde{r}_j(\epsilon D_{\epsilon}\omega^{(m)}).
\end{split}
\end{equation}

When $t=t_0$, we have
\[
\begin{split}
&\partial_t K^{(m+1)}=O(1)\epsilon, \;\; \partial_y K^{(m+1)}=0,\;\partial_y h_i^{(m+1)}=O(1), \;\;\partial_t h^{(m+1)}_i=O(1)\epsilon, \;\;\bar{h}^{(m+1)}_i=O(1),\\
&\bar{h}^{(m+1)}_j=O(1)\epsilon,\;\;\partial_t h^{(m+1)}_j=O(1)\epsilon,\;\; \partial_y h^{(m+1)}_j
=O(1)\epsilon, \quad j\neq i.
\end{split}
\]

\begin{lemma}\label{L:2.10}
The following estimates hold
\begin{equation}\label{eq:2.30}
|\nabla K^{(m)}|\leq \mathcal{K}_1,\;|\nabla h^{(m)}|\leq \mathcal{H}^0_1,\; |\nabla \omega^{(m)}|\leq \mathcal{W}_1.
\end{equation}
 In particular, $|\partial_t K^{(m)}|\leq C\epsilon$,  $|\partial_t h^{(m)}_i|\leq C\epsilon$, $|\nabla h^{(m)}_i|\leq \mathcal{H}^i_1,\; |\nabla\tilde{h}^{(m)}|\leq\tilde{\mathcal{H}}_1\leq C\epsilon$, where $\mathcal{K}_1$, $\mathcal{H}^0_1$, $\mathcal{W}_1$, $\mathcal{H}^i_1$
and $\tilde{\mathcal{H}}_1$ are some positive constants
to be determined later.
\end{lemma}

\begin{proof}

We will show that \eqref{eq:2.30} still holds for $m+1$. Taking the first order derivative of \eqref{eq:2.12} yields
\begin{equation}\label{eq:2.31}
\begin{split}
\partial_t\nabla K^{(m+1)}=&\epsilon \nabla\Big(L^{(1)}(\epsilon D_\epsilon\omega^{(m)})
\tilde{h}^{(m)}\Big)K^{(m)}+\epsilon L^{(1)}(\epsilon D_{\epsilon}\omega^{(m)})\tilde{h}^{(m)}
\nabla K^{(m)}\\
&+\epsilon^2 \nabla_w(\nabla\lambda_i\cdot r_i(\epsilon D_\epsilon \omega^{(m)}))D_{\epsilon}\nabla\omega^{(m)}h^{(m)}_i
+\epsilon\nabla\lambda_i\cdot r_i(\epsilon D_{\epsilon}\omega^{(m)})\nabla h^{(m)}_i.
\end{split}
\end{equation}
Integrating \eqref{eq:2.31} with respect to $t$ on the interval $[t_0, t_1]$ yields
\begin{equation}\label{eq:2.32*}
\begin{split}
|\nabla K^{(m+1)}|\leq& \mathcal{K}^0_1+\epsilon C_{10}(\epsilon\mathcal{W})\tilde{\mathcal{H}}\mathcal{W}_1+C_1(\epsilon \mathcal{W})(\tilde{\mathcal{H}}_1+\tilde{\mathcal{H}}\mathcal{K}_1)+C_{11}(\epsilon \mathcal{W})\epsilon \mathcal{W}_1+ C_2(\epsilon\mathcal{W})\mathcal{H}^i_1\\
\leq & \mathcal{K}^0_1+C\mathcal{H}^0_1+C\epsilon(\mathcal{K}_1+\mathcal{W}_1),
\end{split}
\end{equation}
where $\mathcal{K}^0_1=\displaystyle\max_{y\in[a_0, b_0]}|\nabla K^{(m+1)}(y, t_0)|$. In particular,
\[
|\partial_t K^{(m+1)}(y, t)|\leq C\epsilon,
\]
which will play a crucial role in establishing the uniform boundedness of the higher order derivatives of
$(K^{(m+1)}, h^{(m+1)}, \omega^{(m+1)})$ for all $m$.

Similarly, it follows from \eqref{eq:2.13} that
\[
\begin{split}
\partial_t \partial_y h^{(m+1)}_i=&\epsilon \partial_y(\tilde{h}^{(m)}Q^{(1)})\tilde{h}^{(m)}K^{(m)}
+\epsilon (\tilde{h}^{(m)})^{\top}Q^{(1)}\partial_y\tilde{h}^{(m)}K^{(m)}
+\epsilon (\tilde{h}^{(m)})^{\top}Q^{(1)}\tilde{h}^{(m)}\partial_y K^{(m)}\\
+&\epsilon^2 (\nabla L^{(2)}D_\epsilon \partial_y\omega^{(m)})^{\top}\tilde{h}^{(m)}h^{(m)}_i+\epsilon L^{(2)}(\epsilon D_{\epsilon}\omega^{(m)})\partial_y\tilde{h}^{(m)}h^{(m)}_i+\epsilon
L^{(2)}(\epsilon D_{\epsilon}\omega^{(m)}) \tilde{h}^{(m)} \partial_y h^{(m)}_i.
\end{split}
\]
Let $\mathcal{H}^{i0}_1=\max_{y\in [a_0, b_0]}|\nabla h^{(m+1)}_i(y, t_0)|$, one then has
\begin{equation}\label{eq:2.33}
\begin{split}
|\partial_yh^{(m+1)}_i|\leq & \mathcal{H}^{i0}_1+C_3(\epsilon \mathcal{W})\tilde{\mathcal{H}}\tilde{\mathcal{H}}_1
+\epsilon C_{12}(\epsilon \mathcal{W})\mathcal{W}_1\tilde{\mathcal{H}}^2+C_3(\epsilon\mathcal{W})\tilde{\mathcal{H}}^2\mathcal{K}_1\\
&+\epsilon\tilde{\mathcal{H}}C_{13}(\epsilon \mathcal{W})\mathcal{W}_1+C_4(\epsilon \mathcal{W})(\mathcal{H}_i\tilde{\mathcal{H}}_1+\tilde{\mathcal{H}}\mathcal{H}^i_1)\\
\leq &\mathcal{H}^{i0}_1+C(\tilde{\mathcal{H}}_1+\tilde{\mathcal{H}}\mathcal{H}^i_1).
\end{split}
\end{equation}

In addition, it follows from \eqref{eq:2.13} that
\begin{equation}\label{Eq:2.59*}
|\partial_t h^{(m+1)}_i(y, t)|\leq |\partial_t h^{(m+1)}_i(y, t_0)|+C\tilde{\mathcal{H}}_1\leq C\epsilon.
\end{equation}

In the rest, we establish the uniform boundedness of $\nabla \omega^{(m)}$.
To this end, we have to estimate $\bar{h}^{(m+1)}$ in terms of \eqref{eq:2.54-Y2}.
From the expression \eqref{eq:2.30*} of $\bar{h}_i^{(m+1)}$, one can calculate directly that
\begin{equation}\label{eq:2.34}
\begin{split}
\partial_t\bar{h}^{(m+1)}_i
=&\epsilon(\nabla\tilde{l}_i(\epsilon D_\epsilon\omega^{(m)}) D_{\epsilon}\sum_{j\neq i}h^{(m)}_j\tilde{r}^{(m)}_j)^{\top}\Big(-\sum_{l\neq i}\frac{K^{(m)}\bar{h}^{(m+1)}_l}
{(\lambda_l-\lambda_i)^{(m)}}\tilde{r}^{(m)}_l+\bar{h}^{(m+1)}_i\tilde{r}^{(m)}_i\Big)\\
&-\epsilon\Big(
\nabla\tilde{l}_i(\epsilon D_{\epsilon}\omega^{(m)})D_\epsilon(-\sum_{j\neq i}\frac{ K^{(m)}h^{(m)}_j}{(\lambda_j-\lambda_i)^{(m)}}\tilde{r}^{(m)}_j+h^{(m)}_i\tilde{r}^{(m)}_i)
\Big)^{\top}(\sum_{l\neq i}\bar{h}^{(m+1)}_l\tilde{r}_l^{(m)})\\
=&\epsilon K^{(m)}(\tilde{h}^{(m)})^{\top}Q^{(4)}(\epsilon D_{\epsilon}\omega^{(m)})\tilde{\bar{h}}^{(m+1)}+\epsilon (\tilde{h}^{(m)})^{\top}M_4(\epsilon D_{\epsilon}\omega^{(m)})r^{(m)}_i\bar{h}^{(m+1)}_i\\
&\quad+\epsilon (r^{(m)}_i)^{\top}M_5(\epsilon D_\epsilon\omega^{(m)})\tilde{\bar{h}}^{(m+1)}h^{(m)}_i,
\end{split}
\end{equation}
where $\tilde{\bar{h}}^{(m+1)}$ corresponds to the vector $(\bar{h}_1^{(m+1)}, \cdots, \hat{\bar{h}}_i^{(m+1)}, \cdots, \bar{h}_n^{(m+1)})^{\top}$
with 0 replacing $\bar{h}_i^{(m+1)}$ in ${\bar{h}}^{(m+1)}$.

\vspace{0.1cm}
Integrating both sides of equation \eqref{eq:2.34} with respect to $t$ yields
\begin{equation}\label{eq:2.35}
\begin{split}
\max_{y\in [a_0, b_0] \atop t\in[t_0, t_1]}|\bar{h}^{(m+1)}_i|\leq& \bar{\mathcal{H}}_{i0}+
C_{14}(\epsilon\mathcal{W})\mathcal{H}\max_{y\in [a_0, b_0] \atop t\in[t_0, t_1]}|\tilde{\bar{h}}^{(m+1)}|
+C_{15}(\epsilon\mathcal{W})\tilde{\mathcal{H}}\max_{y\in [a_0, b_0] \atop t\in[t_0, t_1]}|\bar{h}^{(m+1)}_i|\\
\lesssim &\bar{\mathcal{H}}_{i0}+C\mathcal{H}\max_{y\in [a_0, b_0] \atop t\in[t_0, t_1]}|\tilde{\bar{h}}^{(m+1)}|
+C\epsilon\max_{y\in [a_0, b_0]\atop t\in [t_0, t_1]}|\bar{h}^{(m+1)}_i|
\end{split}
\end{equation}
with $\bar{\mathcal{H}}_{i0}=\max_{y\in [a_0, b_0]}|\bar{h}^{(m+1)}_i(y, t_0)|$, here and below $A\lesssim B$
stands for $A\leq CB$ with $C$ being a generic positive constant independent of $\epsilon$.

\vspace{0.1cm}
We now estimate $\bar{h}^{(m+1)}_j=\tilde{l}_j(\epsilon D_{\epsilon}\omega^{(m)})\partial_t\omega^{(m+1)}$ for
$j\neq i$. Differentiating \eqref{eq:2.16} with respect to $t$, one can obtain that
\[
\begin{split}
&(K^{(m)}\partial_t+(\lambda_j-\lambda_i)(\epsilon D_\epsilon\omega^{(m)})\partial_y)\bar{h}^{(m+1)}_j
+\epsilon(\lambda_j-\lambda_i)(\epsilon D_\epsilon\omega^{(m)})\Big((\nabla\tilde{l}_j(\epsilon D_\epsilon\omega^{(m)})D_{\epsilon}\partial_t\omega^{(m)})^{\top}\partial_y \omega^{(m+1)}
\\&-(\nabla \tilde{l}_j(\epsilon D_\epsilon\omega^{(m)}) D_{\epsilon}\partial_y\omega^{(m)})^{\top}\partial_t\omega^{(m+1)}\Big)
+\partial_t K^{(m)}\bar{h}^{(m+1)}_j+\epsilon\nabla(\lambda_j-\lambda_i)D_\epsilon\partial_t
\omega^{(m)}\tilde{l}^{(m)}_j\partial_y\omega^{(m+1)}=0.
\end{split}
\]
This can be rewritten as
\begin{equation}\label{eq:2.36}
\begin{split}
&\Big(K^{(m)}\partial_t+(\lambda_j-\lambda_i)(\epsilon D_\epsilon\omega^{(m)})\partial_y\Big)\bar{h}^{(m+1)}_j+
\epsilon K^{(m)}\Big((\tilde{h}^{(m)})^{\top}Q_5(\epsilon D_\epsilon \omega^{(m)})\tilde{\bar{h}}^{(m+1)}+L^{(3)}(\epsilon D_{\epsilon}\omega^{(m)})\tilde{h}^{(m)}\bar{h}^{(m+1)}_j\Big)\\[3pt]
&\;\;+\epsilon(\tilde{h}^{(m)})^{\top}M_6(\epsilon D_\epsilon\omega^{(m)})\bar{h}^{(m+1)}_i r^{(m)}_i
+\epsilon(r^{(m)}_i)^{\top}M_7(\epsilon D_\epsilon\omega^{(m)})\tilde{\bar{h}}^{(m+1)}h^{(m)}_i+\partial_t K^{(m)}\bar{h}^{(m+1)}_j=0.
\end{split}
\end{equation}
It follows from the characteristics method that
\begin{equation}\label{Eq:2.38}
\begin{split}
\sum_{j\neq i}|\bar{h}^{(m+1)}_j|&\lesssim \tilde{\bar{\mathcal{H}}}^{0}_1+C\epsilon\max_{s\in [t_0, t]}|\tilde{\bar{h}}^{(m+1)}|+ C\epsilon^2\max_{s\in [t_0, t]}|\bar{h}^{(m+1)}_i|,
\end{split}
\end{equation}
where
\[
\tilde{\bar{\mathcal{H}}}^{0}_1:=\max\Big(\max_{y\in[a_0, b_0]}\sum_{j\neq i}|\bar{h}^{(m+1)}_j(y, t_0)|, \max_{t\in [t_0, t_1]}\sum^n_{j=i+1}|\bar{h}^{(m+1)}_j(a_0, t)|, \max_{t\in [t_0, t_1]}\sum^{i-1}_{j=1}|\bar{h}^{(m+1)}_j(b_0, t)|\Big).
\]
Together with \eqref{eq:2.35}, we arrive at
\begin{equation}\label{eq:2.40}
\max_{y\in [a_0, b_0]\atop s\in [t_0, t]}|\tilde{\bar{h}}^{(m+1)}(y, t)|\leq  C \tilde{\bar{\mathcal{H}}}^{0}_1 \leq C \epsilon\tilde{\mathcal{W}}^0_1,\quad
\max_{y\in [a_0, b_0]\atop s\in [t_0, t]}|\bar{h}^{(m+1)}_i(y, t)|\leq C\bar{\mathcal{H}}_{i0},
\end{equation}
where
\[
\tilde{\mathcal{W}}^0_1:=\max\Big(\max_{y\in[a_0, b_0]}\sum_{j\neq i}|\nabla \omega^{(m+1)}_j(y, t_0)|, \max_{t\in [t_0, t_1]}\sum^n_{j=i+1}|\nabla\omega^{(m+1)}_j(a_0, t)|, \max_{t\in [t_0, t_1]}\sum^{i-1}_{j=1}|\nabla\omega^{(m+1)}_j(b_0, t)|\Big).
\]

In the following, we treat $|\nabla h^{(m+1)}_j|$ for $j\neq i$.
Differentiating \eqref{eq:2.14} with respect to $t$
and subsequently taking  direct computations, one has
\[
\begin{split}
&|\Big(K^{(m)}\partial_t+(\lambda_j-\lambda_i)(\epsilon D_{\epsilon}\omega^{(m)})\partial_y\Big)\partial_t h^{(m+1)}_j
+\partial_t K^{(m)}\partial_t h^{(m+1)}_j+\epsilon\nabla(\lambda_j-\lambda_i)D_{\epsilon}\partial_t\omega^{(m)}\partial_y h^{(m+1)}_j|\\
\leq &C\epsilon^3+C\epsilon^2\tilde{\mathcal{H}}_1+C\epsilon\sum_{j\neq i}|\partial_t h^{(m+1)}_j|.
\end{split}
\]
In addition, from \eqref{eq:2.14}, it yields that
\[
|\partial_y h^{(m+1)}_j|\leq C|\partial_t h^{(m+1)}_j|+C\epsilon^2.
\]
Together with the same arguments as in  Step $1$, we arrive at
\[
\max_{y\in [a_0, b_0]\atop t\in [t_0, t_1]}\sum_{j\neq i}|\partial_t h^{(m+1)}_j|\lesssim \tilde{\mathcal{H}}^0_1
+C\epsilon^2 \tilde{\mathcal{H}}_1
+C\epsilon\max_{y\in [a_0, b_0]}\sum_{j\neq i}|\partial_t  h^{(m+1)}_j|,
\]
where $\tilde{\mathcal{H}}^0_1=\max\Big(\displaystyle\max_{y\in[a_0, b_0]}\sum_{j\neq i}|\partial_t h^{(m+1)}_j(y, t_0)|, C\epsilon\Big)$.
Then this means
\begin{equation}\label{eq:2.41}
|\partial_t h^{(m+1)}_j|\leq C\tilde{\mathcal{H}}^{0}_1, \quad |\partial_y h^{(m+1)}_j|\leq C\tilde{\mathcal{H}}^0_1.
\end{equation}

Thus, for small $\epsilon>0$, by choosing
\[
\mathcal{K}_1=\mathcal{K}^0_1+C\mathcal{H}^0_1, \;\; \mathcal{H}^0_1=C(\mathcal{H}^{i0}_1+\tilde{\mathcal{H}}^{0}_1), \;\;\tilde{\mathcal{H}}_1= C\tilde{\mathcal{H}}^0_1,\;\;\mathcal{W}_1=C\big(\bar{\mathcal{H}}_{i0}+\tilde{\mathcal{W}}^0_1\big)
\]
and $\mathcal{H}^{i}_1=C\mathcal{H}^{i0}_1$, the estimates \eqref{eq:2.32*}-\eqref{Eq:2.59*}, \eqref{eq:2.35}, \eqref{eq:2.40}
and \eqref{eq:2.41} hold for $m+1$.

\end{proof}

\vspace{0.08cm}
\underline{\textbf{Step $3$. Estimates of $\nabla^2 K^{(m+1)}$, $\nabla^2 h^{(m+1)}$, and $\nabla^2 \omega^{(m+1)}$}}
\vspace{0.3cm}

For $k=1, \cdots, n$, set
\[
\begin{split}
&\bar{q}^{(m+1)}_k:=\tilde{l}_k(\epsilon D_{\epsilon}\omega^{(m)})\partial^2_t\omega^{(m+1)},
\bar{z}^{(m+1)}_k:=\tilde{l}_k(\epsilon D_{\epsilon}\omega^{(m)})\partial^2_{ty}\omega^{(m+1)}, E^{(m+1)}_k:=\tilde{l}_k(\epsilon D_{\epsilon}\omega^{(m)})\partial^2_{y}\omega^{(m+1)}.
\end{split}
\]
Based on \eqref{eq:2.15} and direct calculations, one has
\[
\bar{q}^{(m+1)}_i=-\epsilon(\nabla\tilde{l}_i D_{\epsilon}\partial_t\omega^{(m)})^{\top}\partial_t\omega^{(m+1)}
\]
and
\[
\partial^2_t\omega^{(m+1)}=\sum_{j\neq i}\bar{q}^{(m+1)}_j\tilde{r}_j(\epsilon D_{\epsilon}\omega^{(m)})
-\epsilon(\nabla\tilde{l}_i D_{\epsilon}\partial_t\omega^{(m)})^{\top}\partial_t\omega^{(m+1)}\tilde{r}_i(\epsilon D_{\epsilon}\omega^{(m)}).
\]
Similarly,
\begin{align*}
\partial^2_{ty}\omega^{(m+1)}=&\sum_{j\neq i}\bar{z}^{(m+1)}_j\tilde{r}_j(\epsilon D_{\epsilon}\omega^{(m)})-\epsilon(\nabla \tilde{l}_i D_{\epsilon}\partial_y\omega^{(m)})^{\top}\partial_t\omega^{(m+1)}\tilde{r}_i(\epsilon D_{\epsilon}\omega^{(m)}),\\
\partial^2_{y}\omega^{(m+1)}=&\sum_{j=1}^n E^{(m+1)}_j\tilde{r}_j(\epsilon D_{\epsilon}
\omega^{(m)}).
\end{align*}
Then we can estimate $\nabla^2 K^{(m)}$, $\nabla^2 h^{(m)}$ and $\nabla^2 \omega^{(m)}$ as follows.
\begin{lemma}\label{L:2.11}
The following estimates hold
\begin{equation}\label{eq:2.44}
|\nabla^2 K^{(m)}|\leq \mathcal{K}_2,\; |\nabla^2 h^{(m)}|\leq \mathcal{H}^0_2, \;|\nabla^2 \omega^{(m)}|\leq \mathcal{W}_2,
\end{equation}
where $|\partial^2_t K^{(m)}|\leq C\epsilon$, $|\partial^2_t h^{(m)}|\leq C\epsilon$, $|\nabla^2 h^{(m)}_i|\leq \mathcal{H}^i_2$, $|\nabla^2 \tilde{h}^{(m)}|\leq \tilde{\mathcal{H}}_2\leq C\epsilon$,  $\mathcal{K}_2$, $\mathcal{H}^0_2$, $\mathcal{W}_2$,   $\mathcal{H}^i_2$ and $\tilde{\mathcal{H}}_2$ are some positive constants.
\end{lemma}
\begin{proof}
We will show that \eqref{eq:2.44} still holds for $m+1$.
At first, we estimate $\nabla^2 K^{(m+1)}$. In fact, it suffices only to treat $\partial^2_t K^{(m+1)}$ since
the other second order derivatives of $K^{(m+1)}$ can be proved analogously.

Taking the first and second order derivatives of \eqref{eq:2.12} with respect to $t$, respectively, one has
\begin{equation}\label{Eq:2.45}
\begin{split}
\partial^2_t K^{(m+1)}=&\epsilon^2(\nabla L^{(1)}(\epsilon D_{\epsilon}\omega^{(m)})
D_{\epsilon}\partial_t\omega^{(m)})^{\top}\tilde{h}^{(m)}K^{(m)}+\epsilon L^{(1)}\partial_t \tilde{h}^{(m)}K^{(m)}+\epsilon L^{(1)}\tilde{h}^{(m)}\partial_t K^{(m)}\\
&+\epsilon^2\nabla(\nabla\lambda_i\cdot r_i^{(m)})D_{\epsilon}\partial_t\omega^{(n)}h_i^{(m)}
+\epsilon\nabla\lambda_i\cdot r_i(\epsilon D_\epsilon\omega^{(m)})\partial_t h^{(m)}_i
\end{split}
\end{equation}
and
\begin{equation}\label{Eq:2.46}
\begin{split}
|\partial^3_t K^{(m+1)}|\leq & C \epsilon |\partial^2_{t} h^{(m)}|+C\epsilon^2\Big(|\partial^2_t K^{(m)}|+|\partial^2_t \omega^{(m)}|\Big).
\end{split}
\end{equation}
Similarly, we have
\begin{equation}\label{Eq:2.71*}
|\partial^3_{tty} K^{(m+1)}|\leq C\epsilon|\partial^2_{ty} h^{(m)}|+C\epsilon^2\Big(|\partial^2_{ty}K^{(m)}|+|\partial^2_{ty}\omega^{(m)}|\Big)+C\epsilon^2
\end{equation}
and
\begin{equation}\label{Eq:2.72*}
|\partial^3_{tyy} K^{(m+1)}|\leq C\epsilon |\partial^2_{y} h^{(m)}|+C\epsilon^2\Big(|\partial^2_{y}K^{(m)}|+|\partial^2_{y}\omega^{(m)}|\Big)+C\epsilon^2.
\end{equation}
Integrating \eqref{Eq:2.46}-\eqref{Eq:2.72*} with respect to $t$ yields
\[
\max_{y\in [a_0, b_0]\atop s\in [t_0, t_1]}|\nabla^2 K^{(m+1)}|\leq \mathcal{K}^0_2+C\mathcal{H}^0_2, \quad \text{ with }\mathcal{K}^0_2:=\max_{y\in [a_0, b_0]}|\nabla^2 K^{(m+1)}(y, t_0)|.
\]
In particular, it holds that
\[
|\partial^2_t K^{(m+1)}(y, t)|\leq C\epsilon.
\]

Next, we derive the boundedness of $\nabla^2 h^{(m+1)}$.
Differentiating \eqref{eq:2.13} with respect to $y$ twice, then $\partial^2_{y}h^{(m+1)}_i$ satisfies
\begin{equation}\label{eq:2.48*}
\begin{split}
|\partial_t\partial^2_y h^{(m+1)}_i|\leq& C\epsilon^2 |\partial^2_{y} h^{(m)}|+C\epsilon^3|\partial^2_{y}K^{(m)}|+C\epsilon^3|\partial^2_{y}\omega^{(m)}|+C\epsilon|\partial^2_y\tilde{h}^{(m)}|+C\epsilon^2.
\end{split}
\end{equation}
Then integrating  \eqref{eq:2.48*} with respect to $t$ in the interval $[t_0, t]$ yields
\[
|\partial^2_{y}h^{(m+1)}_i|\leq |\partial^2_y h^{(m+1)}_{i}(y, t_0)|+C|\partial^2_y\tilde{h}^{(m)}|+C\epsilon |\partial^2_y h^{(m)}|.
\]
Similarly, one has
\[
\begin{split}
|\partial^3_t h^{(m+1)}_i|&\leq C\epsilon |\partial^2_{t}\tilde{h}^{(m)}|+C\epsilon^2 |\partial^2_{t} h^{(m)}|+C\epsilon^3+C\epsilon^3|\partial^2_{t}K^{(m)}|+C\epsilon^3|\partial^2_{t}\omega^{(m)}|,\\
|\partial^3_{tty} h^{(m+1)}_i|&\leq C\epsilon |\partial^2_{ty}\tilde{h}^{(m)}|+C\epsilon^2 |\partial^2_{ty} h^{(m)}|+C\epsilon^2+C\epsilon^3|\partial^2_{ty}K^{(m)}|+C\epsilon^3|\partial^2_{ty}\omega^{(m)}|.
\end{split}
\]
Therefore, it holds that
\[
\max_{y\in [a_0, b_0]\atop t\in [t_0, t_1]}|\nabla^2 h^{(m+1)}_i|\leq \mathcal{H}^{i0}_2+C\tilde{\mathcal{H}}_2\;\;
\text{ with }\mathcal{H}^{i0}_2=\max_{y\in [a_0, b_0]}|\nabla^2 h^{(m+1)}_i(y, t_0)|.
\]
In particular,
\[
|\partial^2_{t} h^{(m+1)}_i|\leq |\partial^2_{t}h^{(m+1)}_i(y, t_0)|+C|\partial^2_{t} \tilde{h}^{(m)}|\leq C\epsilon.
\]

In addition, it follows from direct but tedious computation that
\begin{equation}\label{eq:2.50*}
\begin{split}
&\big|\Big(K^{(m)}\partial_t+(\lambda_j-\lambda_i)(\epsilon D_\epsilon\omega^{(m)})\partial_y\Big)\partial^2_{t}h^{(m+1)}_j+2
\partial_t K^{(m)}\partial^2_{t}h^{(m+1)}_j\\
+&\epsilon h^{(m)}_i\left((r_i(\epsilon D_\epsilon\omega^{(m)}))^{\top}Q^{(3)}\partial^2_t\tilde{h}^{(m+1)}
+\nabla\lambda_i\cdot r_i(\epsilon D_{\epsilon}\omega^{(m)})\partial^2_t h^{(m+1)}_j\right)\big|\\
\leq& C\epsilon^2+C\epsilon^2|\partial^2_{t}h^{(m)}|+C\epsilon^2 |\partial^2_{t}\tilde{h}^{(m+1)}|
+C\epsilon^2|\partial^2_{t}\omega^{(m)}|,
\end{split}
\end{equation}
where we have used the fact that
\begin{equation}\label{Eq:2.88*}
|\partial^2_t K^{(m)}|\leq C\epsilon, \;\;|\partial^2_{yt}h^{(m+1)}_j|\leq C|\partial^2_t h^{(m+1)}_j|+C\epsilon^2.
\end{equation}

This derives that
\[
\sum_{j\neq i}|\partial^2_{t} h^{(m+1)}_j|\leq \tilde{\mathcal{H}}^0_2 +C\epsilon^2+C\epsilon^2(\mathcal{W}_2+\mathcal{H}^0_2)
+C\epsilon|\partial^2_{t}\tilde{h}^{(m+1)}|
\]
and
\[
\max_{y\in[a_0, b_0]\atop t\in [t_0, t_1]}\sum_{j\neq i}|\partial^2_{t} h^{(m+1)}_j|\leq C\tilde{\mathcal{H}}^0_2\;\; \text{ with }\tilde{\mathcal{H}}^0_2=\max\Big(\displaystyle\max_{y\in[a_0, b_0]}\sum_{j\neq i}|\partial^2_t h^{(m+1)}_j(y, t_0)|, C\epsilon\Big).
\]
In particular, $|\partial^2_t h^{(m+1)}_j|\leq C\epsilon$. Differentiating \eqref{eq:2.14} with respect to $y$
and taking direct estimates yield that
\[
|\partial^2_{y}h^{(m+1)}_j|\leq C|\partial^2_{yt}h^{(m+1)}_j|+C\epsilon,
\]
along with \eqref{Eq:2.88*}, this implies
\[
\max_{y\in [a_0, b_0]\atop t\in [t_0, t_1]}\sum_{j\neq i}|\partial^2_{yt}h^{(m+1)}_j|\leq C\tilde{\mathcal{H}}^0_2\leq C\epsilon,
\quad \max_{y\in [a_0, b_0\atop t\in [t_0, t_1]}\sum_{j\neq i}|\partial^2_{y}h^{(m+1)}_j|\leq C\tilde{\mathcal{H}}^0_2\leq C\epsilon, \quad
j\neq i.
\]

In the following, we estimate $\partial^2_{t}\omega^{(m+1)}_j (j\neq i)$. Differentiating \eqref{eq:2.16} with respect to $t$ twice
and taking direct but tedious computations, one has
\[
\begin{split}
&\big|\Big(K^{(m)}\partial_t+(\lambda_j-\lambda_i)(\epsilon D_\epsilon\omega^{(m)})\partial_y\Big)\bar{q}^{(m+1)}_j+2\partial_t K^{(m)}\bar{q}^{(m+1)}_j\big|\\
\leq &C\epsilon^2\max_{y\in [a_0, b_0]}\sum_{j\neq i}\Big(|\bar{q}^{(m+1)}_j|+|\bar{z}^{(m+1)}_j|\Big)+C\epsilon^2\sum_{j\neq i}|\bar{q}^{(m)}_j|+C\epsilon|\partial^2_t K^{(m)}|
\end{split}
\]
and
$$|\bar{z}^{(m+1)}_j|\leq C|\bar{q}^{(m+1)}_j|+C\epsilon^2.$$
It follows from  the characteristics method that
\begin{equation}\label{eq:2.49*}
\sum_{j\neq i}|\bar{q}^{(m+1)}_j|\leq \tilde{\mathcal{Q}}_2+C\epsilon \sum_{j\neq i}|\bar{q}^{(m+1)}_j|
+C\epsilon^2\max_{y\in [a_0, b_0] \atop t\in [t_0, t_1]}\sum_{j\neq i}|\bar{q}^{(m)}_j(y, t)|+C\epsilon^2,
\end{equation}
where
\[
\tilde{\mathcal{Q}}_2:=\max\Big(\max_{y\in[a_0, b_0]}\sum_{j\neq i}|\bar{q}^{(m+1)}_j(\cdot, t_0)|, C\epsilon\Big).
\]

\vspace{0.1cm}
Differentiating \eqref{eq:2.16} with respect to $y$ yields that
\[
\begin{split}
(\lambda_j-\lambda_i)(\epsilon D_\epsilon\omega^{(m)})E^{(m+1)}_j=&-K^{(m)}\bar{z}^{(m+1)}_j-\partial_y K^{(m)}\bar{h}^{(m+1)}_j
-\epsilon\nabla(\lambda_j-\lambda_i)D_\epsilon\partial_y\omega^{(m)}\tilde{l}_j(\epsilon D_\epsilon\omega^{(m)})\\
&-\epsilon(\nabla\tilde{l}_j D_\epsilon\partial_y\omega^{(m)})(K^{(m)}\partial_t\omega^{(m+1)}
+(\lambda_j-\lambda_i)(\epsilon D_\epsilon\omega^{(m)})
\partial_y\omega^{(m+1)}),
\end{split}
\]
which implies that $|E^{(m+1)}_j|\leq C|\bar{z}^{(m+1)}_j|+C\epsilon$.

Thus
\[
\sum_{j\neq i}|\bar{q}^{(m+1)}_j|\leq C\tilde{\mathcal{Q}}_2\leq C\epsilon\tilde{\mathcal{W}}_2,\;\; \tilde{\mathcal{W}}_2:=\max_{y\in[a_0, b_0]}\sum_{j\neq i}|\partial^2_t\omega^{(m+1)}_j(y, t_0)|,
\]
and
\begin{equation}\label{eq:2.51}
\sum_{j\neq i}|\bar{z}^{(m+1)}_j|\leq C\epsilon\tilde{\mathcal{W}}_2,
\quad \sum_{j\neq i}|E^{(m+1)}_j|\leq C\epsilon\tilde{\mathcal{W}}_2.
\end{equation}

Next, we establish the estimates of $E^{(m+1)}_i, \bar{q}^{(m+1)}_i$ and $\bar{z}^{(m+1)}_i$.
By \eqref{eq:2.15}, one has
\begin{equation}\label{Eq:2.57}
\begin{split}
|\partial_t E^{(m+1)}_i|=&\big|\epsilon(\nabla\tilde{l}_iD_{\epsilon}\partial_t\omega^{(m)})^{\top}\sum_{j=1}^n E^{(m+1)}_j\tilde{r}^{(m)}_j-\partial^2_{y}\tilde{l}_i(\epsilon D_{\epsilon}\omega^{(m)})\sum_{k\neq i}\bar{h}^{(m+1)}_k\tilde{r}_k^{(m)}\\
&-2\epsilon(\nabla\tilde{l}_i D_{\epsilon}\partial_y\omega^{(m)})^{\top}\Big(\sum_{k\neq i}\bar{z}^{(m+1)}_k\tilde{r}^{(m)}_k
-\epsilon(\nabla \tilde{l}_i D_{\epsilon}\partial_y\omega^{(m)})^{\top}\partial_t\omega^{(m+1)}\tilde{r}^{(m)}_i(\epsilon D_{\epsilon}\omega^{(m)})\Big)\big|\\
\lesssim& \epsilon^2\sum_{k=1}^{n}|E^{(m+1)}_k|+C\epsilon\sum_{k\neq i}|\bar{z}^{(m+1)}_k|+C\epsilon^2 |E^{(m)}_i|.
\end{split}
\end{equation}
Then
\begin{equation}\label{Eq:2.93}
|E^{(m+1)}_i|\lesssim |E^{(m+1)}_{i}(y, t_0)|+C\epsilon\sum^n_{k=1}|E^{(m+1)}_k|
+C\sum_{k\neq i}|\bar{z}^{(m+1)}_k|+C\epsilon |E^{(m)}_i|.
\end{equation}
In addition,
\begin{align}
\big|\partial_t\bar{q}^{(m+1)}_i\big|=&\big|\partial^2_t\tilde{l}_i(\epsilon D_{\epsilon}\omega^{(m)})\partial_t\omega^{(m+1)}
+\epsilon(\nabla\tilde{l}_iD_{\epsilon}\partial_t\omega^{(m)})^{\top}\partial^2_t\omega^{(m+1)}\big|
\lesssim \epsilon^2\sum^n_{j=1}|\bar{q}^{(m+1)}_j|+\epsilon^2 \mathcal{W}_2,\label{Eq:2.58}\\
\big|\partial_t\bar{z}^{(m+1)}_i\big|=&\big|\partial^2_{yt}\tilde{l}_i(\epsilon D_{\epsilon}\omega^{(m)})\partial_t\omega^{(m+1)}
+\epsilon(\nabla\tilde{l}_iD_{\epsilon}\partial_y\omega^{(m)})^{\top}\partial^2_{t}\omega^{(m+1)}\big|
\lesssim \epsilon\sum^n_{j=1}|\bar{q}^{(m+1)}_j|+\epsilon^2\mathcal{W}_2.\label{Eq:2.59}
\end{align}
Collecting \eqref{Eq:2.93}, \eqref{Eq:2.58} and \eqref{Eq:2.59} yields
\begin{equation*}
|\bar{q}^{(m+1)}_i|+|\bar{z}^{(m+1)}_i|+|E^{(m+1)}_i|\leq C\epsilon \tilde{\mathcal{Q}}_2+
|\bar{q}^{(m+1)}_i(y, t_0)|+|\bar{z}^{(m+1)}_i(y, t_0)|+|E^{(m+1)}_i(y, t_0)|.
\end{equation*}
This means
\[
\max_{y\in [a_0, b_0]\atop t\in [t_0, t_1]}\sum^n_{j=1}\Big(|\bar{q}^{(m+1)}_j|+|\bar{z}^{(m+1)}_j|+|E^{(m+1)}_j|\Big)\leq C\epsilon\tilde{\mathcal{W}}_2+C\mathcal{W}^{i0}_2,
\]
where $\mathcal{W}^{i0}_2=\displaystyle\max_{y\in [a_0, b_0]}(|\bar{q}^{(m+1)}_i(y, t_0)|
+|\bar{z}^{(m+1)}_i(y, t_0)|+|E^{(m+1)}_i(y, t_0)|)$.

In conclusion, when $\epsilon>0$  is small, we can choose
\[
\mathcal{K}_2=\mathcal{K}^0_2+C\mathcal{H}^0_2, \;\;\mathcal{H}^0_2=C(\mathcal{H}^{i0}_2+\tilde{\mathcal{H}}_2),\; \mathcal{W}_2=C(\mathcal{W}^{0i}_2+\tilde{\mathcal{W}}_2),\; \mathcal{H}^i_2=C\mathcal{H}^{i0}_2,
\]
with $\tilde{\mathcal{H}}_2=C\tilde{\mathcal{H}}^0_2\leq C\epsilon$ such that \eqref{eq:2.44} holds for $m+1$.

Therefore, we complete the proof of this lemma.
\end{proof}

\vspace{0.2cm}
\underline{\textbf{Step $4$. The boundedness of $\nabla^3 K^{(m+1)}$, $\nabla^3 h^{(m+1)}$ and
$\nabla^3\omega^{(m+1)}$}}

\vspace{0.2cm}

\begin{lemma}\label{L:2.12}
It holds that
\begin{equation}\label{eq:2.65**}
|\nabla^3 K^{(m)}|\leq \mathcal{K}_3,\; |\nabla^3 h^{(m)}|\leq \mathcal{H}^0_3, \; |\nabla^3 \omega^{(m)}|\leq \mathcal{W}_3.
\end{equation}
In particular, $|\partial^3_t K^{(m)}|\leq C\epsilon$, $|\partial^3_t h^{(m)}|\leq C\epsilon$, $|\nabla^3 h^{(m)}_i|\leq \mathcal{H}^i_3$,  $|\nabla^3 \tilde{h}^{(m)}|\leq \tilde{\mathcal{H}}_3\leq C\epsilon$,  $\mathcal{K}_3$, $\mathcal{H}^0_3$, $\mathcal{H}^i_3$, $\tilde{\mathcal{H}}_3$ and $\mathcal{W}_3$ are
positive constants determined later.
\end{lemma}

\begin{proof}
We need to show that \eqref{eq:2.65**} still holds valid for $m+1$. First of all,
we establish the boundedness of $\partial^3_yK^{(m+1)}$, $\partial^3_{yyt}K^{(m+1)}$, $\partial^3_{ytt} K^{(m+1)}$
and $\partial^3_t K^{(m+1)}$.

It follows from \eqref{eq:2.12} and direct computations that
\begin{equation}\label{eq:2.66}
|\partial_t\partial^3_y K^{(m+1)}|\leq C\epsilon |\partial^3_y h^{(m)}|
+C\epsilon^2(|\partial^3_y\omega^{(m)}|+|\partial^3_y K^{(m)}|)+C\epsilon^2.
\end{equation}
Similarly, one has
\[
\begin{split}
&|\partial^4_{t}K^{(m+1)}|\leq C\epsilon |\partial^3_{t}h^{(m)}|+C\epsilon^2 (|\partial^3_{t}\omega^{(m)}|+|\partial^3_{t}K^{(m)}|)+C\epsilon^3,\\[5pt]
&|\partial_t\partial^3_{yyt} K^{(m+1)}|\leq C\epsilon |\partial^3_{yyt}h^{(m)}|+C\epsilon^2(|\partial^3_{yyt}\omega^{(m)}|+|\partial^3_{yyt}K^{(m)}|)+C\epsilon^2,\\[5pt]
&|\partial_t\partial^3_{tty} K^{(m+1)}|\leq C\epsilon |\partial^3_{tty}h^{(m)}|+C\epsilon^2(|\partial^3_{tty}\omega^{(m)}|+|\partial^3_{tty}K^{(m)}|)+C\epsilon^2.
\end{split}
\]
In conclusion, we have
\begin{equation*}
\max_{y\in [a_0, b_0]\atop t\in [t_0, t_1]}|\nabla^3 K^{(m+1)}|\leq \mathcal{K}^0_3+C \mathcal{H}^0_3,\; \;\;\text{ with }\; \mathcal{K}^0_3:=\max_{y\in [a_0, b_0]}|\nabla^3 K^{(m+1)}(y, t_0)|.
\end{equation*}
In particular, it holds
\begin{equation}
|\partial^3_{t} K^{(m+1)}|\leq |\partial^3_{t} K^{(m+1)}(y, t_0)|+C|\partial^3_{t} h^{(m)}|\leq C\epsilon.
\end{equation}

Note that
\begin{equation}\label{eq:2.65}
\begin{split}
|\partial_t\partial^3_y h^{(m+1)}_i|\leq C\epsilon|\partial^3_{y}\tilde{h}^{(m)}|+C\epsilon^2|\partial^3_{y} h^{(m)}|+C\epsilon^2.
\end{split}
\end{equation}
Then this derives
\[
|\partial^3_y h^{(m+1)}_i|\leq |\partial^3_{y}h^{(m+1)}_i(y, t_0)|+C|\partial^3_{y}\tilde{h}^{(m)}|.
\]
Analogously,
\[
\begin{split}
&|\partial^3_{t} h^{(m+1)}_i|\leq |\partial^3_{t}h^{(m+1)}_i(y, t_0)|+C|\partial^3_{t}\tilde{h}^{(m)}|,\\[5pt]
&|\partial^3_{tty} h^{(m+1)}_i|\leq |\partial^3_{tty}h^{(m+1)}_i(y, t_0)|+C|\partial^3_{tty}\tilde{h}^{(m)}|,\\[5pt]
&|\partial^3_{tyy} h^{(m+1)}_i|\leq |\partial^3_{tyy}h^{(m+1)}_i(y, t_0)|+C|\partial^3_{tyy}\tilde{h}^{(m)}|.
\end{split}
\]
Therefore, it holds that
\begin{equation*}
\max_{y\in [a_0, b_0], t\in [t_0, t_1]}|\nabla^3 h^{(m+1)}_i(y, t)|\leq \mathcal{H}^{i0}_3+C\tilde{\mathcal{H}}_3\;\;\text{ with } \mathcal{H}^{i0}_3=\max_{y\in [a_0, b_0]}|\nabla^3 h^{(m+1)}_i(y, t_0)|.
\end{equation*}

In the following, we derive the estimate of  $\nabla^3 h^{(m+1)}_j$ for $j\neq i$. It only suffices to
derive the boundedness of $\partial^3_t h^{(m+1)}_j$. Analogous to \eqref{eq:2.50*}, we have
\[
\begin{split}
&\Big|\big(K^{(m)}\partial_t+(\lambda_j-\lambda_i)(\epsilon D_\epsilon\omega^{(m)})\partial_y\big)\partial^3_th^{(m+1)}_j+3\partial_tK^{(m)}\partial^3_th^{(m+1)}_j\Big|\\
\leq &C\epsilon^2+C\epsilon|\partial^3_t K^{(m)}|+C\epsilon|\partial^3_t\tilde{h}^{(m+1)}|+C\epsilon^2|\partial^3_t h^{(m)}|+C\epsilon^2|\partial^3_t\omega^{(m)}|,
\end{split}
\]
where the fact $|\partial^3_{ytt}h^{(m+1)}_j|\leq C|\partial^3_t h^{(m+1)}_j|+C\epsilon^2$ is used. Denote
\[
\tilde{\mathcal{H}}^0_3=\max\Big(\max_{y\in[a_0, b_0]}\sum_{j\neq i}|\partial^3_t h^{(m+1)}_j(y, t_0)|, C\epsilon\Big),
\]
then
\[
\begin{split}
\sum_{j\neq i}|\partial^3_t h^{(m+1)}_j|\leq \tilde{\mathcal{H}}^0_3
+C\epsilon\sum_{j\neq i}|\partial^3_t h^{(m+1)}_j|+C\epsilon^2(\mathcal{H}^0_3+\mathcal{W}_3)+
C\epsilon^2.
\end{split}
\]
This derives
\[
\max_{y\in [a_0, b_0], \atop t\in [t_0, t_1]}\sum_{j\neq i}|\partial^3_t h^{(m+1)}_j|\leq C\tilde{\mathcal{H}}^0_3,
\quad \max_{y\in [a_0, b_0], \atop t\in [t_0, t_1]}\sum_{j\neq i}|\partial^3_{ytt} h^{(m+1)}_j|\leq C\tilde{\mathcal{H}}^0_3.
\]
On the other hand, based on \eqref{eq:2.14}, we can deduce that
\[
|\partial^3_{yyt} h^{(m+1)}_j|\leq C|\partial^3_{tty}h^{(m+1)}_j|+C\epsilon, \quad
|\partial^3_{y} h^{(m+1)}_j|\leq C|\partial^3_{yyt}h^{(m+1)}_j|+C\epsilon.
\]
Therefore
\begin{equation*}
\max_{y\in [a_0, b_0], t\in [t_0, t_1]}|\nabla^3 h^{(m+1)}_j|\leq C\tilde{\mathcal{H}}^0_3.
\end{equation*}

Next, we deal with $\nabla^3\omega^{(m+1)}_j$ for $j\neq i$.
For convenience,  some notations are introduced
\[
\begin{split}
&F^{(m+1)}_j:=\tilde{l}_j(\epsilon D_{\epsilon}\omega^{(m)})\partial^3_y\omega^{(m+1)},\;\;
G^{(m+1)}_j:=\tilde{l}_j(\epsilon D_{\epsilon}\omega^{(m)})\partial^3_{tyy}\omega^{(m+1)},\\
&J^{(m+1)}_j:=\tilde{l}_j(\epsilon D_{\epsilon}\omega^{(m)})\partial^3_{ytt}\omega^{(m+1)},\;\;
L^{(m+1)}_j:=\tilde{l}_j(\epsilon D_{\epsilon}\omega^{(m)})\partial^3_{t}\omega^{(m+1)}.
\end{split}
\]

Differentiating \eqref{eq:2.16} with respect to $t$ three times and taking direct computations, it holds that for $j\neq i$,
\begin{equation}\label{eq:2.59}
\begin{split}
&(\lambda_j-\lambda_i)J^{(m+1)}_j+K^{(m)}L^{(m+1)}_j+\partial^2_tK^{(m)}\bar{h}^{(m+1)}_j+2\epsilon
\nabla(\lambda_j-\lambda_i)D_\epsilon\partial_t\omega^{(m)}\bar{z}^{(m+1)}_j+2\partial_tK^{(m)}\bar{q}^{(m+1)}_j\\
&+\partial^2_t\tilde{l}_j(K^{(m)}\partial_t\omega^{(m+1)}+(\lambda_j-\lambda_i)^{(m)}\partial_y\omega^{(m+1)})
+\partial^2_t(\lambda_j-\lambda_i)^{(m)}\tilde{l}_j(\epsilon D_\epsilon\omega^{(m)})\partial_y\omega^{(m+1)}\\
&+2\epsilon(\nabla\tilde{l}_jD_{\epsilon}\partial_t\omega^{(m)})^{\top}
\partial_t\Big(K^{(m)}\partial_t\omega^{(m+1)}+(\lambda_j-\lambda_i)(\epsilon D_\epsilon\omega^{(m)})\partial_y\omega^{(m+1)}\Big)=0,
\end{split}
\end{equation}
and
\begin{equation}\label{eq:2.74}
\begin{split}
&|\Big(K^{(m)}\partial_t+(\lambda_j-\lambda_i)(\epsilon D_{\epsilon}\omega^{(m)})\partial_y\Big)L^{(m+1)}_j+3\partial_tK^{(m)}L^{(m+1)}_j|\\
\leq& C\epsilon^2 \mathcal{W}_3+C\epsilon^2+C\epsilon^2\sum_{j\neq i}(|L^{(m+1)}_j|+|J^{(m+1)}_j|).
\end{split}
\end{equation}
From \eqref{eq:2.59}, one can get
\begin{equation*}
|J^{(m+1)}_j|\leq C|L^{(m+1)}_j|+C\epsilon^2, \text{ for } j\neq i.
\end{equation*}
Let
\[
\tilde{\mathcal{Q}}_3:=\max\Big(\max_{y\in[a_0, b_0]}\sum_{j\neq i}|L^{(m+1)}_j(y, t_0)|, C\epsilon\Big).
\]
Then it follows from \eqref{eq:2.74} that
\begin{equation*}
\sum_{j\neq i}|L^{(m+1)}_j|\leq \tilde{\mathcal{Q}}_3+C\epsilon\sum_{j\neq i}|L^{(m+1)}_j|
+C\epsilon^2\mathcal{W}_3+C\epsilon^2,
\end{equation*}
and further
\begin{equation*}
\sum_{j\neq i}|L^{(m+1)}_j|\leq C\tilde{\mathcal{Q}}_3\leq C\epsilon\tilde{\mathcal{W}}_3,\quad \sum_{j\neq i}|J^{(m+1)}_j|\leq C\epsilon\tilde{\mathcal{W}}_3,
\end{equation*}
where $\tilde{\mathcal{W}}_3:=\displaystyle\max_{y\in[a_0, b_0]}\sum_{j\neq i}|\partial^3_t\omega^{(m+1)}_j(y, t_0)|$.

In addition, in terms of \eqref{eq:2.16}, one has
$|G^{(m+1)}_j|\leq C|J^{(m+1)}_j|+C\epsilon$.
Analogously, it is derived from \eqref{eq:2.16} that
$|F^{(m+1)}_j|\leq C|G^{(m+1)}_j|+C\epsilon$.
Therefore,
\begin{equation}\label{Eq:2.113}
\max_{y\in [a_0, b_0],\atop t\in [t_0, t_1]}\sum_{j\neq i}\Big(|L^{(m+1)}_j|+|J^{(m+1)}_j|+|F^{(m+1)}_j|+|G^{(m+1)}_j|\Big)\leq C\epsilon\tilde{\mathcal{W}}_3.
\end{equation}
Meanwhile, we can derive the estimate of $F^{(m+1)}_i$ from \eqref{eq:2.15} that
\begin{equation}\label{eq:2.75}
\begin{split}
|\partial_t F^{(m+1)}_i|=&|\epsilon(\nabla\tilde{l}_i^{(m)}D_{\epsilon}\partial_t\omega^{(m)})^{\top}\partial^3_y\omega^{(m+1)}-
\partial^3_{y}\tilde{l}_i^{(m)}
\partial_t\omega^{(m+1)}-3\partial^2_y\tilde{l}^{(m)}_i\partial^2_{yt}\omega^{(m+1)}
\\
&\quad-3\epsilon(\nabla\tilde{l}_i D_\epsilon\partial_y\omega^{(m)})^{\top}\partial^3_{yyt}\omega^{(m+1)}|\\
\leq& C\epsilon^2\sum_{k=1}^n|F^{(m+1)}_k|+C\epsilon\sum_{k=1}^n|G^{(m+1)}_k|+C\epsilon^2\mathcal{W}_3+C\epsilon^2.
\end{split}
\end{equation}
Based on \eqref{eq:2.15}, we can directly obtain that
\begin{equation}\label{eq:2.81}
\begin{split}
|\partial_t G^{(m+1)}_i|
\leq& C\epsilon\sum^n_{k=1}|J^{(m+1)}_k|+C\epsilon^2\mathcal{W}_3+C\epsilon^2, \\
|\partial_t J^{(m+1)}_i|
\leq& C\epsilon^2\sum^n_{k=1}|J^{(m+1)}_k|+C\epsilon\sum^n_{k=1}|L^{(m+1)}_k|+C\epsilon^2\mathcal{W}_3,\\
|\partial_t L^{(m+1)}_i|
\leq& C\epsilon^2\sum_{k=1}^n|L^{(m+1)}_k|+C\epsilon^2\mathcal{W}_3+C\epsilon^2.
\end{split}
\end{equation}
Thus, it follows from \eqref{eq:2.81} that
\begin{equation}\label{eq:2.85}
\begin{split}
&|F^{(m+1)}_i|+|G^{(m+1)}_i|+|L^{(m+1)}_i|+|J^{(m+1)}_i|\\
\leq& C\Big(|F^{(m+1)}_{i}(y, t_0)|+|G^{(m+1)}_{i}(y, t_0)|+|L^{(m+1)}_{i}(y, t_0)|+|J^{(m+1)}_{i}(y, t_0)|\Big)+C\epsilon
\tilde{\mathcal{W}}_3.
\end{split}
\end{equation}
Collecting \eqref{Eq:2.113} and \eqref{eq:2.85} yields
\begin{equation}\label{Eq:2.118}
\max_{y\in [a_0, b_0]\atop t\in [t_0, t_1]}\sum^n_{k=1}\Big(|F^{(m+1)}_k|+|G^{(m+1)}_k|+|L^{(m+1)}_k|+|J^{(m+1)}_k|\Big)\leq C\epsilon \tilde{\mathcal{W}}_3+\mathcal{W}^{i0}_3,
\end{equation}
where
\[
\mathcal{W}^{i0}_3:=\max_{y\in[a_0, b_0]}\Big(|F^{(m+1)}_{i}(y, t_0)|+|G^{(m+1)}_{i}(y, t_0)|+|L^{(m+1)}_{i}(y, t_0)|+|J^{(m+1)}_{i}(y, t_0)|\Big).
\]

In conclusion, for small $\epsilon>0$, we can choose
\[
\mathcal{K}_3=\mathcal{K}^0_3+C\mathcal{H}^0_3, \;\; \mathcal{H}^0_3=C(\mathcal{H}^{i0}_3+\tilde{\mathcal{H}}^0_3), \;\; \mathcal{W}_3=
C(\tilde{\mathcal{W}}_3+\mathcal{W}^{i0}_3),\; \mathcal{H}^i_3=\mathcal{H}^{i0}_3+C\tilde{\mathcal{H}}_3,
\]
with $\tilde{\mathcal{H}}_3=C\tilde{\mathcal{H}}^0_3$ such that the estimates in \eqref{eq:2.65**} hold for $m+1$.
\end{proof}

\vspace{0.2cm}
\underline{\textbf{Step $5$. The convergence of the approximate solutions}}

\vspace{0.2cm}
By the uniform boundedness of the approximate solutions $(K^{(m)},  h^{(m)}_i, \tilde{h}^{(m)}, \omega^{(m)})$
established in Step 1-Step 4,  we start to show the uniform convergence of  $(K^{(m)},  h^{(m)}_i, \tilde{h}^{(m)}, \omega^{(m)})$
in $D$. In this case, if we set $(K,  h_i, \tilde{h}, \omega)=\ds\lim_{m\to\infty}(K^{(m)},  h^{(m)}_i, \tilde{h}^{(m)},
\omega^{(m)})$ in $C^2(\bar D)$, then  $(K,  h_i, \tilde{h}, \omega)$ is a classical solution
to problem \eqref{eq:2.98}.

At first, by an analogous argument in Step 1, one can obtain that there exists a uniform constant $C>0$ such that
for $(y,t)\in D$ and all $m\in\Bbb N^+$
\begin{equation}\label{YH-4}
|K^{(m)}-K^{(0)}|\le C\ve.
\end{equation}

Note that due to $T_{\ve}+1-t_0\sim \f{1}{\ve}$, we then have from \eqref{eq:2.12}-\eqref{eq:2.13}
that by the direct integrals on the time $t$,
\begin{equation}\label{YH-5}
|K^{(m+1)}-K^{(m)}|\le \bar M_0|h^{(m)}-h^{(m-1)}|+\text{``contractible terms"}
\end{equation}
and
\begin{equation}\label{YH-6}
|h^{(m+1)}-h^{(m)}|\le \bar M_0|h^{(m)}-h^{(m-1)}|+\text{``contractible terms"},
\end{equation}
where $\bar{M}_0$ is a positive constant independent of $\epsilon$.
This will arise the  difficulty for us to show the Cauchy sequence property of
$(K^{(m)},  h^{(m)}_i, \tilde{h}^{(m)}, \omega^{(m)})$
in $D$ (since it is unknown whether the constant $\bar M_0<1$ in \eqref{YH-5}-\eqref{YH-6} holds or not).
In order to overcome this difficulty, our strategy is to
divide  the time interval $[t_0, T_\epsilon+1]$ into $N$ subintervals as
\begin{equation*}
\begin{split}
&I_1=[t_0, t_0+\frac{T_\epsilon+1-t_0}{N}],\; I_2=[t_0+\frac{T_\epsilon+1-t_0}{N}, t_0+\frac{2(T_\epsilon+1-t_0)}{N}],\;\cdots,\\ &I_N=[t_0+\frac{N-1}{N}(T_\epsilon+1-t_0), T_\epsilon+1],
\end{split}
\end{equation*}
where $N<\f{T_{\ve}+1-t_0}{2}$ is a suitably large integer independent of $\ve$, and prove that  $(K^{(m)}, h^{(m)}_i, \tilde{h}^{(m)},  \omega^{(m)})$
is a Cauchy sequence in any subinterval $I_k$ ($1\le k\le N$) by utilizing the length $|I_k|\sim \f{1}{N\ve}$
and replacing $\bar M_0$ in \eqref{YH-5}-\eqref{YH-6} by the constant $\f{\bar M_0}{N}$ ($\f{\bar M_0}{N}<1$ holds due to
the largeness of $N$).

\vspace{0.2cm}
For $t\in I_1$, set
\[
\mathcal{K}=K^{(m+1)}-K^{(m)},\;\mathcal{I}=h^{(m+1)}_i-h^{(m)}_i,\; \mathcal{J}=h^{(m+1)}_j-h^{(m)}_j,\;\tilde{\mathcal{J}}=\sum_{j\neq i}(h^{(m+1)}_j-h^{(m)}_j).
\]
Then it is derived from \eqref{eq:2.12} that
\[
\begin{split}
\begin{cases}
\partial_t\mathcal{K}=\epsilon L^{(1)}(\epsilon D_\epsilon\omega^{(m)})\tilde{h}^{(m)}K^{(m)}-\epsilon L^{(1)}(\epsilon D_\epsilon\omega^{(m-1)})\tilde{h}^{(m-1)}K^{(m-1)}+\epsilon \nabla\lambda_i\cdot r_i(\epsilon D_\epsilon\omega^{(m)})h^{(m)}_i\\[3pt]
\quad\qquad-\epsilon \nabla\lambda_i\cdot r_i(\epsilon D_\epsilon\omega^{(m-1)})h^{(m-1)}_i,\\[3pt]
\mathcal{K}(y, t_0)=0.
\end{cases}
\end{split}
\]
Integrating with respect to $t$ in the interval $I_1$ yields
\begin{equation}\label{YY:2.99}
|\mathcal{K}|\leq C\epsilon|K^{(m)}-K^{(m-1)}|+C\epsilon^2\sum_{j\neq i}|\omega^{(m)}_j-\omega^{(m-1)}_j|
+C\epsilon|\omega^{(m)}_i-\omega^{(m-1)}_i|+
\frac{C}{N}|h^{(m)}-h^{(m-1)}|.
\end{equation}
Similarly, $\mathcal{I}$ satisfies that
\[
\begin{split}
\begin{cases}
\partial_t\mathcal{I}=\epsilon (\tilde{h}^{(m)})^{\top}Q^{(1)}(\epsilon D_\epsilon\omega^{(m)})\tilde{h}^{(m)}K^{(m)}+\epsilon L^{(2)}(\epsilon D_\epsilon\omega^{(m)})\tilde{h}^{(m)}h^{(m)}_i\\[3pt]
\qquad\quad\;-\epsilon (\tilde{h}^{(m-1)})^{\top}Q^{(1)}(\epsilon D_\epsilon\omega^{(m-1)})\tilde{h}^{(m-1)}K^{(m-1)}
-\epsilon L^{(2)}(\epsilon D_\epsilon\omega^{(m-1)})\tilde{h}^{(m-1)}h^{(m-1)}_i,\\[3pt]
\mathcal{I}(y,  t_0)=0.
\end{cases}
\end{split}
\]
Thus, we have
\begin{equation}\label{Y:2.98}
\begin{split}
|\mathcal{I}|\leq& \frac{C}{N}|\tilde{h}^{(m)}-\tilde{h}^{(m-1)}|+C\epsilon|h^{(m)}_i-h^{(m-1)}_i|+C\epsilon^2|K^{(m)}-K^{(m-1)}|\\[3pt]
&+C\epsilon^3\sum_{j\neq i}|\omega^{(m)}_j-\omega^{(m-1)}_j|+C\epsilon^2|\omega^{(m)}_i-\omega^{(m-1)}_i|.
\end{split}
\end{equation}
In addition, it follows from  \eqref{eq:2.14} and direct computation that
\begin{equation}\label{Y:2.99}
\begin{split}
&|\Big(K^{(m)}\partial_t+(\lambda_j-\lambda_i)(\epsilon D_\epsilon\omega^{(m)})\partial_y\Big)\mathcal{J}|\\
\leq& C\epsilon |\tilde{\mathcal{J}}|+C\epsilon|K^{(m)}-K^{(m-1)}|+C\epsilon^3\sum_{j\neq i}|\omega^{(m)}_j
-\omega^{(m-1)}_j|\\
&\;+C\epsilon^2|\omega^{(m)}_i-\omega^{(m-1)}_i|
+C\epsilon^2|\tilde{h}^{(m)}-\tilde{h}^{(m-1)}|+C\epsilon^2|h^{(m)}_i-h^{(m-1)}_i|.
\end{split}
\end{equation}
Together with $K^{(m)}|_{t=t_0}>0$ and the initial-boundary conditions for $t\in I_1$
\[
\begin{cases}
\mathcal{J}(a_0, t)=0, \quad j=i+1, \cdots, n,\\[3pt]
\mathcal{J}(y,  t_0)=0,
\end{cases}
\]
or
\[
\begin{cases}
\mathcal{J}(b_0, t)=0, \quad j=1, \cdots, i-1,\\[3pt]
\mathcal{J}(y,  t_0)=0,
\end{cases}
\]
we have from \eqref{Y:2.99} and the characteristics method that
\begin{equation}\label{Y:2.100}
\begin{split}
|\mathcal{\tilde{J}}|
\leq& C\epsilon|K^{(m)}-K^{(m-1)}|+C\epsilon^3\sum_{j\neq i}|\omega^{(m)}_j
-\omega^{(m-1)}_j|+C\epsilon^2|\omega^{(m)}_i-\omega^{(m-1)}_i|+C\epsilon^2|h^{(m)}-h^{(m-1)}|.
\end{split}
\end{equation}
\vspace{0.1cm}
Next, we show that $\omega^{(m)}$ is a Cauchy sequence in $I_1$, which is equivalent to prove
the Cauchy sequence property of $p^{(m)}$. Denote
\[
\mathcal{P}_i=p^{(m+1)}_i-p^{(m)}_i,\; \mathcal{P}_j=p^{(m+1)}_j-p^{(m)}_j, \; \mathcal{P}=p^{(m+1)}-p^{(m)}.
\]
From \eqref{eq:2.23}, one has
\[
\begin{cases}
\partial_t\mathcal{P}_i+\epsilon(\tilde{h}^{(m)})^{\top}M_1(\epsilon D_\epsilon\omega^{(m)})\mathcal{P}+
\epsilon((\tilde{h}^{(m)})^{\top}-(\tilde{h}^{(m-1)})^{\top})M_1(\epsilon D_\epsilon\omega^{(m)})p^{(m)}\\
\quad\;\;\;+\epsilon(\tilde{h}^{(m-1)})^{\top}(M_1(\epsilon D_\epsilon\omega^{(m)})
-M_1(\epsilon D_\epsilon\omega^{(m-1)}))p^{(m)}=0,\\[5pt]
\mathcal{P}_i(y, t_0)=0.
\end{cases}
\]
Then this yields that in $I_1$,
\begin{equation}\label{Y:2.101}
|\mathcal{P}_i|\leq \frac{C}{N}|\tilde{h}^{(m)}-\tilde{h}^{(m-1)}|+C\epsilon|\mathcal{P}|
+C\epsilon^3\sum_{j\neq i}|\omega^{(m)}_j-\omega^{(m-1)}_j|+C\epsilon^2|\omega^{(m)}_i-\omega^{(m-1)}_i|.
\end{equation}
On the other hand, by \eqref{eq:2.24},  we have
\begin{equation}
\begin{split}
&|\Big(K^{(m)}\partial_t+(\lambda_j-\lambda_i)(\epsilon D_\epsilon\omega^{(m)})\partial_y\Big)\mathcal{P}_j|
\leq C\epsilon|\mathcal{P}|+C\epsilon|K^{(m)}-K^{(m-1)}|+C\epsilon^3\sum_{j\neq i}|\omega^{(m)}_j-\omega^{(m-1)}_j|\\
&\qquad +C\epsilon^2|\omega^{(m)}_i-\omega^{(m-1)}_i|
+C\epsilon|\tilde{h}^{(m)}-\tilde{h}^{(m-1)}|+C\epsilon|h^{(m)}_i-h^{(m-1)}_i|.
\end{split}
\end{equation}
Analogous to the estimate of $\mathcal{J}$, one can obtain
\begin{equation}\label{Y:2.103}
\begin{split}
\sum_{j\neq i}|\mathcal{P}_j|\leq &C\epsilon|K^{(m)}-K^{(m-1)}|+C\epsilon|h^{(m)}-h^{(m-1)}|+C\epsilon|\mathcal{P}|\\
&+C\epsilon^3\sum_{j\neq i}|\omega^{(m)}_j-\omega^{(m-1)}_j|+
C\epsilon^2|\omega^{(m)}_i-\omega^{(m-1)}_i|.
\end{split}
\end{equation}
Due to $\omega^{(m+1)}=\displaystyle\sum^n_{j=1}p^{(m+1)}_j\tilde{r}_j(\epsilon D_\epsilon\omega^{(m)})$, then for $k\neq i$,
\begin{equation}\label{Y:2.107}
\begin{split}
\epsilon|\omega^{(m+1)}_k-\omega^{(m)}_k|
\leq &C\sum_{j\neq i}|\mathcal{P}_j|+C\epsilon^2\sum_{j\neq i}|\omega^{(m)}_j-\omega^{(m-1)}_j|
+C\epsilon|\omega^{(m)}_i-\omega^{(m-1)}_i|\\
\leq &C\epsilon|K^{(m)}-K^{(m-1)}|+C\epsilon|h^{(m)}-h^{(m-1)}|+C\epsilon|\mathcal{P}|\\
&+C\epsilon^2\sum_{j\neq i}|\omega^{(m)}_j-\omega^{(m-1)}_j|+
C\epsilon|\omega^{(m)}_i-\omega^{(m-1)}_i|
\end{split}
\end{equation}
and
\begin{equation}\label{Y:2.108}
\begin{split}
|\omega^{(m+1)}_i-\omega^{(m)}_i|\leq & |\mathcal{P}|+C\epsilon^2\sum_{j\neq i}|\omega^{(m)}_j-\omega^{(m-1)}_j|+\epsilon|\omega^{(m)}_i-\omega^{(m-1)}_i|\\
\leq &C\epsilon|K^{(m)}-K^{(m-1)}|+\frac{C}{N}|\tilde{h}^{(m)}-\tilde{h}^{(m-1)}|+C\epsilon|h^{(m)}_i-h^{(m-1)}_i|+C\epsilon|\mathcal{P}|\\
&+C\epsilon^2\sum_{j\neq i}|\omega^{(m)}_j-\omega^{(m-1)}_j|+
C\epsilon|\omega^{(m)}_i-\omega^{(m-1)}_i|.
\end{split}
\end{equation}
Collecting \eqref{YY:2.99}, \eqref{Y:2.98}, \eqref{Y:2.100}, \eqref{Y:2.101} and \eqref{Y:2.103}--\eqref{Y:2.108} yields
\begin{equation}\label{Y:2.2.105}
\begin{split}
&|\mathcal{K}|+|\mathcal{I}|+|\mathcal{\tilde{J}}|+|\mathcal{P}|+\epsilon\sum_{j\neq i}|\omega^{(m+1)}_j-\omega^{(m)}_j|+|\omega^{(m+1)}_i-\omega^{(m)}_i|\\
\leq &C\epsilon|K^{(m)}-K^{(m-1)}|+C(\epsilon+\frac{1}{N})|\tilde{h}^{(m)}-\tilde{h}^{(m-1)}|+C(\epsilon+\frac{1}{N})|h^{(m)}_i-
h^{(m-1)}_i|\\
&+C\epsilon|\mathcal{P}|+C\epsilon^2\sum_{j\neq i}|\omega^{(m)}_j-\omega^{(m-1)}_j|+C\epsilon|\omega^{(m)}_i-\omega^{(m-1)}_i|.
\end{split}
\end{equation}
Thus, provided that $\epsilon$ is small and $N$ is suitably large
such that $C(\epsilon+\frac{1}{N})<1$, $(K^{(m)}, h^{(m)}_i, \tilde{h}^{(m)},  \omega^{(m)})$
is a Cauchy sequence in  $I_1$. By the analogous idea, when
$(K^{(m)}, h^{(m)}_i, \tilde{h}^{(m)},  \omega^{(m)})$ is shown to be
a Cauchy sequence in  $I_k$ for $2\le k\le N-1$, we next show that $(K^{(m)}, h^{(m)}_i, \tilde{h}^{(m)},  \omega^{(m)})$
is a Cauchy sequence in  $I_N$.

\vspace{0.3cm}
By \eqref{YH-4} and the expression \eqref{YH-3}, it easy to know that
$K^{(m)}|_{t=t_0+\frac{N-1}{N}(T_\epsilon+1-t_0)}>0$ holds.
As before, for $t\in I_N$, set
\[
\mathcal{K}=K^{(m+1)}-K^{(m)},\;\mathcal{I}=h^{(m+1)}_i-h^{(m)}_i,\; \mathcal{J}=h^{(m+1)}_j-h^{(m)}_j,\;\tilde{\mathcal{J}}=\sum_{j\neq i}(h^{(m+1)}_j-h^{(m)}_j)
\]
and
\[
\mathcal{P}_i=p^{(m+1)}_i-p^{(m)}_i,\; \mathcal{P}_j=p^{(m+1)}_j-p^{(m)}_j, \; \mathcal{P}=p^{(m+1)}-p^{(m)}.
\]
Then we have that for  $t\in I_N$,
\begin{equation}\label{Y:2.97}
\begin{split}
|\mathcal{K}|\leq &|(K^{(m+1)}-K^{(m)})\Big(y, t_0+\frac{N-1}{N}(T_\epsilon+1-t_0)\Big)|+
\frac{C}{N}|h^{(m)}-h^{(m-1)}|\\
&+C\epsilon|K^{(m)}-K^{(m-1)}|+C\epsilon^2\sum_{j\neq i}|\omega^{(m)}_j-\omega^{(m-1)}_j|+C\epsilon|\omega^{(m)}_i-\omega^{(m-1)}_i|,
\end{split}
\end{equation}
where $K^{(m)}\Big(y, t_0+\frac{N-1}{N}(T_\epsilon+1-t_0)\Big)$ has been shown to be a Cauchy sequence.

Similarly, one has that for  $t\in I_N$,
\begin{equation}
\begin{split}
|\mathcal{I}|\leq& |(h^{(m+1)}_i-h^{(m)}_i)(y, t_0+\frac{N-1}{N}(T_\epsilon+1-t_0))|+\frac{C}{N}|\tilde{h}^{(m)}-\tilde{h}^{(m-1)}|+C\epsilon|h^{(m)}_i-h^{(m-1)}_i|\\[3pt]
&+C\epsilon^2|K^{(m)}-K^{(m-1)}|+C\epsilon^3\sum_{j\neq i}|\omega^{(m)}_j-\omega^{(m-1)}_j|+C\epsilon^2|\omega^{(m)}_i-\omega^{(m-1)}_i|,\\[3pt]
\end{split}
\end{equation}
where $h^{(m)}_i(y, t_0+\frac{N-1}{N}(T_\epsilon+1-t_0))$ is a Cauchy sequence.
Analogously, we arrive at
\begin{equation}\label{Y:2.101*}
\begin{split}
|\tilde{\mathcal{J}}|\leq& C\epsilon |\tilde{\mathcal{J}}|+C\epsilon|K^{(m)}-K^{(m-1)}|+C\epsilon^3\sum_{j\neq i}|\omega^{(m)}_j
-\omega^{(m-1)}_j|+C\epsilon^2|\omega^{(m)}_i-\omega^{(m-1)}_i|\\
&+C\epsilon^2|\tilde{h}^{(m)}-\tilde{h}^{(m-1)}|+C\epsilon^2|h^{(m)}_i-h^{(m-1)}_i|
\end{split}
\end{equation}
and
\begin{align}
|\mathcal{P}_i|\leq &|\mathcal{P}_i(y, t_0+\frac{N-1}{N}(T_\epsilon+1-t_0))|+\frac{C}{N}|\tilde{h}^{(m)}-\tilde{h}^{(m-1)}|
+C\epsilon|\mathcal{P}|
+C\epsilon^3\sum_{j\neq i}|\omega^{(m)}_j-\omega^{(m-1)}_j| \notag \\
&+C\epsilon^2|\omega^{(m)}_i-\omega^{(m-1)}_i| \notag \\
\leq&\frac{C}{N}|\tilde{h}^{(m)}-\tilde{h}^{(m-1)}|+C\epsilon|\mathcal{P}|
+C\epsilon^3\sum_{j\neq i}|\omega^{(m)}_j-\omega^{(m-1)}_j|+C\epsilon^2|\omega^{(m)}_i-\omega^{(m-1)}_i|,
\end{align}
\begin{align}
\sum_{j\neq i}|\mathcal{P}_j|\leq& C\epsilon|\mathcal{P}|+C\epsilon|K^{(m)}-K^{(m-1)}|+C\epsilon^3\sum_{j\neq i}|\omega^{(m)}_j-\omega^{(m-1)}_j|+C\epsilon^2|\omega^{(m)}_i-\omega^{(m-1)}_i| \label{Y:2.105} \\ \notag
&+C\epsilon|\tilde{h}^{(m)}-\tilde{h}^{(m-1)}|+C\epsilon|h^{(m)}_i-h^{(m-1)}_i|.
\end{align}
Thus, along with \eqref{Y:2.97}--\eqref{Y:2.105}, we can also get the same estimate \eqref{Y:2.2.105} for
$t\in I_N$.
Therefore, $(K,  h_i, \tilde{h}, \omega)=\ds\lim_{m\to\infty}(K^{(m)},  h^{(m)}_i, \tilde{h}^{(m)},
\omega^{(m)})$ holds in $C(\bar D)$. Together with the uniform boundedness of $\nabla^l K^{(m+1)}$, $\nabla^l h^{(m+1)}$ and
$\nabla^l\omega^{(m+1)}$  ($1\le l\le 3$) in domain $D$ and interpolation,
one easily knows $(K,  h_i, \tilde{h}, \omega)=\ds\lim_{m\to\infty}(K^{(m)},  h^{(m)}_i, \tilde{h}^{(m)},
\omega^{(m)})$ in $C^2(\bar D)$ and further $(K,  h_i, \tilde{h}, \omega)\in C^3(\bar D)$
can be derived.
Hence, the proof of Theorem \ref{T:2.1} is completed.

\subsection{Precise descriptions on the $i-$shock formation}

At first, we illustrate that near the blowup point $(x_\epsilon, T_\epsilon)$ of \eqref{eq:1.1}, the envelope of
the $i$-th characteristics family forms a cusp curve.

\begin{theorem}\label{T:2.3}
Under the assumptions \eqref{eq:1.2}, \eqref{eq:1.3} and \eqref{Eq:2.1},
there exists a unique point $(y_\epsilon, T_{\epsilon})$ for  the blowup system \eqref{eq:2.2}
such that
\begin{equation}\label{eq:2.100}
\partial_y\varphi(y_{\epsilon}, T_{\epsilon})=0,\;\;\partial^2_y\varphi(y_{\epsilon}, T_{\epsilon})=0,\;\;
\partial^3_y\varphi(y_{\epsilon}, T_{\epsilon})>0,\;\;\partial^2_{yt}\varphi(y_{\epsilon}, T_{\epsilon})<0.
\end{equation}

\end{theorem}

\begin{proof}
Let
\[
w=\epsilon D_\epsilon\omega, \; \tau=\epsilon t\;\;\text{and $w_0(x)=D_\epsilon\omega_0(x)$}.
\]
In addition, without loss of generality, $\lambda_i(0)=0$ is assumed
(otherwise, one can apply the translation $(t, x)\mapsto (t, x+\lambda_i(0)t)$ to achieve this).
Note that $\eqref{eq:2.88}_1$ can be reduced into
\[
\partial_{\tau}\omega+\epsilon^{-1}D^{-1}_\epsilon AD_{\epsilon}\partial_x\omega=0.
\]
This yields that for $\epsilon\to 0$,
\begin{equation}\label{eq:2.89}
\partial_{\tau}\omega_i+\partial_{w_i}\lambda_i(0)\omega_{i}\partial_x\omega_i=0.
\end{equation}
By $\eqref{eq:2.2}_2$, one has that for $\epsilon=0$,
\[
\omega_i(x, \tau)=w^{i}_0(y), \quad x=\varphi(y, \tau)=y+\partial_{w_i}\lambda_i(0)w^{i}_0(y)\tau.
\]
Then for $\epsilon=0$,
\[
K=\partial_y\varphi(y, \tau)=1+\partial_{w_i}\lambda_i(0)\Big(w^{i}_0(y)\Big)'\tau.
\]
Note that for $\tau>0$
\[
\partial_\tau K(x_0, \tau)|_{\epsilon=0}<0,\;
\partial_y^2 K(x_0, \tau)|_{\epsilon=0}>0.
\]
On the other hand, for $\tau_0=\Big(\max(-\partial_{w_i}\lambda_i(0)(w^{i}_0(y))')\Big)^{-1}$,
$$K(x_0,\tau_0)|_{\epsilon=0}=0,\quad \partial_yK(x_0,\tau_0)|_{\epsilon=0}=0.$$
Therefore, from the implicit function theorem, there exists a unique point $p(\epsilon)=(y_\epsilon, \tau_{\epsilon})$
such that
\[
K(p(\epsilon))=0, \; \partial_y K(p(\epsilon))=0, \;\partial_\tau K(p(\epsilon))<0, \; \partial^2_{y} K(p(\epsilon))>0,\quad
\ds\lim_{\epsilon\to 0}(y_\epsilon,\tau_{\epsilon})=(y_0,\tau_0).
\]
This implies that for $T_\epsilon=\f{\tau_{\epsilon}}{\epsilon}$,
\[
\partial_{y}\varphi|_{(y_\epsilon, T_\epsilon)}=0, \;\; \partial^2_{y}\varphi|_{(y_\epsilon, T_\epsilon)}=0, \;\; \partial^2_{yt}\varphi|_{(y_\epsilon, T_\epsilon)}<0,\;\;
\partial^3_{y}\varphi|_{( y_\epsilon, T_\epsilon)}>0.
\]
Thus, the desired results in \eqref{eq:2.100} are obtained.
\end{proof}
\begin{remark}\label{Yin-1}
From Theorem \ref{T:2.3}, it is known that $(x_{\ve},T_{\ve})=(\vp(y_{\ve}, T_{\ve}),T_{\ve})$
is the unique blowup point of  \eqref{eq:1.1}, and the envelope of
the $i$-th characteristics family forms a cusp curve. This phenomenon is analogous to that in 1-D Burgers equation
(see \cite{CZ} and \cite{YinZ}).
\end{remark}

Finally, we state a more precise conclusion than Theorem \ref{DY-1}.
\begin{theorem}\label{T:2.14}
There admits a weak entropy solution to problem \eqref{eq:1.1} including an $i-$shock curve
$x=\phi(t)\in C^1[T_{\epsilon}, T_{\epsilon}+\delta_0]$ starting from the blowup point $(x_{\epsilon}, T_{\epsilon})$,
where $\delta_0>0$ is some fixed small constant.
Moreover, close to the point $(x_{\epsilon}, T_\epsilon)$, it holds that for the solution $w$ of \eqref{eq:2.88},
\[
\begin{split}
\phi(t)=&x_{\epsilon}+\lambda_i(w(x_{\epsilon}, T_{\epsilon}))(t-T_\epsilon)+ O(1)(t-T_\epsilon)^2,\\
w_i(x, t)=&w_i(x_{\epsilon}, T_\epsilon)+ O(1)\Big((t-T_\epsilon)^3+\big(x-x_\epsilon-\lambda_i(w(x_{\epsilon}, T_{\epsilon}))(t-T_\epsilon)\big)^2\Big)^{\frac16},\\
w_j(x, t)=&w_j( x_{\epsilon}, T_\epsilon)+ O(1)\Big((t-T_\epsilon)^3+\big(x-x_\epsilon-\lambda_i(w(x_{\epsilon}, T_{\epsilon}))(t-T_\epsilon)\big)^2\Big)^{\frac13}, \quad j\neq i.
\end{split}
\]

\end{theorem}
The proof of Theorem \ref{T:2.14} will be given in Sections 3-5 below.

\section{Analysis on the  pre-shock wave near the blowup point}

In this section, we investigate some properties of the solution $w$ to problem \eqref{eq:2.88}
and construct the first approximation to the resulting shock wave  of \eqref{eq:2.88} from the blowup point $(x_{\ve}, T_{\ve})$.
As illustrated in Remark \ref{Yin-1}, $(x_{\ve}, T_{\ve})=(\vp(y_{\epsilon}, T_{\epsilon}), T_{\ve})$
is just the unique blowup point at time $T_{\ve}$ for problem \eqref{eq:1.1} under the
assumptions \eqref{eq:1.2}, \eqref{eq:1.3} and \eqref{Eq:2.1}, moreover, \eqref{eq:2.100} holds.
In terms of  the unfolding theorem (see Theorem $2.1$ in \cite{Sch}),
there exist smooth functions $h(y, t), A(t)$ and $B(t)$ such that
\begin{equation}\label{eq:3.2}
\varphi(y, t)=h^3(y, t)-A(t)h(y, t)+B(t),
\end{equation}
where $\partial_y h(y_{\epsilon}, T_{\epsilon})>0$,  $A'(T_{\epsilon})>0$, and
\begin{equation}\label{Cao-1}
h(y_\epsilon, T_\epsilon)=A(T_\epsilon)=0,
\;\; \varphi(y_\epsilon, T_\epsilon)=B(T_\epsilon).
\end{equation}
Let
\[
\Sigma=\Big\{(y, t): \partial_y \varphi(y,t)=0, T_{\epsilon}\le t\le T_{\epsilon}+1\Big\}.
\]
Note that on $\Sigma$, one has
\begin{equation}\label{eq:3.5}
\partial^2_{yt}\varphi(y, t)\partial_y t+\partial^2_{y}\varphi(y, t)=0.
\end{equation}
Together with \eqref{eq:2.100}, this yields
\begin{equation}\label{Cao-3}
\partial_y t(y_{\epsilon}, T_{\epsilon})=0.
\end{equation}
Due to $\partial^2_{yt}\varphi(y_{\epsilon}, T_{\epsilon})<0$, then it follows from the implicit function theorem
that there exists a unique $C^2$ function $t=t(y)$ in the neighbourhood of $(y_{\epsilon}, T_{\epsilon})$
satisfying
\[
\partial_y\varphi(y, t(y))=0.
\]
Differentiating \eqref{eq:3.5} with respect to $y$ yields that
\[
2\partial^3_{yyt} \varphi \partial_y t
+\partial^3_{ytt}\varphi(\partial_y t)^2+
\partial_{yt}^2\varphi \partial_y^2 t+
\partial_y^3\varphi=0.
\]
Together with \eqref{eq:2.100}, we have $\partial_y^2 t(y_{\epsilon}, T_{\epsilon})>0$,
which means that $t=t(y)$ achieves the minimum value at the point $(y_{\epsilon}, T_{\epsilon})$.

Due to $\partial_y h(y_{\epsilon}, T_{\epsilon})>0$, then there exist two smooth functions $y=\eta^{\epsilon}_{\pm}(t)$ such that
\begin{equation}\label{Cao-2}
h(\eta^{\epsilon}_{\pm}(t), t)=\pm\sqrt{\frac{A(t)}{3}}, \quad y_{\epsilon}=\eta^{\epsilon}_{\pm}(T_{\epsilon}).
\end{equation}
In the following, we study the properties of  $x_{\pm}(t)=x(\eta_{\pm}(t), t)$ close to
$(x_{\ve}, T_{\ve})$ (see Figure 5). For simplicity, without loss of generality,  set
\begin{equation}\label{Cao-4}
\partial_y h(y_{\epsilon}, T_{\epsilon})=A'(T_{\epsilon})=1.
\end{equation}
It follows from the Taylor expansion formula, \eqref{Cao-1}, \eqref{Cao-3} and \eqref{Cao-4} that
\[
\begin{split}
h(\eta^{\epsilon}_{\pm}(t), t)=&h(y_{\epsilon}, T_{\epsilon})+(\partial_t h\partial_y t+\partial_y h)(y_{\epsilon}, T_{\epsilon})(\eta^{\epsilon}_{\pm}(t)-y_{\epsilon})+O(1)(\eta^{\epsilon}_{\pm}(t)-y_{\epsilon})^2\\
=&(\eta^{\epsilon}_{\pm}(t)-y_{\epsilon})+O(1)(\eta^{\epsilon}_{\pm}(t)-y_{\epsilon})^2.
\end{split}
\]
In addition, one has from \eqref{Cao-2},  \eqref{Cao-1} and \eqref{Cao-4} that
\[
h(\eta^{\epsilon}_{\pm}(t), t)=\pm\sqrt{\frac{t-T_{\epsilon}}{3}}+O(1)(t-T_{\epsilon}).
\]
On the other hand, differentiating \eqref{eq:3.2} with respect to $t$ yields that
\begin{equation*}\label{eq:3.4}
\partial_t \varphi=(3h^2(y, t)-A(t))\partial_t h(y, t)-A'(t)h+B'(t),
\end{equation*}
which means
\begin{equation}\label{eq:3.4}
\partial_t\varphi(y_{\epsilon}, T_{\epsilon})=B'(T_{\epsilon}).
\end{equation}
Then it holds
\begin{equation}\label{eq:3.6}
\begin{split}
&\eta^{\epsilon}_{\pm}(t)-y_{\epsilon}=\pm\sqrt{\frac{t-T_{\epsilon}}{3}}+O(1)(t-T_{\epsilon}),\\
&x_{\pm}(t)=h^3(\eta_{\pm}(t), t)-A(t)h(\eta_{\pm}(t), t)+B(t)\\
&\qquad=\mp\frac{2}{9}\sqrt{3}(t-T_{\epsilon})^{\frac32}+\varphi(y_{\epsilon}, T_{\epsilon})
+\partial_t \varphi(y_{\epsilon}, T_{\epsilon})(t-T_{\epsilon})+O(1)(t-T_{\epsilon})^2,
\end{split}
\end{equation}
where
\[
\begin{split}
A(t)&=(t-T_{\epsilon})+O(1)(t-T_{\epsilon})^2,\\
B(t)&=\varphi(y_{\epsilon}, T_{\epsilon})+\partial_t \varphi(y_{\epsilon}, T_{\epsilon})(t-T_{\epsilon})+O(1)(t-T_{\epsilon})^2.
\end{split}
\]
Therefore,
\[
x_{\pm}(t)-\varphi(y_{\epsilon}, T_{\epsilon})
-\partial_t\varphi(y_{\epsilon}, T_{\epsilon})(t-T_{\epsilon})\sim\mp\frac{2}{9}\sqrt{3}(t-T_{\epsilon})^{\frac32}.
\]
Here and below, for functions $f$ and $g$, $f\sim g$ represents $C_1|g|\leq |f|\leq C_2|g|$
for some positive constants $C_1$ and $C_2$ independent of $\epsilon$.

\begin{center}
\begin{tikzpicture}[scale=0.9]
\draw [thick][->] (-5, 0)--(5, 0);
\draw[red] (-4, 0)--(1.5, 5.35);
\draw [red](-2.8, 0)--(0.35, 4.1);
\draw [red](2, 0)--(-0.7, 4.5);
\draw [red](4, 0)--(-1.5, 5.18);
\draw [green][ultra thick](0,2.3)to [out=90, in=-30](-1.5, 5);
\draw [green][ultra thick](0,2.3)to [out=90, in=195](1.5, 5);
\draw[dashed](-5, 2.3)--(5,2.3);
\draw[dashed](-5, 3.81)--(5,3.81);
\draw (-0.06, 3.81)--(0.2, 0);
\node at (5.1,0) {$x$};
\node at (0.2, -0.3) {$y^{\epsilon}_0$};
\node at (-2.2, 4.8) {$x=x_{+}^{\epsilon}(t)$};
\node at (2.2, 4.8) {$x=x_{-}^{\epsilon}(t)$};
\node at (-4.8, 2.6) {$T_\epsilon$};
\node at (-4, -0.3) {$y^{\epsilon}_{-}(x, t)$};
\node at (4, -0.3) {$y^{\epsilon}_{+}(x, t)$};
\node at (-2.8, -0.3) {$\eta^{\epsilon}_{-}(t)$};
\node at (2.2, -0.3) {$\eta^{\epsilon}_{+}(t)$};
\node at (-4.8, 0.3) {$t$};
\node [below] at (0.5, -0.8){Figure 5. Real roots $y^{\epsilon}_{\pm}(x, t)$ and $y_0^{\ve}$ of the equation $x=\varphi(y, t)$};
\end{tikzpicture}
\end{center}

Next, we derive some properties on the real roots of the equation $x=\vp(y,t)$ with respect to $y$ (see Figure $5$).
\begin{lemma}
\begin{itemize}
For $t\in (T_{\epsilon}, T_{\epsilon}+1]$, it holds that
\begin{itemize}
\item[$(1)$] for $x\in(x^{\epsilon}_+(t),x^{\epsilon}_{-}(t))$, there exist three real roots $y_{-}^{\epsilon}(x, t)<y^{\epsilon}_0<y^{\epsilon}_+(x, t)$ to $x=\varphi(y, t)$.
\item[$(2)$] for  $x\geq x^{\epsilon}_{-}(t)$, there exists a unique real root $y^{\epsilon}_+(x, t)$ to $x=\varphi(y, t)$.
\item[$(3)$] for $x\leq x^{\epsilon}_+(t)$, there exists a unique real root $y^{\epsilon}_{-}(x, t)$ to $x=\varphi(y, t)$.
\end{itemize}
\end{itemize}
\end{lemma}
\begin{proof}
Let
\[
F(h)=h^3(y, t)-A(t)h(y, t)+B(t)-x.
\]
Then $F'(h)=3h^2-A(t)$, and $F(h)$ achieves its local maximum value at $h=-\sqrt{\frac{A(t)}{3}}$. Moreover,
\[
F(-\sqrt{\frac{A(t)}{3}})=-\frac23A(t)h+B(t)-x=\frac{2}{9}\sqrt{3}A(t)^{\frac32}+B(t)-x.
\]
Meanwhile, $F(h)$ also obtains its local minimum value at $h=\sqrt{\frac{A(t)}{3}}$, and
\[
F(\sqrt{\frac{A(t)}{3}})=-\frac23A(t)h+B(t)-x=-\frac{2}{9}\sqrt{3}A(t)^{\frac32}+B(t)-x.
\]
When $x^{\epsilon}_{+}(t)<x<x^{\epsilon}_{-}(t)$, we derive from $\eqref{eq:3.6}_2$ that
\[
-\frac29\sqrt{3}A(t)^{\frac32}=B(t)-x^{\epsilon}_{-}(t)<B(t)-x<B(t)-x^{\epsilon}_+(t)=\frac29\sqrt{3}A(t)^{\frac32},
\]
which implies $F(-\sqrt{\frac{A(t)}{3}})>0$ and $F(\sqrt{\frac{A(t)}{3}})<0$.
Therefore, there exist three real roots $y_{-}^{\epsilon}(x, t)<y^{\epsilon}_0<y^{\epsilon}_+(x, t)$ to the equation $x=\varphi(y, t)$.

When $x\geq x^{\epsilon}_{-}(t)$, $F(-\sqrt{\frac{A(t)}{3}})\leq 0$ holds.
Then there exists a unique solution $y^{\epsilon}_+(x, t)$  to $x=\varphi(y, t)$. Similarly,
when $x\leq x^{\epsilon}_{-}(t)$, one can have $F(\sqrt{\frac{A(t)}{3}})\geq 0$ and there is a
unique solution $y^{\epsilon}_{-}(x, t)$ to $x=\varphi(y, t)$.

\end{proof}

\subsection{The behavior of $y^{\epsilon}_{\pm}(x, t)$ in cusp domain}

In this subsection, we will describe the behavior of $y^{\epsilon}_{\pm}(x, t)$, which
is crucial to construct the first approximation of shock solution.

Denote
\[
\begin{split}
\Omega_+=&\{(x, t)\in \Omega: x>x^{\epsilon}_+(t),\; T_{\epsilon}<t\leq T_{\epsilon}+1\},\\
 \Omega_{-}=&\{(x, t)\in \Omega: x<x^{\epsilon}_{-}(t),\; T_{\epsilon}<t\leq T_{\epsilon}+1\},\\
 \Omega_0=&\{(x, t)\in \Omega: x^{\epsilon}_{+}(t)<x<x^{\epsilon}_{-}(t),\; T_{\epsilon}<t\leq T_{\epsilon}+1\}.
 \end{split}
 \]
In the cusp domain $\Omega_0$, each characteristics can be well-defined through starting from $(y^{\epsilon}_{\pm}(x, t), t)$,
respectively, see Figure $6$ below.

The $i-$th eigenvalue of $n\times n$ matrix $\Big(\int^1_0(\partial_{u_k}f_l)\Big(\theta u(v( y^{\epsilon}_+(x, t), t))+
 (1-\theta)u(v(y^{\epsilon}_{-}(x, t), t))\Big)d\theta\Big)_{k,l=1}^n$ is denoted by
 $\lambda_i\Big(\int^1_0(\partial_{u_k}f_l)\Big(\theta u(v(y^{\epsilon}_+(x, t), t))+
 (1-\theta)u(v(y^{\epsilon}_{-}(x, t), t))\Big)d\theta\Big)$, where $v(y,t)$ is the smooth solution of \eqref{eq:2.1}.
Let $x=\phi^0(t)\in C^{\infty}(T_{\epsilon}, T_{\epsilon}+1]$ satisfy
 \begin{equation}\label{YHC-7}
 \begin{cases}
 \begin{split}
 \frac{d\phi^0(t)}{dt}=&\lambda_i\Big(\int^1_0(\partial_{u_k}f_l)\Big(\theta u(v( y^{\epsilon}_+(\phi^0(t), t), t))+
 (1-\theta)u(v(y^{\epsilon}_{-}(\phi^0(t), t), t))\Big)d\theta\Big),\\
 \phi^0(T_{\epsilon})=&x_{\epsilon},
 \end{split}
 \end{cases}
 \end{equation}
 where $x^{\epsilon}_+(t)<\phi^0(t)<x^{\epsilon}_{-}(t)$, and
 \[
 \phi^0(t)=x_{\epsilon}+\lambda_i(w(x_{\epsilon}, T_{\epsilon}))(t-T_{\epsilon})+O(1)((t-T_{\epsilon})^2),
 \quad t\in [T_{\epsilon}, T_{\epsilon}+1].
 \]
$x=\phi^0(t)$ is called the pre-shock wave of \eqref{eq:1.1}, whose picture is roughly drawn in Figure $7$.
\begin{center}
\begin{tikzpicture}[scale=0.9]
\draw [thick][->] (-5, 0)--(5, 0);
\draw[red] (-4, 0)--(1.3, 5.2);
\draw [red](-3, 0)--(1, 5);
\draw [red](-2, 0)--(0.4, 4.2);
\draw [red](-1, 0)--(0.1, 3.3);
\draw [red](1, 0)--(-0.1, 3.3);
\draw [red](2, 0)--(-0.7, 4.5);
\draw [red](3, 0)--(-1.2, 5);
\draw [red](4, 0)--(-1.9, 5.5);
\draw [green][ultra thick](0,2.3)to [out=90, in=-30](-1.5, 5);
\draw [green][ultra thick](0,2.3)to [out=90, in=195](1.5, 5);
\draw (0, 0)--(0, 2.3);
\draw[dashed](-5, 2.3)--(5,2.3);
\node at (5.1,0) {$x$};
\node at (-2.3, 4.8) {$x=x_{+}^{\epsilon}(t)$};
\node at (2.3, 4.8) {$x=x_{-}^{\epsilon}(t)$};
\node at (-4.5, 2.5) {$t=T_\epsilon$};
\node at (-1.8, -0.2) {$y_{-}^{\epsilon}(x, t)$};
\node at (2.2, -0.2) {$y_{+}^{\epsilon}(x, t)$};
\node at (-4.5, 0.2) {$t_0$};
\node at (-1.2, 3.9){$\Omega_{-}$};
\node at (1.2, 3.9){$\Omega_+$};
\node at (0.1, 5){$\Omega_0$};
\node [below] at (0.5, -0.8){Figure $6$. Cusp domain $\Omega_0$};
\end{tikzpicture}
\end{center}

 \begin{center}
\begin{tikzpicture}[scale=0.9]
\draw [thick][->] (-5, 0)--(5, 0);
\draw (-4, 0)--(0, 4);
\draw (-3, 0)--(0, 3.8);
\draw (-2, 0)--(0, 3.5);
\draw (-1, 0)--(0, 3);
\draw (1, 0)--(0, 3);
\draw (2, 0)--(0, 3.5);
\draw (3, 0)--(0, 3.8);
\draw (4, 0)--(0, 4);
\draw [red][ultra thick](0, 2.3)--(0, 3)--(0, 3.5)--(0, 3.8)--(0, 4);
\draw [green][ultra thick](0,2.3)to [out=90, in=-30](-1.5, 5);
\draw [green][ultra thick](0,2.3)to [out=90, in=195](1.5, 5);
\draw[red][ultra thick](0, 4) to [out=90, in=-60](-0.3, 5);
\draw (0, 0)--(0, 2.3);
\draw[dashed](-5, 2.3)--(5,2.3);
\node at (5.1,0) {$x$};
\node at (-2.4, 4.8) {$x=x_{+}^{\epsilon}(t)$};
\node at (2.3, 4.8) {$x=x_{-}^{\epsilon}(t)$};
\node at (-4.5, 2.5) {$t=T_\epsilon$};
\node at (-1.8, -0.2) {$y_{-}^{\epsilon}(x, t)$};
\node at (2.2, -0.2) {$y_{+}^{\epsilon}(x, t)$};
\node at (-4.5, 0.2) {$t_0$};
\node at (0.1, 5.4){ Pre-shock curve $x=\phi^0(t)$};
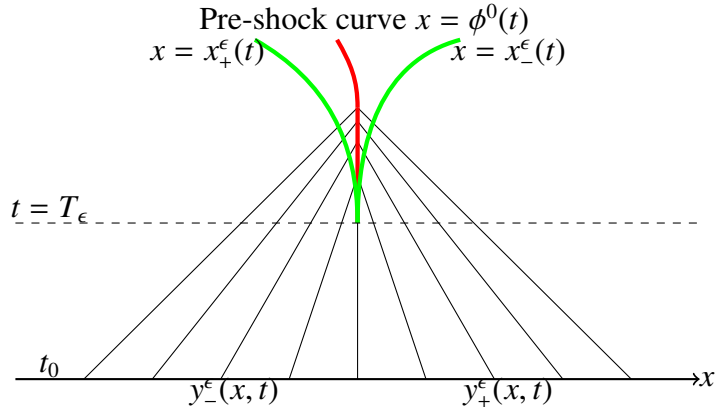
\node [below] at (0.6, -0.8){Figure $7$. Pre-shock curve $x=\phi^0(t)$ in cusp domain};
\end{tikzpicture}
\end{center}

In the following, we derive some important properties of $y^{\epsilon}_{\pm}(x, t)\in C^{\infty}(\Omega_{\pm})$.
\begin{lemma}\label{l:3.2}
 It holds that in $\Omega_{\pm}$,
 \[
 \begin{split}
 &|y^{\epsilon}_{\pm}(x, t)-y_{\epsilon}|\leq Cd^{\frac16}_{\epsilon}, \quad |\partial_x y^{\epsilon}_{\pm}(x, t)|\leq Cd^{-\frac13}_{\epsilon},
 |\partial_l y^{\epsilon}_{\pm}(x, t)|\leq Cd^{-\frac16}_{\epsilon},\\
 &|\partial^2_x y^{\epsilon}_{\pm}(x, t)|\leq Cd^{-\frac56}_{\epsilon},
 |\partial^2_{xt} y^{\epsilon}_{\pm}(x, t)|\leq Cd^{-\frac56}_{\epsilon},\quad |\partial^2_{t} y^{\epsilon}_{\pm}(x, t)|\leq Cd^{-\frac56}_{\epsilon},
 \end{split}
 \]
 where $d_{\epsilon}=(t-T_{\epsilon})^3+\Big(x-x_{\epsilon}-\lambda_i(w(x_{\epsilon}, T_{\epsilon}))(t-T_{\epsilon})\Big)^2$, and $l$ stands for the tangent direction of the $i$-th characteristics passing through the point $(x_{\epsilon}, T_{\epsilon})$.
\end{lemma}

\begin{proof}
It only suffices to prove the desired results for $y^{\epsilon}_+(x, t)$. By direct calculations, one has
\[
A(t)^3+(B(t)-\varphi(y, t))^2\sim d_{\epsilon}.
\]
In addition, we have
\begin{equation}\label{d1}
\begin{split}
h(y^{\epsilon}_+(x, t), t)=&\Big(-\frac{B(t)-\varphi(y^\epsilon_+(x, t), t)}{2}+\sqrt{\frac{1}{4}(B(t)-\varphi(y^\epsilon_+(x, t), t))^2-\frac{1}{27}A(t)^3}\Big)^{\frac{1}{3}}\\
&+\Big(-\frac{B(t)-\varphi(y^\epsilon_+(x, t), t)}{2}-\sqrt{\frac{1}{4}(B(t)-\varphi(y^\epsilon_+(x, t), t))^2
-\frac{1}{27}A(t)^3}\Big)^{\frac{1}{3}}
\end{split}
\end{equation}
and
\begin{equation}\label{dy1}
\begin{split}
&|B(t)-\varphi(y^\epsilon_+(x, t), t)|>\frac{2\sqrt{3}}{9}A(t)^{\frac{3}{2}},\\
&|\Big(-\frac{B(t)-\varphi(y^\epsilon_+(x, t), t)}{2}+\sqrt{\frac{1}{4}(B(t)-\varphi(y^\epsilon_+(x, t), t))^2-\frac{1}{27}A(t)^3}\Big)^{\frac{2}{3}}+\frac{1}{3}A(t)\\
&\Big(-\frac{B(t)-\varphi(y^\epsilon_+(x, t), t)}{2}-\sqrt{\frac{1}{4}(B(t)-\varphi(y^\epsilon_+(x, t), t))^2-\frac{1}{27}A(t)^3}\Big)^{\frac{2}{3}}|\leq Cd^{\frac13}_\epsilon.
\end{split}
\end{equation}
Therefore
\begin{equation}\label{eq:3.8*}
C_1d^{\frac16}_{\epsilon}\leq |h(y^{\epsilon}_+(x, t), t)|\leq C_2 d^{\frac16}_{\epsilon},
\end{equation}
which implies $h(y^{\epsilon}_+(x, t), t)\sim d^{\frac16}_{\epsilon}$
with $C_1$ and $C_2$ are positive constants independent of $\epsilon$.

Note that
\[
\begin{split}
&\varphi(y, t)-\varphi(y_\epsilon, T_{\epsilon})-\partial_t\varphi(y_\epsilon, T_\epsilon)(t-T_\epsilon)\\
=&\frac{1}{2}\partial^2_{ty}\varphi(y_\epsilon, T_\epsilon)(y-y_\epsilon)(t-T_\epsilon)+\frac16\partial^3_y\varphi(y_\epsilon,
T_{\epsilon})(y-y_\epsilon)^3+O(1)\Big((t-T_\epsilon)^2+(y-y_\epsilon)^2(t-T_\epsilon)\Big),
\end{split}
\]
then it follows from the definition of $d_{\epsilon}$ and $\eqref{dy1}_1$ that
\[
(y-y_\epsilon)^3-(y-y_\epsilon)(t-T_\epsilon)\sim d^{\frac12}_\epsilon\sim(t-T_\epsilon)^{\frac32}.
\]
Similarly to the estimate of \eqref{eq:3.8*}, we can prove $|y^{\epsilon}_+(x, t)-y_\epsilon|\leq Cd^{\frac16}_{\epsilon}$.

From \eqref{eq:3.2}, it holds
\begin{equation}\label{eq:3.8}
h^3\big(y^{\epsilon}_+(x, t), t\big)-A(t)h(y^{\epsilon}_+(x, t), t)+B(t)=x=\varphi(y^{\epsilon}_+(x, t), t).
\end{equation}
Differentiating \eqref{eq:3.8} with respect to $x$ yields
\begin{equation}\label{eq:3.9}
\partial_x y^{\epsilon}_+(x, t)=\frac{1}{\partial_y h(y^{\epsilon}_+(x, t), t)\Big(3h^2( y^{\epsilon}_+(x, t), t)-A(t)\Big)}.
\end{equation}
On the other hand, based on \eqref{d1}, $\eqref{dy1}_1$ and \eqref{eq:3.8}, one has
\begin{equation}\label{eq:3.10}
\begin{split}
|3h^2(y^{\epsilon}_+(x, t), t)-A(t)|=&3|h^2(y^{\epsilon}_+(x, t), t)-A(t)|+2A(t)\\
=&|h^{-1}(y^{\epsilon}_+(x, t), t)||\varphi(y^\epsilon_+(x, t), t)-B(t)|+2A(t)\\
\geq&Cd^{-\frac{1}{6}}_\epsilon((\varphi(y^\epsilon_+(x, t), t)-B(t))^2+A(t)^3)^{\frac{1}{2}}=Cd^{\frac{1}{3}}_\epsilon,
\end{split}
\end{equation}
where we have used the fact derived from the formula \eqref{d1} that  $h(y^{\epsilon}_+(x, t), t)$ has
the same sign with $\varphi(y^\epsilon_+(x, t), t)-B(t)$. Then it follows from \eqref{eq:3.8} that
\[
h^2(y^{\epsilon}_+(x, t), t)-A(t)=h^{-1}(y^{\epsilon}_+(x, t), t)(\varphi(y^\epsilon_+(x, t), t)-B(t))>0.
\]
Together with \eqref{eq:3.9}, this yields
\begin{equation}\label{DY-2}
|\partial_x y^{\epsilon}_+(x, t)|\leq Cd^{-\frac13}_{\epsilon}.
\end{equation}
In addition, one has from \eqref{eq:3.8} that
\begin{equation}\label{eq:3.11*}
\partial_t y_+^{\epsilon}=-\frac{(3h^2-A(t))\partial_t h-A'(t)h+B'(t)}{\partial_y h(y^{\epsilon}_+(x, t), t)(3h^2-A(t))}=-\frac{\partial_t\varphi(y^{\epsilon}_+(x, t), t)}{\partial_y h(y^{\epsilon}_+(x, t), t)(3h^2-A(t))}.
\end{equation}
This implies $|\partial_t y_+^{\epsilon}|\leq Cd^{-\frac{1}{3}}_{\epsilon}$ and
\[
|\partial_l y^{\epsilon}_+(x, t)|=\big|\partial_t y^{\epsilon}_+(x, t)+\lambda_i(w(y^{\epsilon}_+(x, t), t))\partial_x y^{\epsilon}_+(x, t)\big|
=\big|\frac{\partial_t\varphi(y^{\epsilon}_+(x, t), t)-\partial_t\varphi(y_\epsilon, T_\epsilon)}{\partial_y h(y^{\epsilon}_+(x, t), t)(3h^2-A)}\big|\leq Cd^{-\frac16}_{\epsilon},
\]
where
\[
\begin{split}
&|\partial_t\varphi(y^{\epsilon}_+(x, t), t)-\partial_t\varphi(y_\epsilon, T_\epsilon)|\\
=&\big|\partial^2_{ty}\varphi(y_\epsilon, T_{\epsilon})(y^{\epsilon}_+-y_{\epsilon})
+\partial^2_{t}\varphi(y_{\epsilon}, T_\epsilon)(t-T_{\epsilon})+O(1)\Big((y^{\epsilon}_{+}-y_{\epsilon})^2+(t-T_{\epsilon})^2\Big)\big|
\leq  Cd^{\frac{1}{6}}_{\epsilon}.
\end{split}
\]


Next, we treat $\partial^2_x y^{\epsilon}_+(x, t)$, $\partial^2_{xt} y^{\epsilon}_+(x, t)$ and $\partial^2_{t} y^{\epsilon}_+(x, t)$.
It follows from \eqref{eq:3.9} and direct computation that
\[
\begin{split}
|\partial^2_{x}y^{\epsilon}_+(x, t)|=&\big|\Big(\frac{\partial^2_{y}h(y^{\epsilon}_+(x, t), t)}{(\partial_y h)^2(y^{\epsilon}_+(x, t), t)(3h^2-A(t))}+\frac{6h(y^{\epsilon}_+(x, t), t)}
{(3h^2-A(t))^2}\Big)\partial_x y^{\epsilon}_+(x, t)\big|\\
\leq& C(d^{-\frac16}_{\epsilon}d^{-\frac13}_{\epsilon}+d^{\frac16}_{\epsilon}d^{-\frac{2}{3}}_{\epsilon})d^{-\frac13}_{\epsilon}\leq Cd^{-\frac56}_{\epsilon},
\end{split}
\]
where we have used the facts that
\[
\begin{split}
|h(y^{\epsilon}_+(x, t), t)|\le &Cd^{\frac16}_{\epsilon},\;\partial_y h(y_\epsilon, T_\epsilon)=1,\;
 |3h^2-A(t)|\geq Cd^{\frac13}_{\epsilon},\;|\partial_x y^{\epsilon}_+|\leq Cd^{-\frac{1}{3}}_{\epsilon},\\
|\partial^2_{y}\varphi(y^{\epsilon}_+(x, t), t)|=&\big|\partial^2_{y}\varphi(y_\epsilon, T_\epsilon)+\partial^3_{y}\varphi(y_{\epsilon}, T_{\epsilon})(y^{\epsilon}_+(x, t)-y_{\epsilon})+\partial^3_{tyy}\varphi(y_{\epsilon}, T_{\epsilon})(t-T_{\epsilon})\\
&+O(1)\Big((t-T_{\epsilon})^2+(y^{\epsilon}_{+}-y_{\epsilon})^2\Big)\big|\leq Cd^{\frac16}_{\epsilon},\\
|\partial^2_{y}h(y^{\epsilon}_+(x, t), t)|=&\big|\frac{1}{3h^2(y^{\epsilon}_+(x, t), t)-A(t)}\Big(\partial^2_{y}\varphi (y^{\epsilon}_+(x, t), t)-6h(\partial_y h)^2\Big)\big|\leq Cd^{-\frac{1}{6}}_{\epsilon}.
\end{split}
\]
Taking the first order derivatives of \eqref{eq:3.2} with respect to $t$ and $y$ respectively, we arrive at
\begin{align}
\partial_t h(y, t)=&\frac{\partial_t \varphi+A'(t)h-B'(t)}{3h^2-A(t)},\label{eq:3.16}\\
\partial_y h(y, t)=&\frac{\partial_y \varphi}{3h^2-A(t)}.\label{eq:3.17}
\end{align}
From \eqref{eq:3.16}--\eqref{eq:3.17}, it holds that
\begin{equation}\label{eq:3.24}
\partial_t h+\partial_t y^{\epsilon}_+\partial_y h=\frac{1}{3h^2-A(t)}(\partial_t\varphi-B'(t)
+\partial_y\varphi\partial_t y^{\epsilon}_++A'(t)h),
\end{equation}
where $A'(t)=1+O(1)(t-T_{\epsilon})$ by $A'(T_{\epsilon})=1$. Using Taylor expansion formula, one has
\[
\begin{split}
\partial_t\varphi(y^{\epsilon}_+, t)=&\partial_t\varphi(y_{\epsilon}, T_{\epsilon})+\partial^2_{ty}\varphi(y_{\epsilon}, T_{\epsilon})(y^{\epsilon}_+-y_{\epsilon})
+\partial^2_{t}\varphi(y_{\epsilon}, T_{\epsilon})(t-T_{\epsilon})+O(1)\Big((t-T_\epsilon)^2+(y^{\epsilon}_+-y_{\epsilon})^2\Big),\\
B'(t)=&\partial_t\varphi(y_{\epsilon}, T_\epsilon)+B''(T_\epsilon)(t-T_{\epsilon})+O(1)(t-T_{\epsilon})^2.
\end{split}
\]
This yields
\begin{equation}\label{eq:3.21}
\begin{split}
|\partial_t\varphi(y^{\epsilon}_+, t)-B'(t)|=&\big|\partial^2_{ty}\varphi(y_{\epsilon}, T_{\epsilon})(y^{\epsilon}_+-y_{\epsilon})
+\Big(\partial^2_{t}\varphi(y_{\epsilon}, T_{\epsilon})-B''(T_{\epsilon})\Big)(t-T_{\epsilon})\\
&+O(1)\Big((t-T_{\epsilon})^2+(y^{\epsilon}_+-y_{\epsilon})^2\Big)\big|\leq Cd^{\frac16}_{\epsilon}.
\end{split}
\end{equation}
Along with $\eqref{eq:3.8*}$, \eqref{eq:3.11*}, \eqref{eq:3.21} and
\begin{equation}\label{eq:3.21*}
\begin{split}
|\partial_y\varphi(y^{\epsilon}_+, t)|=&\big|\partial^2_{ty}\varphi(y_{\epsilon}, T_{\epsilon})
(t-T_{\epsilon})
+\frac12\partial^3_{y}\varphi(y_{\epsilon}, T_{\epsilon})
(y^{\epsilon}_+-y_{\epsilon})^2\\
&+O(1)\Big(
(t-T_{\epsilon})^2+(t-T_{\epsilon})(y^{\epsilon}_+-y_{\epsilon})+(y^{\epsilon}_+-y_{\epsilon})^3\Big)\big|
\leq Cd^{\frac13}_{\epsilon},
\end{split}
\end{equation}
 we can derive from \eqref{eq:3.24}  that
\begin{equation}\label{eq:3.25}
|\partial_t h+\partial_t y^{\epsilon}_+\partial_y h|\leq Cd^{-\frac13}_{\epsilon}.
\end{equation}

Differentiating \eqref{eq:3.17} with respect to $t$ and $y$ respectively yields
\begin{equation}\label{eq:3.19}
\partial^2_{yt}h+\partial_t y^{\epsilon}_+(x, t)\partial^2_{y}h=\frac{
\partial^2_{yt}\varphi+\partial^2_{y}\varphi\partial_t y^{\epsilon}_+}{3h^2-A(t)}
-\frac{6h\partial_y \varphi(\partial_t h+\partial_t y^{\epsilon}_+\partial_y h)}{(3h^2-A(t))^2}+\frac{A'(t)\partial_y\varphi}
{(3h^2-A(t))^2},
\end{equation}
where $|\partial_t y^{\epsilon}_+(t,x)|\leq Cd^{-\frac13}_{\epsilon}$, and
\begin{equation*}
\begin{split}
\quad\;\;|\partial^2_{yt}\varphi(y^{\epsilon}_+, t)|=&\big|\partial^2_{yt}\varphi(y_{\epsilon}, T_{\epsilon})
+\partial^3_{yty}\varphi(y_{\epsilon}, T_{\epsilon})
(y^{\epsilon}_{+}-y_{\epsilon})+\partial^3_{ytt}\varphi(y_{\epsilon}, T_{\epsilon})(t-T_\epsilon)\\
&+O(1)\Big(
(t-T_{\epsilon})^2+(y^{\epsilon}_+-y_{\epsilon})^2\Big)\big|
\leq C,\\
\quad\;\;|\partial^2_{y}\varphi(y^{\epsilon}_+, t)|=&\big|\partial^2_{y}\varphi(y_{\epsilon}, T_{\epsilon})
+\partial^3_{y}\varphi(y_{\epsilon}, T_{\epsilon})
(y^{\epsilon}_{+}-y_{\epsilon})+\partial^3_{yyt}\varphi(y_{\epsilon}, T_{\epsilon})(t-T_\epsilon)\\
&+O(1)\Big(
(t-T_{\epsilon})^2+(y^{\epsilon}_+-y_{\epsilon})^2\Big)\big|
\leq Cd^{\frac16}_{\epsilon},\\
\big|\partial^2_{yt}\varphi(y^{\epsilon}_+, t)+&\partial_t y^{\epsilon}_{+}
\partial^2_{y}\varphi(y^{\epsilon}_+, t)\big|\leq C d^{-\frac16}_{\epsilon}.
\end{split}
\end{equation*}
Then we obtain
\begin{equation}\label{eq:3.20*}
|\partial^2_{ty}h+\partial_t y^{\epsilon}_+(x, t)\partial^2_{y}h|\leq Cd^{-\frac12}_{\epsilon}.
\end{equation}
In addition, differentiating \eqref{eq:3.9} and \eqref{eq:3.11*} with respect to $t$ respectively yields
\begin{equation}\label{eq:3.17*}
\begin{split}
\partial^2_{xt}y^{\epsilon}_+=&-\frac{\partial^2_{yt}h+\partial^2_{y}h \partial_t y^{\epsilon}_+}{(\partial_y h)^2(y^{\epsilon}_+(x, t), t)(3h^2-A(t))}-
\frac{6h(\partial_t h+\partial_y h\partial_t y^{\epsilon}_+)-A'(t)}{\partial_y h(y^{\epsilon}_+(x, t), t)(3h^2-A(t))^2},\\
\partial^2_{t}y^{\epsilon}_+=&-\frac{\partial^2_{t}\varphi+\partial^2_{ty}\varphi\partial_t y^{\epsilon}_+}{\partial_y h(y^{\epsilon}_+(x, t), t)(3h^2-A(t))}+
\frac{\partial_t \varphi(\partial^2_{yt}h+\partial^2_{y}h\partial_t y^{\epsilon}_+)}
{(\partial_y h)^2(y^{\epsilon}_+(x, t), t)(3h^2-A(t))}\\
&+\frac{6\partial_t\varphi h(\partial_t h+\partial_y h\partial_t y^{\epsilon}_+)-\partial_t\varphi A'(t)}{\partial_y h(y^{\epsilon}_+(x, t), t)(3h^2-A)^2}.\\
\end{split}
\end{equation}
Due to $|3h^2-A(t)|\geq Cd^{\frac13}_{\epsilon}$,
$|\partial^2_{t}\varphi(y^{\epsilon}_{+}, t)+\partial_t y^{\epsilon}_{+}\partial^2_{ty}\varphi(y^{\epsilon}_{+}, t)|
\leq Cd^{-\frac{1}{3}}_{\epsilon}$, and
\[
\begin{split}
|\partial^2_{t}\varphi(y^{\epsilon}_+, t)|=&\big|\partial^2_{t}\varphi(y_{\epsilon}, T_{\epsilon})+\partial^3_{t}\varphi(y_{\epsilon}, T_{\epsilon})(t-T_\epsilon)+\partial^3_{tty}\varphi(y_{\epsilon}, T_{\epsilon})
(y^{\epsilon}_+-y_{\epsilon})\\&+O(1)\Big(
(t-T_{\epsilon})^2+(y^{\epsilon}_+-y_{\epsilon})^2\Big)\big|
\leq C,
\end{split}
\]
then together with \eqref{eq:3.8*},  \eqref{eq:3.25} and \eqref{eq:3.20*},  we deduce from \eqref{eq:3.17*} that
\[
|\partial^2_{xt}y^{\epsilon}_+|\leq C d^{-\frac56}_{\epsilon},\quad |\partial^2_{t}y^{\epsilon}_+|\leq C d^{-\frac56}_{\epsilon}.
\]
\end{proof}

\subsection{Estimates on the  pre-shock wave}

For $t\in [T_{\ve}, T_{\ve}+1]$ and $\Omega=\{(x,t): x_{\ve}-\lambda^*(T_{\ve}+1-t)\le x\le x_{\ve}+\lambda^*(T_{\ve}+1-t)\}$
with $\lambda^*=2\displaystyle\max_{1\le k\le n}\{|\lambda_k(0)|\}$, define $w_{\pm}^0(x,t)=v(y^{\epsilon}_{\pm}(x, t),t)$
in $\Omega_{\pm}$ respectively and
\[
w^0(x, t)=
\begin{cases}
w_{-}^0(x, t), \quad x<\phi^0(t),\\[3pt]
w_{+}^0(x, t), \quad x>\phi^0(t),
\end{cases}
\]
where the pre-shock curve $\Gamma$: $x=\phi^0(t)$ has been defined in \eqref{YHC-7}, and $w^0_{-}(x, t)$, $w^0_+(x, t)$
represent the corresponding left and right states of the pre-shock wave. $(w^0(x, t), \phi^0(t))$ will be taken
as the first approximation of shock solutions.
Next, we derive some basic properties of $w^0(x, t)=(w^0_1, \cdot\cdot\cdot, w^0_n)(x,t)$.
\begin{lemma}\label{l:3.3}
 In the domain $\Omega\setminus\Gamma$,
 \begin{itemize}
\item[1.] $w^0_i(x, t)$ fulfills the estimates:
\begin{equation}\label{eq:3.12}
\begin{cases}
|w^0_i(x, t)-w^0_i(x_{\epsilon}, T_{\epsilon})|\leq C\epsilon d^{\frac16}_{\epsilon},\\[5pt]
|\partial_lw^0_i(x, t)|\leq C\epsilon d^{-\frac16}_{\epsilon},\quad
|\partial_xw^0_i(x, t)|\leq C\epsilon d^{-\frac13}_{\epsilon},\\[5pt]
|\partial_x^2w^0_i(x, t)|\leq C\epsilon d^{-\frac56}_{\epsilon}.
\end{cases}
\end{equation}

\item[2.] For $j\neq i$, $w^0_j(x, t)$ satisfies the estimates:
\begin{equation}\label{eq:3.13}
\begin{cases}
|w^0_j(x, t)-w^0_j(x_{\epsilon}, T_{\epsilon})|\leq C\epsilon d^{\frac13}_{\epsilon},\\[5pt]
|\partial_tw^0_j(x, t)|\leq C\epsilon,\;\; |\partial_xw^0_j(x, t)|\leq C\epsilon,\\[5pt]
|\partial_x^2w^0_j(x, t)|\leq C\epsilon d^{-\frac12}_{\epsilon},\;
|\partial_t^2w^0_j(x, t)|\leq C\epsilon d^{-\frac12}_{\epsilon},\;
|\partial_{tx}^2w^0_j(x, t)|\leq C\epsilon d^{-\frac12}_{\epsilon}.
\end{cases}
\end{equation}
\end{itemize}
\end{lemma}

\begin{proof} It suffices to show the results in domain $\Omega_+$. Thanks to Theorem \ref{T:2.1}, one knows that
$$|\partial^{\alpha}_{t, y}v_k(y, t)|\leq C_{\alpha}\epsilon, \quad |\alpha|\geq 0, k=1,..., n.$$
Together with Lemma \ref{l:3.2}, we can obtain
\[
\begin{split}
&|w^0_i(x, t)-w^0_i(x_{\epsilon}, T_{\epsilon})|=|v_i(y^{\epsilon}_+(x, t), t)-v_i(y_{\epsilon}, T_{\epsilon})|\\
=&|\partial_tv_i(y_{\epsilon}, T_{\epsilon})(t-T_{\epsilon})+\partial_yv_i(y_{\epsilon}, T_{\epsilon})(y^{\epsilon}_+(x, t)-y_{\epsilon})+O(1)\epsilon\Big((t-T_{\epsilon})^2+(y^{\epsilon}_+(x, t)-y_{\epsilon})^2\Big)|
\leq C\epsilon d^{\frac16}_{\epsilon},\\[3pt]
&|\partial_lw^0_i(x, t)|=|\partial_tv_i(y^{\epsilon}_+(x, t), t)+\partial_l y^{\epsilon}_+(x, t)\partial_yv_i(y^{\epsilon}_+(x, t), t)|
\leq C\epsilon+C\epsilon d^{-\frac16}_{\epsilon}\leq C\epsilon d^{-\frac16}_{\epsilon},\\
&|\partial_xw^0_i(x, t)|=|\partial_yv_i(y^{\epsilon}_+(x, t), t)\partial_x y^{\epsilon}_+(x, t)|\leq C\epsilon d^{-\frac13}_{\epsilon},\\
&|\partial^2_xw^0_i(x, t)|=|\partial^2_yv_i(y^{\epsilon}_+(x, t), t)\big(\partial_x y^{\epsilon}_+(x, t)\big)^2+
\partial_yv_i(y^{\epsilon}_+(x, t), t)\partial^2_x y^{\epsilon}_+(x, t)|\leq C\epsilon d^{-\frac56}_{\epsilon}.
\end{split}
\]

In the rest, we estimate $w_j$ for $j\neq i$.  Since the $i$-th
right eigenvector of $A(w)$ is $r_i(w)=(0, 0, \cdots, 1, 0, \cdots,0)^{\top}$ and $l_j(w)\cdot r_i(w)=0$
holds, then for small $|w|$,
$$l_{ji}(w)=0 \;\; \text{for $j\neq i$}.$$

Note that from the third equations of the blowup system \eqref{eq:2.1}, one has that for $k\neq i$,
 \begin{equation}\label{eq:3.14}
 \begin{split}
 \sum_{j\neq i}l_{kj}(v)\Big(\partial_tv_j(y, t)\partial_y\varphi(y, t)+(\lambda_k-\lambda_i)(v(y, t))\partial_y
 v_j(y, t)\Big)=0.
 \end{split}
 \end{equation}
Due to
 \[
 \det\left(
 \begin{array}{ccccccc}
 l_{11}&l_{12}&\cdots&l_{1(i-1)}&l_{1(i+1)}&\cdots&l_{1n}\\
 \vdots\\
 l_{j1}&l_{j2}&\cdots&l_{j(i-1)}&l_{j(i+1)}&\cdots&l_{jn}\\
 \vdots\\
 l_{n1}&l_{n2}&\cdots&l_{n(i-1)}&l_{n(i+1)}&\cdots&l_{nn}\\
 \end{array}
 \right)\neq 0
 \]
and $\partial_y\varphi(y_{\epsilon}, T_{\epsilon})=0$, then for small $|w|$, we have
\[
\partial_yv_j(y_{\epsilon}, T_{\epsilon})=0, \quad j\neq i.
\]
Taking the first order derivative on the equations in \eqref{eq:3.14}  with respect to $y$ yields
\[
\begin{split}
&\sum_{j\neq i}\Big(\partial_y l_{kj}(v)(\partial_tv_j\partial_y\varphi+(\lambda_k-\lambda_i)(v)\partial_yv_j)
+l_{kj}(v)\Big(\partial^2_{ty}v_j\partial_y\varphi
+\partial_tv_j\partial^2_y\varphi\\
&\quad+\sum^n_{l=1}\partial_{v_l}(\lambda_k-\lambda_i)(v)\partial_yv_l\partial_yv_j+
(\lambda_k-\lambda_i)(v)\partial^2_yv_j\Big)\Big)=0,\quad k\neq i.
\end{split}
\]
Because of
\[
\partial_y\varphi(y_{\epsilon}, T_{\epsilon})=\partial^2_y\varphi(y_{\epsilon}, T_{\epsilon})=\partial_yv_j(y_{\epsilon}, T_{\epsilon})=0,
\]
then $\partial^2_y v_j(y_{\epsilon}, T_{\epsilon})=0$ for $j\not=i$.
It follows from Taylor expansion formula and Lemma \ref{l:3.2} that  for $j\not=i$,
\[
\begin{split}
&|w^0_j(x, t)-w^0_j(x_{\epsilon}, T_{\epsilon})|\\
=&\big|\partial_tv_j(y_\epsilon, T_{\epsilon})(t-T_{\epsilon})+\partial_yv_j(y_\epsilon, T_{\epsilon})(y^{\epsilon}_{+}-y_{\epsilon})
+\partial^2_{y}v_j(y_{\epsilon}, T_{\epsilon})(y^{\epsilon}_{+}-y_{\epsilon})^2\\
&+O(1)\epsilon\Big((t-T_{\epsilon})^2+(t-T_{\epsilon})(y^{\epsilon}_{+}-y_{\epsilon})
+(y^{\epsilon}_{+}-y_{\epsilon})^3\Big)\big|
\leq C\epsilon d^{\frac13}_{\epsilon},\\[3pt]
&|\partial_x w^0_j(x, t)|
=|\partial_yv_j(y^{\epsilon}_+(x, t), t)\partial_x y^{\epsilon}_+|
\leq C\epsilon,\\
&|\partial_t w^0_j(x, t)|
\leq
|\partial_tv_j(y^{\epsilon}_+(x, t), t)|+|\partial_t y^{\epsilon}_+||\partial_yv_j(y^{\epsilon}_+(x, t), t)|
\leq C\epsilon,\\
&|\partial^2_x w^0_j(x, t)|=|\partial^2_yv_j(y^{\epsilon}_+(x, t), t)(\partial_x y^{\epsilon}_+)^2+
\partial_yv_j(y^{\epsilon}_+(x, t), t)\partial^2_x y^{\epsilon}_+(x, t)|\leq C\epsilon d^{-\frac12}_{\epsilon},\\
&|\partial^2_t w^0_j(x, t)|=|\partial^2_tv_j(y^{\epsilon}_+(x, t), t)
+2\partial^2_{ty}v_j(y^{\epsilon}_+(x, t), t)\partial_t y^{\epsilon}_+(x, t)+\partial^2_{y}v_j(y^{\epsilon}_+(x, t), t)(\partial_t  y^{\epsilon}_+)^2\\
&\qquad\quad\quad\quad+\partial_yv_j(y^{\epsilon}_+(x, t), t)\partial^2_{t} y^{\epsilon}_+|
\leq C\epsilon+C\epsilon d^{-\frac13}_{\epsilon}+C\epsilon d^{-\frac12}_{\epsilon}\leq C\epsilon
d^{-\frac12}_{\epsilon},\\
&|\partial^2_{tx} w^0_j(x, t)|\\
=&|\partial^2_{ty}v_j(y^{\epsilon}_+(x, t), t)\partial_x y^{\epsilon}_{+}+\partial^2_{y}v_j(y^{\epsilon}_+(x, t), t)
\partial_t y^{\epsilon}_+(x, t)\partial_x y^{\epsilon}_+(x, t)+\partial_{y}v_j(y^{\epsilon}_+(x, t), t)\partial^2_{tx}  y^{\epsilon}_+|\\
\leq &C\epsilon d^{-\frac13}_{\epsilon}+C\epsilon d^{-\frac12}_{\epsilon}\leq C\epsilon
d^{-\frac12}_{\epsilon},
\end{split}
\]
where we have used the facts that
\[
\begin{split}
&|\partial^3_{y}v_j(y_\epsilon, T_{\epsilon})|\leq C\epsilon, \quad |\partial^3_{tyy}v_j(y_{\epsilon}, T_{\epsilon})|\leq C\epsilon,\\
&|\partial_yv_j(y^{\epsilon}_+(x, t), t)|
=\big|\partial_yv_j(y_\epsilon, T_{\epsilon})+\partial^2_{ty}v_j(y_{\epsilon},
T_{\epsilon})(t-T_{\epsilon})+\partial^2_{y}v_j(y_{\epsilon},
T_{\epsilon})(y^{\epsilon}_{+}-y_{\epsilon})\\
&\quad+\partial^3_{y}v_j(y_{\epsilon},
T_{\epsilon})(y^{\epsilon}_{+}-y_{\epsilon})^2+O(1)\epsilon\Big((t-T_{\epsilon})^2+(t-T_{\epsilon})(y^{\epsilon}_{+}-y_{\epsilon})
+(y^{\epsilon}_{+}-y_{\epsilon})^3\Big)\big|
\leq C\epsilon d^{\frac13}_{\epsilon},\\
&|\partial^2_{y}v_j(y^{\epsilon}_+(x, t), t)|=\big|\partial^2_{y}v_j(y_\epsilon, T_\epsilon)+\partial^3_{y}v_j(y_\epsilon, T_{\epsilon})(y-y_\epsilon)+\partial^3_{tyy}v_j(y_{\epsilon}, T_{\epsilon})(t-T_{\epsilon})\\
&+O(1)\Big((t-T_\epsilon)^2+(y-y_\epsilon)(t-T_{\epsilon})\Big)\big|\leq C\epsilon d^{\frac16}_{\epsilon},\\
&|\partial_t y^{\epsilon}_+(x, t)|\leq Cd^{-\frac13}_{\epsilon}, \;\; |\partial^2_{t} y^{\epsilon}_+(x, t)|\leq Cd^{-\frac{5}{6}}_{\epsilon},
\;\; |\partial^2_{xt} y^{\epsilon}_+(x, t)|\leq Cd^{-\frac{5}{6}}_{\epsilon}.
\end{split}
\]
Therefore, we have finished the proof of this lemma.
\end{proof}

Next, we estimate the jump of the pre-shock wave. Let the jump of $w^0(x, t)$ across the pre-shock curve $x=\phi^0(t)$ be
\[
[w^0]=w^0(\phi^0(t)+0, t)-w^0(\phi^0(t)-0, t).
\]
\begin{lemma}\label{l:3.4}
The following estimates hold
\[
|[w^0_i]|\leq {\bar C}_0\epsilon(t-T_{\epsilon})^{\frac12}, \quad |[w^0_j]|\leq {\bar C}_0\epsilon(t-T_{\epsilon})^{\frac32} \quad
\text{for $j\neq i$.}
\]
\end{lemma}
\begin{proof}
By $\phi^0(t)-x_\epsilon-\lambda_i(w(x_{\epsilon}, T_{\epsilon}))(t-T_{\epsilon})=O(1)\Big((t-T_{\epsilon})^2\Big)$,
one has
\begin{equation}\label{Eq:3.25*}
d_{\epsilon}=(t-T_{\epsilon})^3+\Big(\phi^0(t)-x_{\epsilon}
-\lambda_i(w(x_{\epsilon}, T_{\epsilon}))(t-T_{\epsilon})\Big)^2\sim (t-T_{\epsilon})^3.
\end{equation}
Based on Lemma \ref{l:3.3}, we have
\[
\begin{split}
\left|[w^0_i]\right|\leq& \left|w^0_i(\phi^0(t)+0, t)-w^0_i(x_{\epsilon}, T_{\epsilon})\right|+\left|w^0_i(x_{\epsilon}, T_{\epsilon})
-w^0_i(\phi^0(t)-0, t)\right|\\
\leq &C\epsilon d^{\frac16}_{\epsilon}\leq {\bar C}_0\epsilon(t-T_{\epsilon})^{\frac12}.
\end{split}
\]
On the other hand, in $\Omega_{\pm}$ it holds that
\[
\begin{split}
w^0_j(x, t)-w^0_j(x_{\epsilon}, T_{\epsilon})=&\partial_tv_j(y_{\epsilon}, T_{\epsilon})(t-T_{\epsilon})+\partial_yv_j(y_{\epsilon}, T_{\epsilon})(y^{\epsilon}_+-y_{\epsilon})\\
&+O(1)\Big(\epsilon(t-T_{\epsilon})^2+\epsilon(t-T_{\epsilon})(y^{\epsilon}_{+}-y_{\epsilon})
+\epsilon(y^{\epsilon}_+-y_{\epsilon})^3\Big).
\end{split}
\]
Thus
\[
\begin{split}
|[w^0_j]|=&\left|w^0_j(\phi^0(t)+0, t)-w^0_j(x_{\epsilon}, T_{\epsilon})-(w^0_j(\phi^0(t)-0, t)-w^0_j(x_{\epsilon}, T_{\epsilon}))\right|\\
\leq& C\epsilon d^{\frac12}_{\epsilon}\leq {\bar C}_0\epsilon(t-T_{\epsilon})^{\frac32}, \quad j\neq i.
\end{split}
\]
\end{proof}

\section{Approximate shock solutions}

In this section, as in \cite{CXY},  we will take an analogous iterative scheme to construct the shock solution of \eqref{eq:1.1}.
For the general conservation law \eqref{eq:1.1},
the following Rankine-Hugoniot conditions across the shock curve $x=\phi(t)$ hold
\begin{equation}\label{eq:3.27}
\sigma [u_1]=[f_1(u)], \;\sigma[u_2]=[f_2(u)], \cdots,\; \sigma[u_n]=[f_n(u)],
\end{equation}
where $\sigma=\phi'(t)$ denotes the shock speed, and $[u]=u(\phi(t)+0)-u(\phi(t)-0)$. The corresponding entropy conditions
on the $i-$shock are given by
\begin{equation}\label{eq:3.28}
\lambda_{i-1}(w_{-}(t))<\sigma<\lambda_i(w_{-}(t)), \quad \lambda_i(w_+(t))<\sigma<\lambda_{i+1}(w_+(t)),
\end{equation}
where $w_{\pm}(t)=w_{\pm}(\phi(t)\pm, t)$, and $w_{\pm}(x,t)$ are the solutions of \eqref{eq:2.88}
on the left and right side of $x=\phi(t)$, respectively.

\subsection{Reformulated problem}

In order to avoid the difficulty caused by the movement of the shock curve,
it is natural to introduce such a coordinate transformation to fix the shock by
\[
\begin{cases}
t=t,\\[3pt]
z=x-\phi(t).
\end{cases}
\]

Under the new coordinate $(z, t)$, the blowup point becomes $(0, T_{\epsilon})$. By
 multiplying the equation in $\eqref{eq:2.88}$ by $l_j(w)$ for $j=1, 2, \cdots, n$, the resulting
 system is given as
 \begin{equation}\label{eq:3.16*}
 l_j(w)\cdot\left(
 \begin{array}{c}
 \partial_t w_1+(\lambda_j(w)-\sigma(t))\partial_zw_1\\[3pt]
 \partial_t w_2+(\lambda_j(w)-\sigma(t))\partial_zw_2\\[3pt]
 \vdots\\
 \partial_t w_n+(\lambda_j(w)-\sigma(t))\partial_zw_n
 \end{array}
 \right)=0.
 \end{equation}
Divided by $l_{ii}(w)\neq 0$ $(i=1, 2, \cdots, n)$, \eqref{eq:3.16*} can be transformed into
\begin{equation}\label{eq:3.29}
\begin{cases}
\partial_tw_{j}+(\lambda_j(w)-\sigma(t))\partial_zw_{j}+\displaystyle\sum_{k\neq i, j}p_{jk}(w)
\Big(\partial_tw_{k}+(\lambda_j(w)-\sigma(t))\partial_zw_{k}\Big)=0, \quad  j\neq i,\\[6pt]
\partial_tw_{i}+(\lambda_i(w)-\sigma(t))\partial_zw_{i}+\displaystyle\sum_{k\neq i}p_{ik}(w)
\Big(\partial_tw_{k}+(\lambda_i(w)-\sigma(t))\partial_zw_{k}\Big)=0,\\
w_j(z, t)|_{t=T_{\epsilon}}=w^0_j(z+x_\epsilon, T_{\epsilon}), \quad j=1, 2, \cdots, n,
\end{cases}
\end{equation}
where the coefficients $p_{jk}(w)|_{j=1, \cdots, n}$ are smooth functions of $w$, and
$p_{jk}(0)=0$.

Let
\[
\begin{split}
\tilde{\Omega}_{+}=&\Big\{(z, t): 0<z\leq \lambda^*(T_{\epsilon}
+1-t),\; T_\epsilon\leq t\leq T_{\epsilon}+1\Big\},\\
\tilde{\Omega}_{-}=&\Big\{(z, t):  -\lambda^*(T_{\epsilon}
+1-t)<z\leq 0,\; T_\epsilon\leq t\leq T_{\epsilon}+1\Big\}.
\end{split}
\]

We will construct the shock solutions to problem \eqref{eq:3.29} in the domain $\tilde{\Omega}_{-}\bigcup\tilde{\Omega}_+$ by the approximate procedure.
Problem \eqref{eq:3.29} can be reformulated by
\begin{equation}\label{eq:3.30}
\begin{cases}
\partial_tw_{j, \pm}+(\lambda_j(w_{\pm})-\sigma(t))\partial_zw_{j, \pm}+\displaystyle\sum_{k\neq i, j}p_{jk}(w_{\pm})
\Big(\partial_tw_{k, \pm}+(\lambda_j(w_{\pm})-\sigma(t))\partial_zw_{k, \pm}\Big)=0, \quad  j\neq i,\\[6pt]
\partial_tw_{i, \pm}+(\lambda_i(w_{\pm})-\sigma(t))\partial_zw_{i, \pm}+\displaystyle\sum_{k\neq i}p_{ik}(w_{\pm})
\Big(\partial_tw_{k, \pm}+(\lambda_i(w_{\pm})-\sigma(t))\partial_zw_{k, \pm}\Big)=0,\\
\sigma(t)=\lambda_i\Big(\int^1_0(\partial_{u_k}f_l)(\theta u(w_+(0+, t))+(1-\theta)u(w_{-}(0-, t)))\text{d}\theta\Big),\\
w_{j, \pm}(z, t)|_{t=T_{\epsilon}}=w^0_{j, \pm}(z+x_\epsilon, T_{\epsilon}), \quad  j=1, 2, \cdots, n,\\[3pt]
w_{j, -}(z, t)|_{z=0}=w_{j, -}(0-, t), \qquad  j=1, \cdots, i-1,\\
w_{j, +}(z, t)|_{z=0}=w_{j, +}(0+, t), \qquad j=i+1, \cdots, n,
\end{cases}
\end{equation}
where $w_{j, \pm}$ is defined in $\tilde{\Omega}_{\pm}$, the boundary values $w_{j, -}(0-, t)|_{j=1, \cdots, i-1}$  and $w_{j, +}(0+, t)|_{j=i+1, \cdots, n}$ satisfy the Rankine-Hugoniot conditions \eqref{eq:3.27}.

\subsection{Iteration schemes for resulting initial-value and initial-boundary value problems}

According to the entropy conditions \eqref{eq:3.28}, problem \eqref{eq:3.29} can be decomposed into an
initial-value problem and an initial-boundary value problem, which are coupled by the
Rankine-Hugoniot conditions \eqref{eq:3.27}. To solve them, we take the following  iterative schemes
 \begin{equation}\label{eq:3.31}
\begin{cases}
\partial_tw^{m+1}_{j, +}+(\lambda_j(w^m_{+})-\sigma^m(t))\partial_zw^{m+1}_{j, +}
+\displaystyle\sum_{k\neq i, j}p_{jk}(w^m_{+})
\Big(\partial_tw^m_{k, +}+(\lambda_j(w^m_{+})-\sigma^m(t))\partial_zw^m_{k, +}\Big)=0, \\
\qquad\qquad\qquad\qquad\qquad\qquad\qquad\qquad\qquad\qquad\qquad\qquad\qquad\;\text{for }\;  1\le j\le i-1,\\
\partial_tw^{m+1}_{i, \pm}+(\lambda_i(w^m_{\pm})-\sigma^m(t))\partial_zw^{m+1}_{i, \pm}+\displaystyle\sum_{k\neq i}p_{ik}(w^m_{\pm})
\Big(\partial_tw^m_{k, \pm}+(\lambda_i(w^m_{\pm})-\sigma^m(t))\partial_zw^m_{k, \pm}\Big)=0,\\
\partial_tw^{m+1}_{j, -}+(\lambda_j(w^m_{-})-\sigma^m(t))\partial_zw^{m+1}_{j, -}
+\displaystyle\sum_{k\neq i, j}p_{jk}(w^m_{-})
\Big(\partial_tw^m_{k, -}+(\lambda_j(w^m_{-})-\sigma^m(t))\partial_zw^m_{k, -}\Big)=0, \\
\qquad\qquad\qquad\qquad\qquad\qquad\qquad\qquad\qquad\qquad\qquad\qquad\qquad\;\text{for }\;   i+1\le j\le n,\\
w^{m+1}_{j, +}(z, t)|_{t=T_{\epsilon}}=w^0_{j, +}(z+x_\epsilon, T_{\epsilon}), \quad j=1, 2, \cdots, i-1,\\[3pt]
w^{m+1}_{i, \pm}(z, t)|_{t=T_\epsilon}=w^0_{i, \pm}(z+x_\epsilon, T_\epsilon),\\[3pt]
w^{m+1}_{j,-}(z, t)|_{t=T_\epsilon}=w^0_{j, -}(z+x_{\epsilon}, T_{\epsilon}), \quad j=i+1, \cdots, n
\end{cases}
\end{equation}
and
\begin{equation}\label{eq:3.32}
\begin{cases}
\partial_tw^{m+1}_{j, -}+(\lambda_j(w^m_{-})-\sigma^m(t))\partial_zw^{m+1}_{j, -}
+\displaystyle\sum_{k\neq i, j}p_{jk}(w^m_{-})
\Big(\partial_tw^m_{k, -}+(\lambda_j(w^m_{-})-\sigma^m(t))\partial_zw^m_{k, -}\Big)=0, \\
\qquad\qquad\qquad\qquad\qquad\qquad\qquad\qquad\qquad\qquad\qquad\qquad\qquad\;\text{for }\;  1\le j\le i-1,\\
\partial_tw^{m+1}_{j, +}+(\lambda_j(w^m_{+})-\sigma^m(t))\partial_zw^{m+1}_{j, +}
+\displaystyle\sum_{k\neq i, j}p_{jk}(w^m_{+})
\Big(\partial_tw^m_{k, +}+(\lambda_j(w^m_{+})-\sigma^m(t))\partial_zw^m_{k, +}\Big)=0,\\
\qquad\qquad\qquad\qquad\qquad\qquad\qquad\qquad\qquad\qquad\qquad\qquad\qquad\;\text{for }\;   i+1\le j\le n,\\
w^{m+1}_{j, -}(z, t)|_{z=0}=w^{m+1}_{j,-}(0-, t), \quad w^{m+1}_{j, -}(z, t)|_{t=T_\epsilon}=w^0_{j,-}(z+x_\epsilon, T_\epsilon),\;\;  j=1, \cdots, i-1,\\[3pt]
w^{m+1}_{j, +}(z, t)|_{z=0}=w^{m+1}_{j,+}(0+, t), \quad w^{m+1}_{j, +}(z, t)|_{t=T_\epsilon}=w^0_{j,+}(T_\epsilon,z+x_\epsilon),\;\; j=i+1, \cdots, n,\\[5pt]
\sigma^m(t)=\lambda_i\Big(\int^1_0(\partial_{u_k}f_l)(\theta u(w^{m}_+(0+, t))+(1-\theta)u(w^{m}_{-}( 0-, t)))\text{d}\theta\Big),\\
\end{cases}
\end{equation}
where $w^{m+1}_{j, -}(0-, t)|_{j=1, \cdots, i-1}$ and $w^{m+1}_{j, +}(0+, t)|_{j=i+1, \cdots, n}$ are determined by the
approximate Rankine-Hugoniot conditions
\[
\sigma^m[u_j(w^{m+1})]=[f_j(u(w^{m+1}))].
\]

In the sequel, we establish some uniform estimates on the approximate solutions $w^m_{\pm}(z, t)$ and $\sigma^m(t)$,
which will be used in the proof of  the convergence of the solutions.

\begin{lemma}\label{l:3.5}
For sufficiently small $\epsilon>0$, there exists a constant $M>{\bar C}_0$ independent of $\epsilon$,
where ${\bar C}_0>0$ is the positive constant given in Lemma \ref{l:3.4}, such that for all $m$, one has
that in $\tilde{\Omega}_{-}$ or $\tilde{\Omega}_+$,
\begin{align}
&w^m_{\pm}\in C^1(\tilde{\Omega}_{\pm}\setminus(0, T_{\epsilon})),\label{eq:3.33}\\
&\Big|w^m_{i, \pm}-w^0_{i, \pm}\Big|\leq M\epsilon(t-T_{\epsilon}),\label{eq:3.34}\\
&\Big|\partial_t\Big(w^m_{i, \pm}-w^0_{i, \pm}\Big)\Big|\leq M\epsilon \Big((t-T_{\epsilon})^3+z^2\Big)^{-\frac16},\label{eq:3.35}\\
&\Big|\partial_z\Big(w^m_{i, \pm}-w^0_{i, \pm}\Big)\Big|\leq M\epsilon \Big((t-T_{\epsilon})^3+z^2\Big)^{-\frac16}, \label{eq:3.36}\\
&\Big|w^m_{j, \pm}-w^0_{j, \pm}\Big|\leq M\epsilon(t-T_{\epsilon})^{\frac32},\qquad j\neq i,\label{eq:3.37}\\
&\Big|\partial_t\Big(w^m_{j, \pm}-w^0_{j, \pm}\Big)\Big|\leq M\epsilon (t-T_{\epsilon})^{\frac12},\quad  j\neq i,\label{eq:3.38}\\
&\Big|\partial_z\Big(w^m_{j, \pm}-w^0_{j, \pm}\Big)\Big|\leq M\epsilon (t-T_{\epsilon})^{\frac12}, \quad
j\neq i.\label{eq:3.39}
\end{align}

\end{lemma}

\subsection{The proof of Lemma $4.1$}

In this subsection, we will show the proof of  Lemma \ref{l:3.5}.

\begin{proof}
We apply the induction method to prove  Lemma \ref{l:3.5}. It is obvious that \eqref{eq:3.33}--\eqref{eq:3.39} are valid for $m=0$.
Suppose that these estimates  hold for $m$, then we need to establish the desired results for $m+1$. This
procedure is divided into the following six steps.

\vspace{0.3cm}
\underline{\textbf{Step 1. Estimate of $\sigma^m(t)$}}

From the expression of $\sigma^m(t)$, one has that for $t\in [T_{\ve}, T_{\ve}+1]$,
\begin{equation}\label{eq:3.40}
\begin{split}
|\sigma^m(t)-\sigma^0(t)|
\leq &C(|w^m_+-w^0_+|+|w^m_{-}-w^0_{-}|)\leq C_M \epsilon(t-T_{\epsilon}),
\end{split}
\end{equation}
where and below $C_M>0$ is a generic constant depending only on $M$.

\vskip 0.2 true cm
\underline{\textbf{Step 2. Estimates of
$w^{m+1}_{i, \pm}(z, t)$, $w^{m+1}_{j, +}(z, t)
|_{1\leq j\leq i-1}$ and $w^{m+1}_{j, -}(z, t)|_{i+1\leq j\leq n}$}}

\vspace{0.2cm}
Let $\displaystyle r(z, t)=w^{m+1}_{i, +}-w^0_{i, +}$, then $r(z, t)$ satisfies
\begin{equation}\label{eq:3.41}
\begin{cases}
\partial_t r+\Big(\lambda_i(w^m_+)-\sigma^m(t)\Big)\partial_z r=\Big(\lambda_i(w^0_+)-\lambda_i(w^m_+)
+\sigma^m(t)-\sigma^0(t)\Big)\partial_zw^0_{i, +}
-\displaystyle\sum_{k\neq i}\Big\{p_{ik}(w^m_{+})\times\\
\quad \Big(\partial_t(w^m_{k, +}-w^0_{k, +})
+(\lambda_i(w^m_+)-\sigma^m(t))\partial_z(w^m_{k, +}-w^0_{k, +})
-(\lambda_i(w^0_+)-\lambda_i(w^m_+)+\sigma^m-\sigma^0)\partial_zw^0_{k, +}\Big)\Big\}\\
\quad -\displaystyle\sum_{k\neq i}\Big(p_{ik}(w^m_+)-p_{ik}(w^0_+)\Big)\Big(\partial_tw^0_{k, +}+
(\lambda_i(w^0_{+})-\sigma^0(t))\partial_zw^0_{k, +}\Big),\\
r(z, T_{\epsilon})=0.
\end{cases}
\end{equation}

From Lemma \ref{l:3.3} and the inductive hypothesis, we can obtain that
\[
\begin{split}
&|\Big(\lambda_i(w^0_+)-\lambda_i(w^m_+)
+\sigma^m(t)-\sigma^0(t)\Big)\partial_zw^0_{i, +}|\\
\leq &C_M\epsilon^2(t-T_{\epsilon})
\Big((t-T_{\epsilon})^3+(x-x_{\epsilon}-
\lambda_i(w(x_{\epsilon}, T_{\epsilon}))(t-T_{\epsilon}))^2\Big)^{-\frac13}\leq C_M\epsilon^2.\\
\end{split}
\]
In addition,
\[
\begin{split}
&-\sum_{k\neq i}p_{ik}(w^m_{+})\Big(\partial_t(w^m_{k, +}-w^0_{k, +})
+(\lambda_i(w^m_+)-\sigma^m(t))\partial_z(w^m_{k, +}-w^0_{k, +})\Big)\\
=&-\sum_{k\neq i}\Big(\partial_t+(\lambda_i(w^m_+)-\sigma^m(t))\partial_z\Big)\Big(p_{ik}(w^m_{+})(w^m_{k, +}-w^0_{k, +})\Big)\\
\end{split}
\]

\[
\begin{split}
&\quad +\sum_{k\neq i}\sum^n_{j=1}\Big(\partial_{w_j}p_{ik}(w^m_+)\big(\partial_t
w^m_{j, +}+(\lambda_i(w^m_+)-\sigma^m(t))\partial_zw^m_{j, +}\big)\Big)(w^m_{k, +}-w^0_{k, +})\\
\end{split}
\]
and
\[
\begin{split}
&|\partial_tw^m_{j, +}+(\lambda_i(w^m_+)-\sigma^m(t))\partial_zw^m_{j, +}|\\
=&|\partial_t(w^m_{j, +}-w^0_{j, +})+\Big(\lambda_i(w^m_{+})-\lambda_i(w^0_+)-\sigma^m(t)+\sigma^0(t)\Big)
\partial_z(w^m_{j, +}-w^0_{j, +})\\
&+\partial_tw^0_{j, +}+(\lambda_i(w^0_+)-\sigma^0(t))\partial_zw^0_{j,+}+\Big(\lambda_i(w^m_{+})-\lambda_i(w^0_+)
-\sigma^m(t)+\sigma^0(t)\Big)
\partial_zw^0_{j, +}\\
&+(\lambda_i(w^0_+)-\sigma^0(t))\partial_z(w^m_{j, +}-w^0_{j, +})|\\
\leq &\begin{cases}
M\epsilon(t-T_\epsilon)^{\frac{1}{2}}+C_M\epsilon^2(t-T_\epsilon)^{\frac32}
+C_1\epsilon+C_M\epsilon^2((t-T_\epsilon)^3+z^2)^{\frac{1}{6}}, \qquad\;\; j\neq i,\\[3pt]
C_M\epsilon((t-T_\epsilon)^3+z^2)^{-\frac{1}{6}}
+C_M\epsilon^2(t-T_\epsilon)((t-T_\epsilon)^3+z^2)^{-\frac{1}{3}}+C_M\epsilon^2, \quad\;\; j=i,
\end{cases}
\end{split}
\]
where we have used the fact that
\[
\begin{split}
|\partial_tw^0_{i, +}(z, t)|=&\big|(\partial_t+\lambda_i(w^0_{+})\partial_x)w^0_{i, +}(z, t)+\Big(\sigma^m(t)-\lambda_i(w^0_{+})\Big)\partial_xw^0_{i, +}(z, t)\big|\\
\leq&C\epsilon\Big((t-T_{\epsilon})^3+(x-x_{\epsilon}-
\lambda_i(w(x_{\epsilon},T_{\epsilon}))(t-T_{\epsilon}))^2\Big)^{-\frac16}\\
&+C\epsilon^2(t-T_{\epsilon})^{\frac12}
\Big((t-T_{\epsilon})^3+(x-x_{\epsilon}-
\lambda_i(w(x_{\epsilon}, T_{\epsilon}))(t-T_{\epsilon}))^2\Big)^{-\frac13}\\
\leq&C_M\epsilon((t-T_{\epsilon})^3+z^2)^{-\frac16}
\end{split}
\]
and
\[
\begin{split}
|\lambda_i(w^0_+)-\sigma^0(t)|
\leq& C\sum^n_{l=1}\Big(|w^0_{l, +}(z, t)-w^0_{l, +}(0+, t)|+|w^0_{l, +}(z, t)-w^0_{l, -}(0-, t)|\Big)\\
\leq& C\sum_{l=1}^n\Big(|\partial_zw^0_{l, +}||z|+|w^0_{l, +}(0+, t)-w^0_{l, +}(0-, t)|\Big)\\
\leq &C_M\epsilon\Big((t-T_{\epsilon})^3+z^2\Big)^{\frac16}.
\end{split}
\]
Then
\[
|\partial_tw^m_{j, +}+(\lambda_i(w^m_+)-\sigma^m(t))\partial_zw^m_{j, +}|\leq
\begin{cases}
C_M\epsilon, \qquad\quad\qquad\qquad\quad  j\neq i,\\
C_M\epsilon\Big((t-T_{\epsilon})^3+z^2\Big)^{-\frac16}, \quad\; j=i
\end{cases}
\]
and
\[
\begin{split}
&\big|\sum_{k\neq i}p_{ik}(w^m_+)\Big(\lambda_i(w^0_+)-\lambda_i(w^m_+)+\sigma^m(t)-\sigma^0(t)\Big)\partial_zw^0_{k, +}\\
&-\sum_{k\neq i}\Big(p_{ik}(w^m_{+})-p_{ik}(w^0_+)\Big)\Big(\partial_tw^0_{k, +}+(\lambda_i(w^0_+)-\sigma^0(t))
\partial_zw^0_{k, +}\Big)\big|\\
\leq & C_M\epsilon(t-T_{\epsilon})\Big(\epsilon+\epsilon^2((t-T_{\epsilon})^3+z^2)^{\frac16}\Big)\leq C_M\epsilon^2(t-T_{\epsilon}).
\end{split}
\]
Integrating \eqref{eq:3.41} along the characteristics yields that for small $\ve>0$,
\[
\begin{split}
|r(z, t)|\leq&\sum_{k\neq i}\Big|p_{ik}(w^m_{+})(w^m_{k, +}-w^0_{k, +})\Big|
+C_M\epsilon^2\int^{t}_{T_{\epsilon}}(1+s-T_{\epsilon})\text{d}s\\
\leq& C_M\epsilon^2(t-T_{\epsilon})\le M\epsilon(t-T_{\epsilon}),
\end{split}
\]
where we have used the fact that  $p_{ik}(0)=0$ $(k\neq i)$. Similarly, one can prove the estimate
\eqref{eq:3.34} for $|w^{m+1}_{i, -}-w^0_{i, -}|$ by continuous induction.

\vspace{0.2cm}
Let $\displaystyle \mu_j(z, t)=w^{m+1}_{j, +}-w^0_{j, +}$ with $j=1, \cdots, i-1$, then $\mu_j(z, t)$ satisfies
\begin{equation}\label{eq:3.42}
\begin{cases}
\partial_t \mu_j+\Big(\lambda_j(w^m_+)-\sigma^m(t)\Big)\partial_z \mu_j=\Big(\lambda_j(w^0_+)-\lambda_j(w^m_+)
+\sigma^m(t)-\sigma^0(t)\Big)\partial_zw^0_{j, +}
-\displaystyle\sum_{k\neq i, j }\Big\{p_{jk}(w^m_{+})\times\\
\quad \Big(\partial_t(w^m_{k, +}-w^0_{k, +})
+(\lambda_j(w^m_+)-\sigma^m(t))\partial_z(w^m_{k, +}-w^0_{k, +})
-(\lambda_j(w^0_+)-\lambda_j(w^m_+)+\sigma^m-\sigma^0)\partial_zw^0_{k, +}\Big)\Big\}\\[5pt]
\quad -\displaystyle\sum_{k\neq i, j}\Big(p_{jk}(w^m_+)-p_{jk}(w^0_+)\Big)\Big(\partial_tw^0_{k, +}+
(\lambda_j(w^0_{+})-\sigma^0(t))\partial_zw^0_{k, +}\Big),\\[3pt]
\mu_j(z, T_{\epsilon})=0.
\end{cases}
\end{equation}
It follows from direct computation as in the treatment of \eqref{eq:3.41} that
\[
|\mu_j(z, t)|=|w^{m+1}_{j, +}-w^0_{j, +}|\leq \sum_{k\neq i, j}|p_{jk}(w^m_{+})||w^m_{k, +}-w^0_{k, +}|+
C_M\epsilon^2\int^t_{T_\epsilon}(\sqrt{s-T_\epsilon}+s-T_\epsilon)\text{d}s,
\]
which implies for small $\ve>0$,
\[
|\mu_j(z, t)|=|w^{m+1}_{j, +}-w^0_{j, +}|\leq C_M\epsilon^2 (t-T_{\epsilon})^{\frac32}
\leq M\epsilon (t-T_{\epsilon})^{\frac32}, \qquad j=1, \cdots, i-1.
\]
Similarly, we can also obtain that for small $\ve>0$,
\[
|w^{m+1}_{j, -}-w^0_{j, -}|\leq M\epsilon (t-T_{\epsilon})^{\frac32}, \qquad j=i+1, \cdots, n.
\]

\underline{\textbf{Step 3. Estimates of
$w^{m+1}_{j, -}(z, t)
|_{1\leq j\leq i-1}$, $w^{m+1}_{j, +}(z, t)|_{i+1\leq j\leq n}$}}

\vspace{0.3cm}
It suffices to establish the estimate of $w^{m+1}_{j, -}(z, t)
|_{1\leq j\leq i-1}$. For convenience, we still denote by
\[
\mu_j(z, t)=w^{m+1}_{j, -}(z, t)-w^{0}_{j, -}(z, t).
\]
Then one can formulate the problem of $\mu_j(z, t)$ by
\begin{equation}\label{eq:3.43}
\begin{cases}
\partial_t \mu_j+\Big(\lambda_j(w^m_{-})-\sigma^m(t)\Big)\partial_z \mu_j=\Big(\lambda_j(w^0_{-})-\lambda_j(w^m_{-})
+\sigma^m(t)-\sigma^0(t)\Big)\partial_zw^0_{j, -}
-\displaystyle\sum_{k\neq i, j }\Big\{p_{jk}(w^m_{-})\times\\
\quad\Big(\partial_t(w^m_{k, -}-w^0_{k, -})
+(\lambda_j(w^m_{-})-\sigma^m(t))\partial_z(w^m_{k, -}-w^0_{k, -})
-(\lambda_j(w^0_{-})-\lambda_j(w^m_{-})+\sigma^m-\sigma^0)\partial_zw^0_{k, -}\Big)\Big\}\\[5pt]
\quad-\displaystyle\sum_{k\neq i, j}\Big(p_{jk}(w^m_{-})-p_{jk}(w^0_{-})\Big)\Big(\partial_tw^0_{k, -}+
(\lambda_i(w^0_{-})-\sigma^0(t))\partial_zw^0_{k, -}\Big),\\[3pt]
\mu_j(z, T_{\epsilon})=0,\\
\mu_j(z, t)|_{z=0}=w^{m+1}_{j, -}(0-, t)-w^0_{j, -}(0-, t).
\end{cases}
\end{equation}

 Let $\xi=\xi(z, t; s)$ be the backward characteristics of \eqref{eq:3.43} through the point $(z, t)$ in the
 domain $\tilde{\Omega}_{-}$. If the characteristics $\xi=\xi(z, t; s)$ intersects with $z$-axis before $t$-axis,
 then we can obtain
 \[
 |\mu_j(z, t)|\leq C_M\epsilon^2(t-T_{\epsilon})^{\frac32}.
 \]
Otherwise, if $\xi=\xi(z, t; s)$ intersects with $t$-axis at the point $(0, s)$, and $s>T_{\epsilon}$, then
\begin{equation}\label{eq:3.43*}
|\mu_j(z, t)|\leq |w^{m+1}_{j, -}(0-, s)-w^0_{j, -}(0-, s)|+C_M\epsilon^2
(t-T_{\epsilon})^{\frac32}.
\end{equation}

Next, we estimate the term $|w^{m+1}_{j, -}(0-, s)-w^0_{j, -}(0-, s)|$ on the right hand side of \eqref{eq:3.43*}.
Firstly, we claim that
 \begin{equation}\label{eq:3.44*}
 [w^{m+1}_j]=\mathcal{F}_j(w_{1,+}(0+,s), \cdots, w_{i, +}(0+,s), w_{i, -}(0-,s), \cdots, w_{n, -}(0-,s))
 [w^{m+1}_i]^3, \quad j\neq i,
 \end{equation}
where $\mathcal{F}_j$ is smooth on its arguments.

In fact, it follows from \eqref{eq:3.27} that
 \begin{equation}\label{eq:3.44}
 \Big((\partial_{u_k}f_l)(w_{-})
 -\sigma\mathbb{I}_n\Big)(\p_wu)|_{w=w_{-}(0-, t)}
 \left(
 \begin{matrix}
[w_1]\\
[w_2]\\
\vdots\\
[w_n]
\end{matrix}
\right)=
\left(
 \begin{array}{c}
\displaystyle\sum^n_{i, j=1} Q^1_{ij}(w_{-}(0-, t), w_+(0+, t))[w_i][w_j]\\
\displaystyle\sum^n_{i, j=1} Q^2_{ij}(w_{-}(0-, t), w_+(0+, t))[w_i][w_j]\\
\vdots\\
\displaystyle\sum^n_{i, j=1} Q^n_{ij}(w_{-}(0-, t), w_+(0+, t))[w_i][w_j]
 \end{array}
\right).
 \end{equation}
Multiplying \eqref{eq:3.44} by
$(\p_wu)^{-1}|_{w=w_{-}(0-, t)}$ yields
\[
\begin{split}
&\left(
\begin{matrix}
a_{11}(w_{-})-\sigma&a_{12}(w_{-})&\cdots&0&\cdots&a_{1n}(w_{-})\\
\vdots\\
a_{i1}(w_{-})&a_{i2}(w_{-})&\cdots&\lambda_i(w_{-})-\sigma&\cdots&a_{in}(w_{-})\\
\vdots\\
a_{n1}(w_{-})&a_{n2}(w_{-})&\cdots&0&\cdots&a_{nn}(w_{-})-\sigma
\end{matrix}
\right)
\left(
\begin{matrix}
[w_1]\\
\vdots\\
[w_i]\\
\vdots\\
[w_n]
\end{matrix}
\right)\\
=&\left(
 \begin{array}{c}
\sum^n_{i, j=1} \bar{Q}^1_{ij}(w_{-}(0-, t), w_+(0+, t))[w_i][w_j]\\[5pt]
\sum^n_{i, j=1} \bar{Q}^2_{ij}(w_{-}(0-, t), w_+(0+, t))[w_i][w_j]\\[3pt]
\vdots\\
\sum^n_{i, j=1} \bar{Q}^n_{ij}(w_{-}(0-, t), w_+(0+, t))[w_i][w_j]
 \end{array}
\right),
\end{split}
\]
where $Q^l_{ij}(w_{-}(0-, t), w_{+}(0+, t))$ and  $\bar{Q}^l_{ij}(w_{-}(0-, t), w_{+}(0+, t))$ are smooth functions.
Thus, we obtain
\begin{equation}\label{eq:3.45}
[w_l]=\sum^n_{i, j=1}\tilde{Q}_{ij}(w_{-}(0-, t))[w_i][w_j]+\sum^n_{i, j, k=1}
Q^l_{ijk}(w_{-}(0-, t), w_+(0+, t))[w_i][w_j][w_k], \quad l\neq i,
\end{equation}
where $Q^l_{ijk}$ are smooth functions. Similarly, one can use the Taylor's formula to
Rankine-Hugoniot conditions \eqref{eq:3.27} at $w=
w_+(0+, t)$,  and get
\begin{equation}\label{eq:3.46}
[w_l]=-\sum^n_{i, j=1}\tilde{Q}_{ij}(w_+(0+, t))[w_i][w_j]+\sum^n_{i, j, k=1}
Q^l_{ijk}(w_{+}(0+, t), w_{-}(0-, t))[w_i][w_j][w_k].
\end{equation}
Summing \eqref{eq:3.45} and \eqref{eq:3.46} together yields that
\begin{equation}\label{eq:3.47}
[w_l]=\sum^n_{i, j, k=1}\tilde{Q}^l_{ijk}(w_{-}(0-, t), w_{+}(0+, t))
[w_i][w_j][w_k],
\end{equation}
where $\tilde{Q}^l_{ijk}$ are smooth. Let
\begin{equation}\label{eq:3.48}
[w_j]=\zeta_j[w_i]^3,\quad j\neq i.
\end{equation}
Substituting \eqref{eq:3.48} into \eqref{eq:3.47},
we obtain from the implicit function theorem that for $j\neq i$,
\[
\zeta_j=\mathcal{F}_j\Big(w_{1,+}(0+, t),\cdots, w_{i, +}(0+, t), w_{i, -}(0-, t), \cdots, w_{n, -}(0-, t)\Big),
\]
where $\mathcal{F}_j$ are smooth. By $w_{j,-}=w_{j, +}-[w_j]$ with $1\leq j\leq i-1$
and $w_{j, +}=w_{j, -}+[w_j]$ with $i+1\leq j\leq n$,
the claim \eqref{eq:3.44*} is shown.

On the other hand, one has
\begin{equation}\label{eq:4.26}
\begin{split}
&|w^{m+1}_{j, -}(0-, s)-w^0_{j, -}(0-, s)|\leq |w^{m+1}_{j, -}(0-, s)-w^{m+1}_{j, +}(0+, s)|+|w^{m+1}_{j, +}(0+, s)
-w^0_{j, +}(0+, s)|\\
&\qquad\qquad\qquad\qquad\qquad\quad\;+|w^0_{j, +}(0+, s)-w^0_{j, -}(0-, s)|\\
&\qquad\qquad\qquad\qquad\qquad=|[w^{m+1}_j]|+|w^{m+1}_{j, +}(0+, s)-w^0_{j, +}(0+, s)|+|[w^0_j]|,\\[3pt]
&\qquad\quad\qquad\qquad|[w^{m+1}_i]|\leq |w^{m+1}_{i, +}(0+, s)-w^0_{i, +}(0+, s)|+|w^0_{i, +}( 0+, s)-w^0_{i, -}(0-, s)|\\
&\qquad\quad\qquad\qquad\qquad\qquad+|w^0_{i, -}(0-, s)-w^{m+1}_{i, -}(0-, s)|\\
&\qquad\quad\qquad\qquad\qquad\quad\leq C_M\epsilon (t-T_{\epsilon})^{\frac12}.
\end{split}
\end{equation}
It follows from \eqref{eq:3.43*}, \eqref{eq:3.44*}, \eqref{eq:4.26} and Step $2$ that for small $\epsilon>0$ and $M>{\bar C}_0$,
\[
\begin{split}
&|w^{m+1}_{j, -}(0-, s)-w^0_{j, -}(0-, s)|\leq C_M\epsilon^3(t-T_{\epsilon})^{\frac32}+{\bar C}_0\epsilon(t-T_{\epsilon})^{\frac32},\\
&\Big|\mu_j(z, t)\Big|\leq  C_M\epsilon^2(t-T_{\epsilon})^{\frac32}+{\bar C}_0\epsilon(t-T_{\epsilon})^{\frac32}\leq M\epsilon (t-T_{\epsilon})^{\frac32}.
\end{split}
\]

Therefore, \eqref{eq:3.37} holds true for $w^{m+1}_{j,-}(z, t)|_{1\leq j\leq i-1}$, and the estimate for $w^{m+1}_{j,+}(z, t)|_{i+1\leq j\leq n}$ can be obtained similarly.

\vspace{0.1cm}
\underline{\textbf{Step 4. Estimates of
$\p_{t,z}(w^{m+1}_{i, \pm}-w^0_{i, \pm})$}}

\vspace{0.1cm}
For convenience, we still denote $r(z, t)=
\partial_z(w^{m+1}_{i, +}-w^0_{i, +})$ without confusions. Then
$r(z, t)$ satisfies
\begin{equation}\label{eq:3.52}
\begin{cases}
\quad\partial_t r+\Big(\lambda_i(w^m_+)-\sigma^m(t)\Big)\partial_z r+\partial_z\lambda_i(w^m_+)r\\
=\displaystyle\sum^n_{j=1}\Big(\partial_{w_j}\lambda_i(w^0_{+})\partial_zw^0_{j, +}-\partial_{w_j}\lambda_i(w^m_+)\partial_zw^m_{j,+}\Big)
\partial_zw^0_{i,+}
-\Big(\lambda_i(w^0_+)-\lambda_i(w^m_+)+\sigma^m-\sigma^0\Big)\partial^2_zw^0_{i, +}\\
\quad-\displaystyle\sum_{k\neq i}p_{ik}(w^m_{+})\cdot\Big\{\Big(\partial^2_{tz}(w^m_{k, +}-w^0_{k, +})
+(\lambda_i(w^m_+)-\sigma^m(t))\partial^2_z(w^m_{k, +}-w^0_{k, +})
-(\lambda_i(w^0_+)-\lambda_i(w^m_+)\\
\quad+\sigma^m-\sigma^0)\partial^2_zw^0_{k, +}\Big)\Big\}-\displaystyle\sum_{k\neq i}\sum_{j=1}^n p_{ik}(w^m_+)\Big(\partial_{w_j}\lambda_i(w^m_+)\partial_zw^m_{j, +}\partial_zw^m_{k, +}
-\partial_{w_j}\lambda_i(w^0_+)\partial_zw^0_{j, +}\partial_zw^0_{k, +}\Big)\\
\quad-\displaystyle\sum_{k\neq i}\Big(p_{ik}(w^m_+)-p_{ik}(w^0_+)\Big)\Big(\partial^2_{tz}w^0_{k, +}+
(\lambda_i(w^0_{+})-\sigma^0(t))\partial^2_zw^0_{k, +}+\sum^n_{j=1}
\partial_{w_j}\lambda_i(w^0_+)\partial_zw^0_{j, +}\partial_zw^0_{k, +}\Big)\\[3pt]
\quad\displaystyle-\sum_{k\neq i}\sum_{j=1}^n\Big(\partial_{w_j}p_{ik}(w^m_+)\partial_zw^m_{j, +}
-\partial_{w_j}p_{ik}(w^0_+)\partial_zw^0_{j,+}\Big)\Big(\partial_tw^0_{k, +}+(\lambda_i(w^0_+)-\sigma^0)\partial_z
w^0_{k, +}\Big)\\
\quad\displaystyle-\sum_{k\neq i}\sum_{j=1}^n\partial_{w_j}p_{ik}(w^m_+)\partial_zw^m_{j, +}\Big(\partial_t(w^m_{k, +}-w^0_{k, +})
+(\lambda_i(w^m_+)-\sigma^m)\partial_z(w^m_{k, +}-w^0_{k, +})\\
\quad-(\lambda_i(w^0_+)-\lambda_i(w^m_+)+\sigma^m-\sigma^0)\partial_zw^0_{k, +}\Big),\\
\quad r(z, T_{\epsilon})=0.
\end{cases}
\end{equation}

Note that
\begin{equation}\label{Y-005}
\begin{split}
&\big|\displaystyle\sum^n_{j=1}\Big(\partial_{w_j}\lambda_i(w^0_{+})\partial_zw^0_{j, +}-\partial_{w_j}\lambda_i(w^m_+)\partial_zw^m_{j,+}\Big)
\partial_zw^0_{i,+}\big|\\
=&\big|\sum_{j=1}^n\sum^n_{k=1}\partial^2_{w_jw_k}\lambda_i(\theta w^0_{+}
+(1-\theta)w^m_{+})\Big(w^0_{k, +}-w^m_{k, +}\Big)\partial_zw_{j, +}^0
\partial_zw_{i, +}^0-\sum^n_{j=1}\partial_{w_j}\lambda_i(w^m_{+})\partial_z
(w^m_{j, +}-w^0_{j, +})\partial_zw^0_{i, +}\big|\\
\leq &C_M\epsilon^2\Big((t-T_{\epsilon})^3+z^2\Big)^{-\frac12} \qquad\quad \text{for }0<\theta<1,\\
&\big|\Big(\lambda_i(w^0_+)-\lambda_i(w^m_+)+\sigma^m(t)-\sigma^0(t)\Big)\partial^2_zw^0_{i, +}\big|\leq C(|w^m_+-w^0_+|+|w^m_{-}-w^0_{-}|)|\partial^2_zw^0_{i, +}|\\
\leq &C_M\epsilon^2(t-T_{\epsilon})\Big((t-T_{\epsilon})^3+z^2\Big)^{-\frac56}.\\[5pt]
\end{split}
\end{equation}
In addition, we have
\begin{equation}\label{Y-006}
\begin{split}
&-\sum_{k\neq i}p_{ik}(w^m_+)\Big(\partial^2_{tz}(w^m_{k, +}-w^0_{k, +})
+(\lambda_i(w^m_+)-\sigma^m(t))\partial^2_z(w^m_{k, +}-w^0_{k, +})\Big)\\
=&-\sum_{k\neq i}\Big(\partial_t+(\lambda_i(w^m_{+})-\sigma^m(t))\partial_z\Big)\Big(p_{ik}(w^m_{+})\partial_z(w^m_{k, +}
-w^0_{k, +})\Big)+\sum_{k\neq i}\sum^n_{j=1}
\partial_{w_j}p_{ik}(w^m_+)\Big(\partial_tw^m_{j, +}\\
&\quad+(\lambda_i(w^m_{+})-\sigma^m(t))\partial_zw^m_{j, +}\Big)
\partial_z(w^m_{k, +}-w^0_{k, +})\\
\end{split}
\end{equation}
and
\begin{equation}\label{Y-007}
\begin{split}
&\big|\sum_{k\neq i}p_{ik}(w^m_+)(\lambda_i(w^0_+)-\lambda_i(w^m_+)+\sigma^m-\sigma^0)\partial^2_zw^0_{k, +}\big|
\leq C_M \epsilon^2(t-T_{\epsilon})
((t-T_{\epsilon})^3+z^2)^{-\frac12},\\
&\big|\sum_{k\neq i}\sum_{j=1}^n p_{ik}(w^m_+)\Big(\partial_{w_j}\lambda_i(w^m_+)\partial_zw^m_{j, +}\partial_zw^m_{k, +}
-\partial_{w_j}\lambda_i(w^0_+)\partial_zw^0_{j, +}\partial_zw^0_{k, +}\Big)\big|\\
\leq &\big|\sum_{k\neq i}\sum_{j=1}^n p_{ik}(w^m_{+})\Big\{\partial_{w_j}\lambda_i(w^m_+)\partial_z(w^m_{j, +}
-w^0_{j, +})\partial_z(w^m_{k, +}-w^0_{k, +})+\partial_{w_j}\lambda_i(w^m_+)\partial_zw^0_{k, +}\partial_z(w^m_{j, +}
-w^0_{j, +})\\
&+\partial_zw^0_{j, +}\partial_{w_j}\lambda_i(w^m_+)\partial_z(w^m_{k, +}-w^0_{k, +})
+\Big(\partial_{w_j}\lambda_i(w^m_+)-\partial_{w_j}\lambda_i(w^0_+)\Big)\partial_zw^0_{j, +}\partial_zw^0_{k, +}\Big\}\big|\\
\leq & C_M\epsilon^2\Big((t-T_{\epsilon})^3+z^2\Big)^{-\frac16},\\[3pt]
&\big|\sum_{k\neq i}\Big(p_{ik}(w^m_+)-p_{ik}(w^0_+)\Big)\Big(\partial^2_{tz}w^0_{k, +}+
(\lambda_i(w^0_{+})-\sigma^0(t))\partial^2_zw^0_{k, +}+\sum^n_{j=1}
\partial_{w_j}\lambda_i(w^0_+)\partial_zw^0_{j, +}\partial_zw^0_{k, +}\Big)\big|\\[3pt]
\leq&|w^m_{+}-w^0_{+}|\Big(C_M\epsilon\Big((t-T_{\epsilon})^3+z^2\Big)^{-\frac12}
+C_M\epsilon^2\Big((t-T_{\epsilon})^3+z^2\Big)^{-\frac13}
+C_M\epsilon \sum^n_{j=1}\Big|\partial_{w_j}\lambda_i(w^0_+)\partial_zw^0_{j, +}\Big|\Big)\\
\leq &C_M\epsilon^2(t-T_\epsilon)((t-T_{\epsilon})^3+z^2)^{-\frac12}+C_M\epsilon^3,
\end{split}
\end{equation}
where we have used the fact that
\[
|\partial_{z}w^0_{i, +}|\leq C\epsilon\Big((t-T_\epsilon)^3+z^2\Big)^{-\frac{1}{3}}, \;\;|\lambda_i(w^0_+)-\sigma^0(t)|\leq C\epsilon\Big((t-T_\epsilon)^3+z^2\Big)^{\frac16},
\] and for $k\neq i$,
\[
|\partial_z w^0_{k, +}|\leq C\epsilon,\; |\partial^2_{tz}w^0_{k, +}|\leq C\epsilon\Big((t-T_\epsilon)^3+z^2\Big)^{-\frac{1}{2}},\;
|\partial^2_{z}w^0_{k, +}|\leq C\epsilon\Big((t-T_\epsilon)^3+z^2\Big)^{-\frac{1}{2}}.
\]
On the other hand, one can assert
\begin{equation}\label{Y-6}
\begin{split}
&\displaystyle\big|\sum_{k\neq i}\sum_{j=1}^n\Big(\partial_{w_j}p_{ik}(w^m_+)\partial_zw^m_{j, +}
-\partial_{w_j}p_{ik}(w^0_+)\partial_zw^0_{j,+}\Big)\Big(\partial_tw^0_{k, +}+(\lambda_i(w^0_+)-\sigma^0(t))\partial_z
w^0_{k, +}\Big)\big|\\
&\leq C_M\epsilon^2\Big((t-T_{\epsilon})^3+z^2\Big)^{-\frac16}.
\end{split}
\end{equation}
Indeed, for $j\neq i$,
\[
\begin{split}
&|\partial_{w_j}p_{ik}(w^m_+)\partial_zw^m_{j, +}
-\partial_{w_j}p_{ik}(w^0_+)\partial_zw^0_{j,+}|\\
=&|\partial_{w_j}p_{ik}(w^m_+)\partial_z(w^m_{j, +}-w^0_{j, +})
+\big(\partial_{w_j}p_{ik}(w^m_+)-\partial_{w_j}p_{ik}(w^0_+)\big)\partial_zw^0_{j, +}|\\
\leq &C_M\epsilon (t-T_\epsilon)^{\frac{1}{2}}+C_M\epsilon^2(t-T_\epsilon)\leq C_M\epsilon(t-T_\epsilon)^{\frac{1}{2}},
\end{split}
\]
for $j=i$,
\[
\begin{split}
&|\partial_{w_j}p_{ik}(w^m_+)\partial_zw^m_{j, +}
-\partial_{w_j}p_{ik}(w^0_+)\partial_zw^0_{j,+}|\\
\leq &C\epsilon \Big((t-T_\epsilon)^3+z^2\Big)^{-\frac{1}{6}}+C_M\epsilon^2(t-T_\epsilon)\Big((t-T_\epsilon)^3+z^2\Big)^{-\frac{1}{3}}
\leq  C_M\epsilon \Big((t-T_\epsilon)^3+z^2\Big)^{-\frac{1}{6}},
\end{split}
\]
meanwhile,
\[
|\partial_tw^0_{k, +}+(\lambda_i(w^0_{+})-\sigma^0(t))\partial_zw^0_{k, +}|
\leq C_M\epsilon+C_M\epsilon^2((t-T_\epsilon)^3+z^2)^{\frac{1}{6}}\leq C_M\epsilon, \quad \text{for } k\neq i.
\]
Collecting these estimates yields \eqref{Y-6}.

In addition,
\begin{equation}\label{Y-008}
\begin{split}
&\displaystyle\big|\sum_{k\neq i}\sum_{j=1}^n\partial_{w_j}p_{ik}(w^m_+)\partial_zw^m_{j, +}\Big(\partial_t(w^m_{k, +}
-w^0_{k, +})
+(\lambda_i(w^m_+)-\sigma^m)\partial_z(w^m_{k, +}-w^0_{k, +})\\
&\qquad-(\lambda_i(w^0_+)-\lambda_i(w^m_+)+\sigma^m-\sigma^0)\partial_zw^0_{k, +}\Big)\big|\\
&\leq C_M\epsilon^2(t-T_{\epsilon})^{\frac12}\Big((t-T_{\epsilon})^3+z^2\Big)^{-\frac13},
\end{split}
\end{equation}
where the following facts are used
\[
\begin{split}
|\partial_zw^m_{j, +}|\le&|\partial_z(w^m_{j, +}-w^0_{j, +})|+|\partial_zw^0_{j, +}|
\leq\begin{cases}
M\epsilon(t-T_\epsilon)^{\frac12}+C\epsilon\leq C_M\epsilon, \qquad j\neq i,\\
C_M\epsilon
((t-T_\epsilon)^3+z^2)^{-\frac{1}{3}}, \qquad\quad\;\; j=i,
\end{cases}
\end{split}
\]
\[
|(\lambda_i(w^0_+)-\lambda_i(w^m_+)+\sigma^m(t)-\sigma^0(t))\partial_zw^0_{k, +}|\leq C_M\epsilon^2(t-T_\epsilon)
\quad \text{for }k\neq i,
\]
and
\[
\begin{split}
|\lambda_i(w^m_+)-\sigma^m(t)|
\leq &|\lambda_i(w^m_+)-\lambda_i(w^0_+)|+|\sigma^m(t)-\sigma^0(t)|+|\lambda_i(w^0_+)
-\sigma^0(t)|\\
\leq & C_M\epsilon((t-T_\epsilon)^3+z^2)^{\frac{1}{6}}.
\end{split}
\]

Let $\xi^{m+1}_i=\xi^{m+1}_i(z, t; s)$ be the backward characteristics of \eqref{eq:3.52} through the point $(z, t)$, namely,
\[
\begin{cases}
\frac{\text{d}\xi^{m+1}_i}{\text{d}s}=\lambda_i(w^m_+(\xi^{m+1}_i, s))-\sigma^m(s), \qquad T_{\epsilon}\leq s\leq t,\\[5pt]
\xi^{m+1}_i|_{s=t}=z.
\end{cases}
\]

Due to the genuinely nonlinear condition of \eqref{eq:2.88} with respect to $\lambda_i$, one can assume that for small $|w|$,
\[
\partial_{w_i}\lambda_i(w)>\frac{1}{2}\partial_{w_i}\lambda_i(0)>0.
\]

Motivated by the conclusions of  Lemmas $11.1$ and Lemma $7.4$ in \cite{Lebaud} for the $2\times 2$ $p-$system
with one constant Riemann invariant before blowup time,
we can show that there exists a
constant $C_M>0$ independent of $m$ and $\epsilon$ such that
\begin{equation}\label{eq:3.53}
(s-T_{\epsilon})^3+(\xi^{m+1}_i)^2\geq C_M((t-T_{\epsilon})^3+z^2)
\end{equation}
and
\begin{equation}\label{eq:3.54}
\int^t_{T_{\epsilon}}|\partial_z\lambda_i(w^m_+)(\xi^{m+1}_i, s)|\text{d}s\leq ln\frac32+C_M\sqrt{t-T_{\epsilon}},
\end{equation}
see the Appendix for details.

Integrating both sides of \eqref{eq:3.52} along the characteristics, and combining with \eqref{Y-005}-\eqref{eq:3.54},  we
arrive at
\[
\begin{split}
|r(z, t)|\leq&\sum_{k\neq i}|p_{ik}(w^m_+)\partial_z(w^m_{k, +}-w^0_{k, +})|+\int^t_{T_\epsilon}|\partial_z\lambda_i(w^m_+)(\xi^{m+1}_i, s)||r(\xi^{m+1}_i, s)|\text{d}s\\
+&C_M\epsilon^2\int^t_{T_\epsilon}\Big((s-T_{\epsilon})((t-T_\epsilon)^3+z^2)^{-\frac56}
+(s-T_{\epsilon})^{\frac12}((t-T_{\epsilon})^3+z^2)^{-\frac13}+((t-T_\epsilon)^3+z^2)^{-\frac12}\Big)\text{d}s\\
\leq &C_M \epsilon^2((t-T_\epsilon)^3+z^2)^{-\frac16}+\int^t_{T_\epsilon}|\partial_z\lambda_i(w^m_+)(\xi^{m+1}_i, s)||r(\xi^{n+1}_i, s)|\text{d}s.
\end{split}
\]
Together with the Gronwall's inequality, this yields
\[
|r(z, t)|\leq C_M\epsilon^2((t-T_\epsilon)^3+z^2)^{-\frac16}.
\]
On the other hand, it follows from \eqref{eq:3.41} and the inductive hypothesis that for small $\ve>0$,
\[
\begin{split}
|\partial_t(w^{m+1}_{i, +}-w^0_{i, +})|
\leq  C_M\epsilon^2((t-T_\epsilon)^3+z^2)^{-\frac16}\leq M\epsilon((t-T_\epsilon)^3+z^2)^{-\frac16}.
\end{split}
\]
Therefore, \eqref{eq:3.35} is obtained.

\vspace{0.3cm}
\underline{\textbf{Step 5. Estimates of
$\p_{t,z}(w^{m+1}_{j, +}-w^0_{j, +}) (1\leq j\leq i-1)$ and $\p_{t,z}(w^{m+1}_{j, -}-w^0_{j, -})(i+1\leq j\leq n)$}}

\vspace{0.2cm}
For convenience, we still denote $\mu_j(z, t)=
\partial_z(w^{m+1}_{j, +}-w^0_{j, +})(1\leq j\leq i-1)$ without confusions. Then
$\mu_j(z, t)$ satisfies
\begin{equation}\label{eq:3.55}
\begin{cases}
\;\;\;\partial_t \mu_j+\Big(\lambda_j(w^m_+)-\sigma^m(t)\Big)\partial_z \mu_j+\partial_z\lambda_j(w^m_+)\mu_j\\
=\displaystyle\sum^n_{l=1}\Big(\partial_{w_l}\lambda_j(w^0_{+})\partial_zw^0_{l, +}-\partial_{w_l}\lambda_j(w^m_+)\partial_zw^m_{l,+}\Big)
\partial_zw^0_{j,+}
+\Big(\lambda_j(w^0_+)-\lambda_j(w^m_+)+\sigma^m-\sigma^0\Big)\partial^2_zw^0_{j, +}\\[3pt]
\;\;\;-\displaystyle\sum_{k\neq i, j}p_{jk}(w^m_{+})\cdot
\Big\{\Big(\partial^2_{tz}(w^m_{k, +}-w^0_{k, +})
+(\lambda_j(w^m_+)-\sigma^m(t))\partial^2_z(w^m_{k, +}-w^0_{k, +})
-(\lambda_j(w^0_+)-\lambda_j(w^m_+)\\[5pt]
\;\;\;+\sigma^m-\sigma^0)\partial^2_zw^0_{k, +}\Big)\Big\}-\displaystyle\sum_{k\neq i, j}\sum_{l=1}^n p_{jk}(w^m_+)\Big(\partial_{w_l}\lambda_j(w^m_+)\partial_zw^m_{l, +}\partial_zw^m_{k, +}
-\partial_{w_l}\lambda_j(w^0_+)\partial_zw^0_{l, +}\partial_zw^0_{k, +}\Big)\\
\;\;\;\displaystyle-\sum_{k\neq i, j}\sum_{l=1}^n\partial_{w_l}p_{jk}(w^m_+)\partial_zw^m_{l, +}\Big(\partial_t(w^m_{k, +}-w^0_{k, +})
\;\;\;+(\lambda_j(w^m_+)-\sigma^m)\partial_z(w^m_{k, +}-w^0_{k, +})\\
\qquad-(\lambda_j(w^0_+)-\lambda_j(w^m_+)+\sigma^m-\sigma^0)\partial_zw^0_{k, +}\Big)\\
\;\;\;-\displaystyle\sum_{k\neq i, j}\Big(p_{jk}(w^m_+)-p_{jk}(w^0_+)\Big)\Big(\partial^2_{tz}w^0_{k, +}+
(\lambda_j(w^0_{+})-\sigma^0(t))\partial^2_zw^0_{k, +}+\sum^n_{l=1}
\partial_{w_l}\lambda_j(w^0_+)\partial_zw^0_{l, +}\partial_zw^0_{k, +}\Big)\\[3pt]
\;\;\;\displaystyle-\sum_{k\neq i, j}\sum_{l=1}^n\Big(\partial_{w_l}p_{jk}(w^m_+)\partial_zw^m_{l, +}
-\partial_{w_l}p_{jk}(w^0_+)\partial_zw^0_{l,+}\Big)\Big(\partial_tw^0_{k, +}+(\lambda_j(w^0_+)-\sigma^0)\partial_z
w^0_{k, +}\Big),\\
\;\;\;\mu_j(z, T_{\epsilon})=0.
\end{cases}
\end{equation}

Let $\xi^{m+1}_j=\xi^{m+1}_j(z, t; s)$ be the backward $j$-th characteristics of the system \eqref{eq:3.55} through the point
$(z, t)$, satisfying
\[
\begin{cases}
\frac{\text{d}\xi^{m+1}_j}{\text{d}s}=\lambda_j(w^m_+)(\xi^{m+1}_j, s)-\sigma^m(s), \qquad T_{\epsilon}\leq s\leq t,\\[5pt]
\xi^{m+1}_j|_{s=t}=z.
\end{cases}
\]
Owing to the entropy conditions \eqref{eq:3.28} and the strictly hyperbolic condition \eqref{eq:1.3}, one has
\[
\xi^{m+1}_j=z+\int^t_s(\sigma^m(s)-\lambda_j(w^m_+)(\xi^{m+1}_j,s))\text{d}s
\geq z+\frac{1}{2}|\lambda_j(0)-\lambda_i(0)|(t-s).
\]
This yields
\begin{equation}\label{eq:3.56}
(\xi^{m+1}_j)^2\geq z^2+\frac{1}{4}|\lambda_j(0)-\lambda_i(0)|^2(t-s)^2.
\end{equation}

As shown in Step 4, by integrating the equation \eqref{eq:3.55} along the characteristics and making the related
estimates for the terms on the right hand side of \eqref{eq:3.55}, we arrive at
\[
\begin{split}
&|\mu_j(z, t)|\leq \sum_{k\neq i, j}|p_{jk}(w^m_+)\partial_z(w^m_{k, +}-w^0_{k, +})|+C_M\epsilon\int^t_{T_\epsilon}|\mu_j(\xi^{m+1}_j, s)|\Big((s-T_\epsilon)^3+(\xi^{m+1}_j)^2\Big)^{-\frac13}\text{d}s\\
&+C_M\epsilon^2\int^s_{T_\epsilon}\Big\{\frac{s-T_{\epsilon}}{\Big((s-T_\epsilon)^3+
(\xi^{m+1}_j)^2\Big)^{\frac12}}+\frac{(s-T_{\epsilon})^{\frac12}}{\Big((s-T_\epsilon)^3+
(\xi^{m+1}_j)^2\Big)^{\frac13}}+\frac{1}{\Big((s-T_\epsilon)^3+
(\xi^{m+1}_j)^2\Big)^{\frac16}}\Big\}\text{d}s\\
\leq &C_M\epsilon^2(t-T_\epsilon)^{\frac12}+C_M\epsilon\int^t_{T_\epsilon}|\mu_j( \xi^{m+1}_j, s)|(t-s)^{-\frac23}\text{d}s.
\end{split}
\]
Together with Gronwall's inequality, this yields that for small $\ve>0$,
\[
|\partial_z(w^{m+1}_{j, +}-w^0_{j, +})|\leq C_M\epsilon^2(t-T_\epsilon)^{\frac12}\le M\epsilon(t-T_\epsilon)^{\frac12}.
\]
Therefore, for $1\leq j\leq i-1$, it follows from \eqref{eq:3.41} and direct computation that for small $\ve>0$,
\[
\begin{split}
|\partial_t(w^{m+1}_{j, +}-w^0_{j, +})|
\leq C_M\epsilon^2(t-T_{\epsilon})^{\frac12}\le M\epsilon(t-T_{\epsilon})^{\frac12},
\end{split}
\]
where we have used the facts of $p_{jk}(0)|_{k\neq i, j}=0$, $|\partial_z(w^{m+1}_{j, +}-w^0_{j, +})|
\leq C_M\epsilon^2(t-T_{\epsilon})^{\frac12}$ and
\[
\begin{split}
&|(\lambda_j(w^m_+)-\sigma^m(t))\partial_z(w^{m+1}_{j, +}-w^0_{j, +})|\leq
C_M\epsilon^2(t-T_{\epsilon})^{\frac12},\\[3pt]
&|(\lambda_j(w^0_+)-\lambda_j(w^m_+)+\sigma^m(t)-\sigma^0(t))\partial_zw^0_{j, +}|\leq C_M\epsilon^2(t-T_{\epsilon}),\\[3pt]
&\big|\sum_{k\neq i, j}p_{jk}(w^m_+)\Big(\partial_t(w^m_{k, +}-w^0_{k, +})
+(\lambda_j(w^m_+)-\sigma^m(t))\partial_z(w^m_{k, +}-w^0_{k, +})\Big)\big|
\leq C_M\epsilon^2(t-T_\epsilon)^{\frac12},\\
&\big|\sum_{k\neq i, j}p_{jk}(w^m_+)(\lambda_j(w^0_+)-\lambda_j(w^m_+)
+\sigma^m(t)-\sigma^0(t))\partial_zw^0_{k, +}\big|\leq C_M\epsilon^2(t-T_\epsilon),\\
&\big|\displaystyle\sum_{k\neq i, j}\Big(p_{jk}(w^m_+)-p_{jk}(w^0_+)\Big)\Big(\partial_tw^0_{k, +}+
(\lambda_j(w^0_{+})-\sigma^0(t))\partial_zw^0_{k, +}\Big)\big|\leq C_M\epsilon^2(t-T_\epsilon).
\end{split}
\]

\vspace{0.1cm}
\underline{\textbf{Step 6. Estimates of
$\p_{t,z}(w^{m+1}_{j, -}-w^0_{j, -})(1\leq j\leq i-1)$ and $\p_{t,z}(w^{m+1}_{j, +}-w^0_{j, +})(i+1\leq j\leq n)$}}

\vspace{0.1cm}

It suffices to estimate $\partial_t(w^{m+1}_{j, -}-w^0_{j, -})$ and
$\partial_z(w^{m+1}_{j, -}-w^0_{j, -})$ for $j=1, \cdots, i-1$. Let
\[
\mu_j(z, t)=\partial_t(w^{m+1}_{j, -}-w^0_{j, -}), \quad j=1, \cdots, i-1.
\]
Then
\begin{equation}\label{eq:3.57}
\begin{cases}
\;\;\;\partial_t \mu_j+(\lambda_j(w^m_{-})-\sigma^m(t))\partial_z \mu_j+\partial_t(\lambda_j(w^m_{-})
-\sigma^m(t))\partial_z(w^{m+1}_{j, -}-w^0_{j, -})\\[3pt]
=\partial_t(\lambda_j(w^0_{-})-\lambda_j(w^m_{-})+\sigma^m(t)-\sigma^0(t))\partial_z
w^0_{j, -}+(\lambda_j(w^0_{-})-\lambda_j(w^0_{-})+\sigma^m(t)-\sigma^0(t))
\partial^2_{tz}w^0_{j, -}\\[3pt]
\quad -\displaystyle\sum_{k\neq i, j}p_{jk}(w^m_{-})\Big(\partial^2_t(w^m_{k, -}-w^0_{k, -})
+(\lambda_j(w^m_{-})-\sigma^m(t))\partial^2_{zt}(w^m_{k, -}-w^0_{k, -})\\
\quad -(\lambda_j
(w^0_{-})-\lambda_j(w^m_{-})+\sigma^m(t)-\sigma^0(t))\partial^2_{zt}w^0_{k, -}\Big)\\[3pt]
\quad -\displaystyle\sum_{k\neq i, j}p_{jk}(w^m_{-})\Big(\partial_t(\lambda_j(w^m_{-})-\sigma^m(t))
\partial_z(w^m_{k, -}-w^0_{k, -})-\partial_t(\lambda_j(w^0_{-})-\lambda_j(w^m_{-})+\sigma^m-\sigma^0)
\partial_zw^0_{k, -}\Big)\\[3pt]
\quad -\displaystyle\sum_{k\neq i, j}\sum_{l=1}^n\partial_{w_l}p_{jk}(w_{-}^m)
\partial_tw^m_{l, -}\Big(\partial_t(w^m_{k, -}-w^0_{k, -})+(\lambda_j(w^m_{-})-\sigma^m(t))\partial_z(w^m_{k, -}-w^0_{k, -})\\[3pt]
\quad -(\lambda_j(w^0_{-})-\lambda_j(w^m_{-})+\sigma^m-\sigma^0)\partial_zw^0_{k, -}\Big)-\displaystyle\sum_{k\neq i, j}(p_{jk}(w^m_{-})-p_{jk}(w^0_{-}))\Big(\partial^2_tw^0_{k, -}+(\lambda_j(w^0_{-})\\
\quad-\sigma^0(t))\partial^2_{tz}w^0_{k, -}+\partial_t(\lambda_j
(w^0_{-})-\sigma^0(t))\partial_zw^0_{k, -}\Big)\\[3pt]
\quad -\displaystyle\sum_{k\neq i, j}\sum^n_{l=1}
\Big(\partial_{w_l}p_{jk}(w^m_{-})\partial_tw^m_{l, -}-\partial_{w_l}p_{jk}
(w^0_{-})\partial_tw^0_{l, -}\Big)\Big(\partial_tw^0_{k, -}+(\lambda_j(w^0_{-})
-\sigma^0(t))\partial_zw^0_{k, -}\Big),\\
\;\;\;\mu_j(z, T_\epsilon)=0,\\
\;\;\;\mu_j(z, t)|_{z=0}=\partial_t(w^{m+1}_{j, -}-w^0_{j, -})(0-, t).
\end{cases}
\end{equation}
Note that
\[
|\lambda_j(w^m_{-})-\sigma^m(t)|\geq \frac{1}{2}|\lambda_j(0)-\lambda_i(0)|>0.
\]
Let $\xi^{m+1}_j=\xi^{m+1}_j(z, t; s)$ be the backward $j$-th characteristics of the system \eqref{eq:3.57}
through the point $(z, t)$. If $\xi^{m+1}_j$ intersects with $z$-axis before it meets the $t$-axis, then as shown in Steps 4-5,
by integrating the equation \eqref{eq:3.57} along the characteristics and taking the related estimates for the terms
on the right hand side of  \eqref{eq:3.57}, we have
\[
\begin{split}
|\mu_j|\leq& \sum_{k\neq i, j}|p_{jk}(w^m_{-})\partial_t(w^m_{k, -}-w^0_{k, -})|+
\frac{2}{|\lambda_j(0)-\lambda_i(0)|}\int^t_{T_\epsilon}
|\partial_t\lambda_j(w^m_{-})-\partial_t\sigma^m(t)||\mu_j(\xi^{m+1}_j, s)|\text{d}s\\
&+C_M\epsilon^2\int^t_{T_\epsilon}(1+\frac{s-T_{\epsilon}}{\sqrt{(s-T_\epsilon)^3+z^2}}+\frac{1}{\sqrt{s-T_\epsilon}})\text{d}s\\
\leq &C_M\epsilon^2(t-T_\epsilon)^{\frac12}+\frac{2}{|\lambda_j(0)-\lambda_i(0)|}\int^t_{T_\epsilon}
|\partial_t\lambda_j(w^m_{-})-\partial_t\sigma^m(t)||\mu_j(\xi^{m+1}_j, s)|\text{d}s.
\end{split}
\]
Since
\[
\begin{split}
&|\partial_t\lambda_j(w^m_{-})-\partial_t\sigma^m(t)|\\
=&\Big|\sum^n_{l=1}\partial_{w_l}\lambda_j(w^m_{-})\partial_t w^m_{l,-}-\sum^n_{l=1}
\partial_{w_l}\lambda_i(\theta w^m_+(0+, t)+(1-\theta)w^m_{-}(0-, t))\Big(\theta\partial_t w^m_{l, +}(0+, t)\\
&\quad+(1-\theta)\partial_t
w^m_{l, -}(0-,t)\Big)\Big|\\
\leq& C_M\epsilon\Big((s-T_\epsilon)^3+(\xi^{m+1}_j)^2\Big)^{-\frac13}\leq C_M\epsilon (t-s)^{-\frac23}
\quad \text{for } 0<\theta<1,
\end{split}
\]
where we have used \eqref{eq:3.56} and the fact that $|\partial_t w^m_{l, -}|
\leq C\epsilon \Big((s-T_\epsilon)^3+(\xi^{m+1}_j)^2\Big)^{-\frac13}$,
then from Gronwall's inequality, this yields that for small $\ve>0$,
\[
|\partial_t(w^{m+1}_{j, -}-w^0_{j, -})|\leq C_M\epsilon^2(t-T_\epsilon)^{\frac12}\le M\epsilon(t-T_\epsilon)^{\frac12}.
\]
If $\xi^{m+1}_j=\xi^{m+1}_j(z, t; s)$ intersects with $t$-axis at $(0, s)$ with $s>T_\epsilon$, then
one can get
\begin{equation}\label{Y-7}
\begin{split}
|\mu_j|\leq &\sum_{k\neq i, j}|p_{jk}(w^m_{-})\partial_t(w^m_{k, -}-w^0_{k, -})|
+|\partial_s(w^{m+1}_{j, -}-w^0_{j, -})(0-, s)|\\
&+\frac{2}{|\lambda_j(0)-\lambda_i(0)|}\int^t_{T_\epsilon}
|\partial_t\lambda_j(w^m_{-})-\partial_t\sigma^m(t)||\mu_j(\xi^{m+1}_j, s)|\text{d}s\\
&+C_M\epsilon^2\int^t_{T_\epsilon}(1+\frac{s-T_{\epsilon}}{\sqrt{(s-T_\epsilon)^3+z^2}}+\frac{1}{\sqrt{s-T_\epsilon}})\text{d}s\\
\leq &C_M\epsilon^2(t-T_\epsilon)^{\frac12}+|\partial_s(w^{m+1}_{j, -}-w^0_{j, -})(0-, s)|
\\&+\frac{2}{|\lambda_j(0)-\lambda_i(0)|}\int^t_{T_\epsilon}
|\partial_t\lambda_j(w^m_{-})-\partial_t\sigma^m(t)||\mu_j(\xi^{m+1}_j, s)|\text{d}s.
\end{split}
\end{equation}
Next, we deal with the term $\partial_s(w^{m+1}_{j, -}-w^0_{j, -})(0-, s)$ in \eqref{Y-7}. Note that for $1\leq j\leq i-1$,
\[
|\partial_s(w^{m+1}_{j, -}(0-, s)-w^0_{j, -}(0-, s))|\leq |\partial_s(w^{m+1}_{j, +}(0+, s)-w^0_{j, +}(0+, s))|+
|\partial_s[w^{m+1}_j-w^0_j]|.
\]
Due to \eqref{eq:3.44*}, we hae that for $j=1, \cdots, i-1$,
\[
\begin{split}
&|\partial_s[w^{m+1}_j-w^0_j]|\\
=&\big|\partial_s\Big(\Big(\mathcal{F}_j(w^{m+1}_{1,+}, \cdots,w^{m+1}_{i, +}, w^{m+1}_{i, -}, \cdots, w^{m+1}_{n, -})-\mathcal{F}_j(w^{0}_{1,+}, \cdots,w^{0}_{i, +}, w^{0}_{i, -}, \cdots, w^{0}_{n, -})\Big)[w^{m+1}_i]^3\Big)\\
&+\partial_s(\mathcal{F}_j(w^0_{1, +}, \cdots,w^0_{i, +}, w^0_{i, -},\cdots, w^0_{n, -})([w^{m+1}_i]^3-[w^0_i]^3))\big|\\
\leq &C_M\epsilon^3(s-T_\epsilon)^{\frac12}.
\end{split}
\]
This yields
\[
|\partial_s(w^{m+1}_{j, -}(0-, s)-w^0_{j, -}(0-, s))|\leq C_M\epsilon^2(s-T_\epsilon)^{\frac{1}{2}},
\]
Together with Gronwall's inequality, one obtains from \eqref{Y-7} that
\[
|\mu_j(z, t)|\leq C_M\epsilon^2(t-T_\epsilon)^{\frac{1}{2}}\leq M \epsilon(t-T_\epsilon)^{\frac{1}{2}}.
\]
On the other hand, it follows  from \eqref{eq:3.43} and direct computation that  for small $\ve>0$,
\[
\begin{split}
&|\partial_z(w^{m+1}_{j, -}-w^0_{j, -})|\leq C_M \epsilon^2(t-T_\epsilon)^{\frac12}\le M \epsilon(t-T_\epsilon)^{\frac12}, \quad  j=1, \cdots, i-1.
\end{split}
\]
In conclusion, we complete the proof of Lemma \ref{l:3.5} by continuous induction.
\end{proof}

\section{Convergence of the approximate  shock solutions and proofs of Theorem \ref{T:2.14} and Theorem \ref{DY-1}}

In the section,  based on the uniform estimates of the approximate shock solutions $w^m_{\pm}$  in $\tilde{\Omega}_{\pm}$
and shock speed $\sigma^m$ in $[T_\epsilon, T_\epsilon+1]$ in Section 4, we now derive the convergence of the approximate solutions
for $[T_{\ve}, T_\epsilon+\delta_0]$ with $\delta_0>0$ being small. Denote by $\Omega_{\pm,\delta_0}=\Omega_{\pm}\cap\{(x,t): x\in\Bbb R, T_{\ve}\le t\le T_\epsilon+\delta_0\}$.

\begin{lemma}\label{l:3.6}
For sufficiently small $\epsilon>0$, there exists a constant $C_M>0$ independent of $\epsilon_0$ and $m$ such that
when $\delta_0>0$ is small,
\begin{align}
&\|\sigma^m(t)-\sigma^{m-1}(t)\|_{L^{\infty}[T_\epsilon, T_{\epsilon}+\delta_0]}\leq C_M\sum_{j=1}^n\|w^m_{j, \pm}-w^{m-1}_{j, \pm}\|_{L^{\infty}(\tilde{\Omega}_{\pm,\delta_0})}, \label{eq:3.59} \\
&\|w^{m+1}_{i, \pm}-w^m_{i, \pm}\|_{L^{\infty}(\tilde{\Omega}_{\pm,\delta_0})}+C_M\sum_{j\neq i}\|w^{m+1}_{j, \pm}-
w^m_{j, \pm}\|_{L^{\infty}(\tilde{\Omega}_{\pm,\delta_0})} \notag\\
\leq &(1-\epsilon)\Big(\|w^{m}_{i, \pm}-w^{m-1}_{i, \pm}\|_{L^{\infty}(\tilde{\Omega}_{\pm,\delta_0})}+C_M\sum_{j\neq i}\|w^{m}_{j, \pm}-
w^{m-1}_{j, \pm}\|_{L^{\infty}(\tilde{\Omega}_{\pm,\delta_0})}\Big), \label{eq:3.60}
\end{align}
where $\|w^{m+1}_{j, \pm}-
w^m_{j, \pm}\|_{L^{\infty}(\tilde{\Omega}_{\pm,\delta_0})}=\|w^{m+1}_{j, +}-
w^m_{j, +}\|_{L^{\infty}(\tilde{\Omega}_{+,\delta_0})}+\|w^{m+1}_{j, -}-
w^m_{j, -}\|_{L^{\infty}(\tilde{\Omega}_{-,\delta_0})}$.
\end{lemma}
\begin{remark}
Note that the number $1-\epsilon<1$. Then $w^m_{\pm}$  in $\tilde{\Omega}_{\pm,\delta_0}$
and $\sigma^m$ in $[T_\epsilon, T_\epsilon+\delta_0]$ are Cauchy sequences, respectively.
\end{remark}

\begin{proof}
From the expression of $\sigma(t)$ and Lemma \ref{l:3.5}, \eqref{eq:3.59} obviously holds.
In the sequel, we prove estimate \eqref{eq:3.60}. Let
\[
r(z, t)=w^{m+1}_{i, +}(z, t)-w^m_{i, +}(z, t).
\]
Then $r(z, t)$ satisfies
\begin{equation}\label{eq:3.61}
\begin{split}
\begin{cases}
\partial_t r+(\lambda_i(w^m_{+})-\sigma^m(t))\partial_z r=\Big(\lambda_i(w^{m-1}_{+})-\lambda_i(w^m_+)+\sigma^m(t)-\sigma^{m-1}(t)\Big)\partial_zw^m_{i, +}\\[5pt]
\quad \displaystyle-\sum_{k\neq i}p_{ik}(w^m_+)\Big\{\partial_t(w^m_{k, +}-w^{m-1}_{k, +})+(\lambda_i(w^m_{+})-\sigma^m(t))\partial_z(w^m_{k, +}-w^{m-1}_{k, +})\\
\quad+\Big(\lambda_i(w^m_+)-\lambda_i(w^{m-1}_+)-\sigma^m(t)+\sigma^{m-1}(t)\Big)\partial_zw^{m-1}_{k, +}\Big\}-
\displaystyle\sum_{k\neq i}\Big(p_{ik}(w^m_{+})-p_{ik}(w^{m-1}_+)\Big)\times\\
\quad\;\Big(\partial_tw^{m-1}_{k, +}+(\lambda_i(w^{m-1}_+)-\sigma^{m-1}(t))\partial_zw^{m-1}_{k, +}\Big),\\[5pt]
r(z, T_{\epsilon})=0.
\end{cases}
\end{split}
\end{equation}
Since the term  $\partial_z w^m_{i, +}$ is not integral along the characteristics of \eqref{eq:3.61}
by the estimate $|\partial_z w^m_{i, +}|\leq C_M\epsilon\Big((t-T_\epsilon)^3+z^2\Big)^{-\frac{1}{3}}$,
then we need to take more delicate analysis on the most singular part $\Big(\lambda_i(w^{m-1}_{+})-\lambda_i(w^m_+)+\sigma^m(t)-\sigma^{m-1}(t)\Big)\partial_zw^m_{i, +}$
(namely, the first term on the right hand side of \eqref{eq:3.61}),
and further make some appropriate decompositions or combinations of the related singular terms
to control their singularity orders of space-time near $(0,T_{\ve})$
so that the corresponding integrals along the characteristics are bounded.

It is observed that
\[
\partial_{w_i}\lambda_i(w^{m-1}_+)\partial_zw^{m-1}_{i, +}
=\partial_z(\lambda_i(w^{m-1}_+(z, t)))-\sum_{j\neq i}\partial_{w_j}\lambda_i(w^{m-1}_+)\partial_zw^{m-1}_{j, +},
\]
and for $j\neq i$,
\[
\partial_{w_j}\lambda_i(w^{m-1}_+)\partial_zw^{m-1}_{i, +}
=\frac{\partial_{w_j}\lambda_i(w^{m-1}_+)}{\partial_{w_i}\lambda_i(w^{m-1}_+)}
\Big(\partial_z(\lambda_i(w^{m-1}_+))-\sum_{k\neq i}\partial_{w_k}\lambda_i(w^{m-1}_+)\partial_zw^{m-1}_{k, +}\Big),
\]
here we have applied the genuinely nonlinear condition $\partial_{w_i}\lambda_i(w)>\frac{1}{2}\partial_{w_i}\lambda_i(0)>0$ for small $|w|$.

For the term $(\lambda_i(w^{m-1}_{+})-\lambda_i(w^m_+)\Big)\partial_zw^m_{i, +}$, we set
\begin{equation}\label{YH-15}
\Big(\lambda_i(w^{m-1}_{+})-\lambda_i(w^m_+)\Big)\partial_zw^m_{i, +}=\sum^7_{i=1} I_i,\\
\end{equation}
where
\[
\begin{split}
&I_1:=\sum^n_{j, k=1}\int^1_0\int^1_0(\partial^2_{w_jw_k}\lambda_i)\bigg(\theta_1(\theta w^{m-1}_{+}+(1-\theta)
w^m_+)+(1-\theta_1)w^m_+\bigg)\theta\text{d}\theta\text{d}\theta_1\\
&\qquad\quad\cdot(w^{m-1}_{j, +}-w^m_{j, +})
(w^{m-1}_{k, +}-w^{m}_{k, +})\partial_zw^m_{i, +},\\
&I_2:=\sum^n_{j=1}\Big((\partial_{w_j}\lambda_i)(w^m_+)-
(\partial_{w_j}\lambda_i)(w^{m-1}_+)\Big)\partial_zw^m_{i, +}(w^{m-1}_{j, +}-w^m_{j, +}),\\
&I_3:=\sum^n_{j=1}(\partial_{w_j}\lambda_i)(w^{m-1}_{+})\partial_z(w^m_{i, +}-w^{m-1}_{i, +})
(w^{m-1}_{j, +}-w^m_{j, +}),\\
&I_4:=\partial_z(\lambda_i(w^{m-1}_{+}))(w^{m-1}_{i, +}-w^m_{i, +}),\\
&I_5:=-\sum_{j\neq i}(\partial_{w_j}\lambda_i)(w^{m-1}_{+})\partial_zw^{m-1}_{j, +}(w^{m-1}_{i, +}-w^m_{i, +}),\\
&I_6:=\sum_{j\neq i}\frac{(\partial_{w_j}\lambda_i)(w^{m-1}_{+})}
{(\partial_{w_i}\lambda_i)(w^{m-1}_{+})}\partial_z\lambda_i(w^{m-1}_{+})\Big(
w^{m-1}_{j, +}-w^m_{j, +}\Big),\\
&I_7:=-\sum_{j\neq i, k\neq i}\frac{(\partial_{w_j}\lambda_i)(w^{m-1}_{+})}
{(\partial_{w_i}\lambda_i)(w^{m-1}_{+})}\partial_{w_k}\lambda_i(w^{m-1}_{+})\partial_zw^{m-1}_{k, +}\Big(
w^{m-1}_{j, +}-w^m_{j, +}\Big).\\
\end{split}
\]
Then based on the estimates in Section 4,  by the expressions of $I_1-I_7$, one has that
\begin{equation}\label{YH-7}
\begin{split}
&|(\lambda_i(w^{m-1}_{+})-\lambda_i(w^m_+))\partial_zw^m_{i, +}|\le \bigg(|\partial_z(\lambda_i(w^{m-1}_{+}))|
+\f{C_M\ve}{\sqrt{t-T_{\ve}}}\bigg)
|w_{i,+}^m-w_{i,+}^{m-1}|\\
&\qquad +C_M\Big(|\partial_z(\lambda_i(w^{m-1}_{+}))|
+\frac{\epsilon}{\sqrt{t-T_\epsilon}}\Big)\sum_{j\neq i}|w^m_{j, +}-w^{m-1}_{j, +}|.\\
\end{split}
\end{equation}
In addition, as shown in $(8.1.9)$ of Chapter VIII in  \cite{Dafermos}, we have that
\[
\begin{split}
\sigma^m(t)=&\frac{\lambda_i(w^m_{-}(0-,t))+\lambda_i(w^m_+(0+,t))}{2}+O(|w^m_{+}(0+,t)-w^m_{-}(0-,t)|^2)\\
=&\lambda_i(w^m_{-}(0-,t))+\frac12\sum^n_{k=1}
(\partial_{w_k}\lambda_i)(w^m_{-}(0-,t))[w^m_k]+O([w^m]^2)
\end{split}
\]
and
\[
\begin{split}
&(\sigma^m(t)-\sigma^{m-1}(t))\partial_zw^m_{i, +}\\
=&(\lambda_i(w^m_{-}(0-,t))-\lambda_i(w^{m-1}_{-}(0-,t)))\partial_zw^m_{i, +}
+\frac12\sum_{k=1}^n\Big(\partial_{w_k}\lambda_i(w^m_{-}(0-,t))[w^m_k]\\
&\quad-\partial_{w_k}\lambda_i(w^{m-1}_{-}(0-,t))[w^{m-1}_k]\Big)\partial_zw^m_{i, +}
+O(1)\Big([w^m]^2-[w^{m-1}]^2\Big)\partial_zw^m_{i, +},
\end{split}
\]
where the term $(\lambda_i(w^m_{-}(0-,t))-\lambda_i(w^{m-1}_{-}(0-,t)))\partial_zw^m_{i, +}$ can be treated analogously to $\Big(\lambda_i(w^{m}_{+})-\lambda_i(w^{m-1}_+)\Big)\partial_zw^m_{i, +}$ in \eqref{YH-15}. For example,
$$(\lambda_i(w^m_{-}(0-,t))-\lambda_i(w^{m-1}_{-}(0-,t)))\partial_zw^m_{i, +}=\sum^8_{j=1}J_j,$$
where
\[
\begin{split}
&J_1:=\sum^n_{j, k=1}\int^1_0\int^1_0\partial^2_{w_jw_k}\lambda_i\Big(\theta_1(\theta w^{m}_{-}(0-, t)+(1-\theta)
w^{m-1}_{-}(0-, t))+(1-\theta_1)w^{m-1}_{-}(0-, t)\Big)\theta\text{d}\theta\text{d}\theta_1\\
&\qquad\cdot(w^{m}_{j, -}(0-, t)-w^{m-1}_{j, -}(0-, t))
(w^{m}_{k, -}(0-, t)-w^{m-1}_{k, -}(0-, t))\partial_zw^m_{i, +},\\
&J_2:=\sum^n_{j=1}\Big(\partial_{w_j}\lambda_i(w^{m-1}_{-}(0-, t))-
\partial_{w_j}\lambda_i(w^{m-1}_+(0+, t))\Big)\Big(w^{m}_{j, -}(0-, t)-w^{m-1}_{j, -}(0-, t)\Big)\partial_zw^m_{i, +},\\
&J_3:=\sum^n_{j=1}\partial_{w_j}\lambda_i(w^{m-1}_{+}(0+, t))(w^{m}_{j, -}(0-, t)-w^{m-1}_{j, -}(0-, t))\partial_z(w^m_{i, +}-w^{m-1}_{i, +})
,\\
&J_4:=\partial_z(\lambda_i(w^{m-1}_{+}))(w^{m}_{i, -}(0-, t)-w^{m-1}_{i, -}(0-, t)),\\
&J_5:=-\sum_{j\neq i}\partial_{w_j}\lambda_i(w^{m-1}_{+})\partial_zw^{m-1}_{j, +}(w^{m}_{i, -}(0-, t)-w^{m-1}_{i, -}(0-, t)),\\
&J_6:=\sum_{j\neq i}\frac{\partial_{w_j}\lambda_i(w^{m-1}_{+})}
{\partial_{w_i}\lambda_i(w^{m-1}_{+})}\partial_z\lambda_i(w^{m-1}_{+})(
w^{m}_{j, -}(0-, t)-w^{m-1}_{j, -}(0-, t)),\\
&J_7:=-\sum_{j\neq i, k\neq i}\frac{\partial_{w_j}\lambda_i(w^{m-1}_{+})}
{\partial_{w_i}\lambda_i(w^{m-1}_{+})}\partial_{w_k}\lambda_i(w^{m-1}_{+})\partial_zw^{m-1}_{k, +}(
w^{m}_{j, -}(0-, t)-w^{m-1}_{j, -}(0-, t)),\\
&J_8:=\sum_{j=1}^n\Big(\partial_{w_j}\lambda_i(w^{m-1}_+(0+, t))-\partial_{w_j}\lambda_i(w^{m-1}_+(z, t))\Big)(w^m_{j, -}(0-, t)
-w^{m-1}_{j, -}(0-, t))\partial_z w^{m-1}_{i, +}.
\end{split}
\]
Analogously, we denote
\[
\begin{split}
&\frac{1}{2}\sum_{k=1}^n\Big(\partial_{w_k}\lambda_i(w^m_{-}(0-,t))[w^m_k]
-\partial_{w_k}\lambda_i(w^{m-1}_{-}(0-,t))[w^{m-1}_k]\Big)\partial_zw^m_{i, +}
+O(1)\Big([w^m]^2-[w^{m-1}]^2\Big)\partial_zw^m_{i, +}\\
&\quad =\sum^8_{j=1}L_j,\\
&(\sigma^m(t)-\sigma^{m-1}(t))\partial_zw^m_{i, +}=\sum^8_{j=1}(L_j+J_j),
\end{split}
\]
where
\[
\begin{split}
&L_1:=\frac12\sum_{k=1}^n\Big(\partial_{w_k}\lambda_i(w^m_{-}(0-, t))
-\partial_{w_k}\lambda_i(w^{m-1}_{-}(0-, t))\Big)[w^m_k-w^{m-1}_k]\partial_zw^m_{i, +}\\
&\qquad\;+O(1)\Big([w^m]^2-[w^{m-1}]^2\Big)\partial_zw^m_{i, +},\\
&L_2:=\frac{1}{2}\sum^n_{k=1}\partial_{w_k}\lambda_i(w^{m-1}_{-}(0-, t))
[w^m_k-w^{m-1}_k]\partial_z(w^m_{i, +}-w^{m-1}_{i, +}),\\
&L_3:=\frac12\sum^n_{k=1}\Big(\partial_{w_k}\lambda_i(w^m_{-}(0-, t))-\partial_{w_k}\lambda_i(w^{m-1}_{-}(0-, t))\Big)
[w^{m-1}_k]\partial_zw^m_{i, +}+\frac{1}{2}\sum^n_{k=1}\Big(\partial_{w_k}\lambda_i(w^{m-1}_{-}(0-, t))\\
&\qquad-\partial_{w_k}
\lambda_i(w^{m-1}_{+}(0+, t))\Big)[w^m_k-w^{m-1}_k]\partial_zw^{m-1}_{i, +},\\
&L_4:=\frac12\partial_z(\lambda_i(w^{m-1}_{+}))[w^m_i-w^{m-1}_i],\\
&L_5:=-\frac12\sum_{j\neq i}\partial_{w_j}\lambda_i
(w^{m-1}_{+})\partial_zw^{m-1}_{j, +}[w^m_i-w^{m-1}_i],\\
&L_6:=\frac12\sum_{j\neq i}\frac{\partial_{w_j}\lambda_i(w^{m-1}_{+})}{\partial_{w_i}\lambda_i(w^{m-1}_{+})}
\partial_z\lambda_i(w^{m-1}_{+})[w^m_j-w^{m-1}_j],\\
&L_7:=-\frac12\sum_{j\neq i, k\neq i}\frac{\partial_{w_j}\lambda_i(w^{m-1}_{+})}{\partial_{w_i}\lambda_i(w^{m-1}_{+})}
\partial_{w_k}\lambda_i(w^{m-1}_{+})\partial_zw^{m-1}_{k, +}[w^m_j-w^{m-1}_j],\\
&L_8:=\frac12\sum_{k=1}^n\Big(\partial_{w_k}\lambda_i(w^{m-1}_+(0+, t))-\partial_{w_k}\lambda_i(w^{m-1}_+(z, t))\Big)\partial_z w^{m-1}_{i, +}[w^m_k-w^{m-1}_k].
\end{split}
\]

Specially noting the good combination of $L_4+J_4$, it follows from direct computation that
\[
\begin{split}
|J_1+L_1|\leq&C_M\epsilon^2\sum^n_{k=1}(|[w^m_k-w^{m-1}_k]|+|w^m_{k, -}(0-, t)-w^{m-1}_{k, -}(0-, t)|)+\frac{C_M\epsilon^2}{\sqrt{t-T_\epsilon}}|[w^m-w^{m-1}]|\\
\leq&\frac{C_M\epsilon^2}{\sqrt{t-T_\epsilon}}\sum_{j=1}^n|w^m_{j, \pm}(0\pm, t)-w^{m-1}_{j, \pm}(0\pm, t)|,\\
|J_2+L_2|\leq &\sum_{j=1}^nC_M\epsilon^2(t-T_\epsilon)^{\frac12}\Big((t-T_\epsilon)^3+z^2)^{-\frac13}|w^{m-1}_{j, -}(0-, t)-w^m_{j, -}(0-, t)|
+\frac{C_M\epsilon}{\sqrt{t-T_\epsilon}}\sum^n_{k=1}|[w^m_k-w^{m-1}_k]|\\
\leq & \frac{C_M\epsilon}{\sqrt{t-T_\epsilon}}\sum^n_{j=1}|w^m_{j, \pm}(0\pm, t)-w^{m-1}_{j, \pm}(0\pm, t)|,\\
\end{split}
\]

\[
\begin{split}
|J_3+L_3|\leq &\frac{C_M\epsilon}{\sqrt{t-T_\epsilon}}\sum^n_{j=1}|w^m_{j, -}(0-, t)-w^{m-1}_{j, -}(0-, t)|
+\frac{C_M\epsilon^2}{\sqrt{t-T_\epsilon}}\sum_{k=1}^n|w^m_{k, +}(0+, t)-w^{m-1}_{k, +}(0+, t)|,\\
|J_4+L_4|\leq &\frac12|\partial_z(\lambda_i(w^{m-1}_+))|\Big(|w^{m}_{i, +}(0+, t)-w^{m-1}_{i, +}(0+, t)|+|w^{m}_{i, -}(0-, t)-w^{m-1}_{i, -}(0-, t)|\Big),\\
|J_5+L_5|
\leq& C_M\epsilon\Big(|w^{m}_{i, +}(0+, t)-w^{m-1}_{i, +}(0+, t)|+|w^{m}_{i, -}(0-, t)-w^{m-1}_{i, -}(0-, t)|\Big),\\
|J_6+L_6|
\leq& C_M|\partial_z(\lambda_i(w^{m-1}_{+}))|\sum_{j\neq i}\Big(|w^{m}_{j, +}(0+, t)-w^{m-1}_{j, +}(0+, t)|
+|w^{m}_{j, -}(0-, t)-w^{m-1}_{j, -}(0-, t)|\Big),\\[3pt]
|J_7+L_7|=& \frac12\sum_{j\neq i, k\neq i}|\frac{\partial_{w_j}\lambda_i(w^{m-1}_{+})}{\partial_{w_i}\lambda_i(w^{m-1}_{+})}
\partial_{w_k}\lambda_i(w^{m-1}_{+})\partial_zw^{m-1}_{k, +}([w^m_j-w^{m-1}_j]+w^m_{j, -}(0-, t)-w^{m-1}(0-, t))|\\[3pt]
\leq&  C_M\epsilon\sum_{j\neq i}\Big(|w^{m}_{j, +}(0+, t)-w^{m-1}_{j, +}(0+, t)|+|w^{m}_{j, -}(0-, t)-w^{m-1}_{j, -}(0-, t)|\Big),\\
|J_8+L_8|\leq&\frac{C_M\epsilon^2}{\sqrt{t-T_\epsilon}}\sum^n_{k=1}\sum^n_{j=1}|w^m_{j, \pm}(0\pm, t)-w^{m-1}_{j, \pm}(0\pm, t)|.
\end{split}
\]
Therefore, we have
\begin{equation}\label{YH-8}
\begin{split}
&|(\sigma^m(t)-\sigma^{m-1}(t))\partial_zw^m_{i, +}|\\
\le & \bigg(\f12|\partial_z(\lambda_i(w^{m-1}_{+}))|+\f{C_M\ve}{\sqrt{t-T_{\ve}}}\bigg)
|w_{i,\pm}^m-w_{i,\pm}^{m-1}|+C_M|\partial_z\lambda_i(w^{m-1}_+)|\sum_{j\neq i}|w^m_{j, \pm}-w^{m-1}_{j, \pm}|.
\end{split}
\end{equation}
By  the estimate \eqref{eq:3.54}, integrating $\eqref{eq:3.61}_1$ along the characteristics and noting $|[w^m_j-w^{m-1}_j]|\leq |w^m_{j, +}-w^{m-1}_{j, +}|+|w^m_{j, -}-w^{m-1}_{j, -}|$ yield
\begin{equation}\label{eq:3.62}
\begin{split}
&\|w^{m+1}_{i, +}-w^m_{i, +}\|_{L^{\infty}(\tilde{\Omega}_{+})}
\leq \Big(\ln\frac32+C_M\sqrt{t-T_\epsilon}\Big)\|w^m_{i, +}-w^{m-1}_{i, +}\|_{L^{\infty}(\tilde{\Omega}_{+})}\\
&\quad +\Big(\frac12\ln\frac32+C_M\sqrt{t-T_\epsilon}\Big)\|w^m_{i, \pm}-w^{m-1}_{i, \pm}\|_{L^{\infty}(\tilde{\Omega}_{\pm})}+C_M
\sum_{k\neq i}\|w^{m}_{k, \pm}-w^{m-1}_{k, \pm}\|_{L^{\infty}(\tilde{\Omega}_{\pm})}.
\end{split}
\end{equation}
Similarly, one has
\begin{equation}\label{eq:3.63}
\begin{split}
&\|w^{m+1}_{i, -}-w^m_{i, -}\|_{L^{\infty}(\tilde{\Omega}_{-})}
\leq \Big(\ln\frac32+C_M\sqrt{t-T_\epsilon}\Big)\|w^m_{i, -}-w^{m-1}_{i, -}\|_{L^{\infty}(\tilde{\Omega}_{-})}\\
&\qquad +\Big(\frac12\ln\frac32+C_M\sqrt{t-T_\epsilon}\Big)\|w^m_{i, \pm}-w^{m-1}_{i, \pm}\|_{L^{\infty}(\tilde{\Omega}_{\pm})}
+C_M
\sum_{k\neq i}\|w^{m}_{k, \pm}-w^{m-1}_{k, \pm}\|_{L^{\infty}(\tilde{\Omega}_{\pm})}.
\end{split}
\end{equation}
Summing up \eqref{eq:3.62} and \eqref{eq:3.63} derives
\begin{equation}\label{YH-10}
\begin{split}
&\|w^{m+1}_{i, \pm}-w^m_{i, \pm}\|_{L^{\infty}(\tilde{\Omega}_{\pm})}\\
\leq &\Big(2\ln\frac32+C_M\sqrt{t-T_\epsilon}\Big)\|w^m_{i, \pm}-w^{m-1}_{i, \pm}\|_{L^{\infty}(\tilde{\Omega}_{\pm})}
+C_M
\sum_{k\neq i}\|w^{m}_{k, \pm}-w^{m-1}_{k, \pm}\|_{L^{\infty}(\tilde{\Omega}_{\pm})}.
\end{split}
\end{equation}
Next, we estimate $w^{m+1}_{j, +}-w^m_{j, +}$, $1\leq j\leq i-1$. Let $\mu_j(z, t)=w^{m+1}_{j, +}-w^m_{j, +}$, then $\mu_j(z, t)$ satisfies
\begin{equation}\label{eq:3.65}
\begin{split}
\begin{cases}
\partial_t \mu_j+(\lambda_j(w^m_{+})-\sigma^m(t))\partial_z \mu_j=\Big(\lambda_j(w^{m-1}_{+})-\lambda_j(w^m_+)+\sigma^m(t)-\sigma^{m-1}(t)\Big)\partial_zw^m_{j, +}\\[3pt]
\;\;\displaystyle-\sum_{k\neq i, j}p_{jk}(w^m_+)\Big\{\partial_t(w^m_{k, +}-w^{m-1}_{k, +})+
(\lambda_j(w^m_{+})-\sigma^m(t))\partial_z(w^m_{k, +}-w^{m-1}_{k, +})+
\Big(\lambda_j(w^m_+)-\lambda_j(w^{m-1}_+)\\
\;\;-\sigma^m(t)+\sigma^{m-1}(t)\Big)\partial_zw^{m-1}_{k, +}\Big\}-
\displaystyle\sum_{k\neq i, j}\Big(p_{jk}(w^m_{+})-p_{jk}(w^{m-1}_+)\Big)
(\partial_tw^{m-1}_{k, +}+(\lambda_j(w^{m-1}_+)-\sigma^{m-1}(t))\partial_zw^{m-1}_{k, +}),\\[3pt]
\mu_j(z, T_{\epsilon})=0.
\end{cases}
\end{split}
\end{equation}
It follows from direct calculations that
\[
\begin{split}
&\big|\Big(\lambda_j(w^{m-1}_{+})-\lambda_j(w^m_{+})+\sigma^m(t)-\sigma^{m-1}(t)\Big)\partial_zw^m_{j, +}\big|\leq C_M\epsilon\Big(|w^m_{+}-w^{m-1}_{+}|+|w^m_{-}-w^{m-1}_{-}|\Big),\\
&\big|\sum_{k\neq i, j}p_{jk}(w^m_{+})\Big(\lambda_j(w^{m-1}_+)-\lambda_j(w^{m}_+)+\sigma^m(t)
-\sigma^{m-1}(t)\Big)\partial_zw^{m-1}_{k, +}\big|
\leq C_M\epsilon^2\Big(|w^m_{+}-w^{m-1}_{+}|+|w^m_{-}-w^{m-1}_{-}|\Big),\\
&\big|\sum_{k\neq i, j}\Big(p_{jk}(w^m_{+})-p_{jk}(w^{m-1}_+)\Big)
\Big(\partial_tw^{m-1}_{k, +}+(\lambda_j(w^{m-1}_+)-\sigma^{m-1}(t))\partial_zw^{m-1}_{k, +}\Big)\big|\leq C_M\epsilon
|w^m_{+}-w^{m-1}_{+}|,
\end{split}
\]
where we have used the fact of
\[
|\partial_tw^m_{l, +}+\Big(\lambda_j(w^m_+)-\sigma^m(t)\Big)\partial_zw^m_{l, +}|\leq
\begin{cases}
C_M\epsilon, \qquad\quad\qquad\qquad  \;\;\;l\neq i,\\[3pt]
C_M\epsilon\Big((t-T_{\epsilon})^3+z^2)^{-\frac13}, \quad l=i.
\end{cases}
\]
Note that
\[
\begin{split}
&-\sum_{k\neq i, j}p_{jk}(w^m_{+})\Big(\partial_t(w^m_{k, +}-w^{n-1}_{k, +})+
(\lambda_j(w^m_{+})-\sigma^m(t))\partial_z(w^m_{k, +}-w^{m-1}_{k, +})\Big)\\
=&-\sum_{k\neq i, j}\Big\{\partial_t(p_{jk}(w^m_{+})(w^m_{k, +}-w^{m-1}_{k, +}))+(\lambda_j(w^m_{+})-\sigma^m(t))
\partial_z\Big(p_{jk}(w^m_{+})(w^m_{k, +}-w^{m-1}_{k, +})\Big)\Big\}\\
&+\sum_{k\neq i, j}\sum^n_{l=1}\partial_{w_l}p_{jk}(w^m_{+})\Big(\partial_tw^m_{l, +}+
(\lambda_j(w^m_{+})-\sigma^m(t))\partial_zw^m_{l, +}\Big)(w^m_{k, +}-w^{m-1}_{k, +}).\\[3pt]
\end{split}
\]
Therefore, integrating $\eqref{eq:3.65}_1$ along the back $j$-th characteristics $\xi^{m+1}_j=\xi^{m+1}_j(z, t; s)$
of $\eqref{eq:3.65}_1$
through the point $(z, t)$ yields
\begin{equation}\label{Y-8}
\begin{split}
|\mu_j|\leq& C_M\epsilon(t-T_\epsilon)|w^m_{\pm}-w^{m-1}_{\pm}|+\sum_{k\neq i, j}|p_{jk}(w^m_{+})|
|w^m_{k, +}-w^{m-1}_{k, +}|\\
&+\sum_{k\neq i, j}\int^t_{T_\epsilon}\bigg(C_M\epsilon+C_M\epsilon\Big((s-T_\epsilon)^3+(\xi^{n+1}_j)^2\Big)^{-\frac13}\bigg)
|w^m_{k, +}-w^{m-1}_{k, +}|\text{d}s.
\end{split}
\end{equation}
In addition, by the estimate similar to \eqref{eq:3.56}, one has
\[
\int^t_{T_\epsilon}\frac{1}{\Big((s-T_\epsilon)^3+(\xi^{m+1}_j)^2\Big)^{\frac{1}{3}}}\text{d}s\leq C_M\int^t_{T_\epsilon}
\frac{1}{(t-s)^{\frac{2}{3}}}\text{d}s\leq C_M(t-T_\epsilon)^{\frac{1}{3}}.
\]
Then it follows from \eqref{Y-8} that for $1\leq j\leq i-1$,
\begin{equation}\label{eq:d-512}
\|w^{m+1}_{j, +}-w^m_{j, +}\|_{L^{\infty}(\tilde{\Omega}_+)}\leq C_M\epsilon\sum^n_{l=1}\|w^{m}_{l,\pm}
-w^{m-1}_{l, \pm}\|_{L^{\infty}(\tilde{\Omega}_{\pm})}.
\end{equation}

Analogously, we can also show that for $i+1\leq j\leq n$,
\begin{equation}\label{YH-11}
\begin{split}
\|w^{m+1}_{j, -}-w^m_{j, -}\|_{L^{\infty}(\tilde{\Omega}_{-})}\leq C_M\epsilon\sum^n_{l=1}\|w^{m}_{l,\pm}
-w^{m-1}_{l, \pm}\|_{L^{\infty}(\tilde{\Omega}_{\pm})}.
\end{split}
\end{equation}

Finally, we  estimate $(w^{m+1}_{j, -}-w^m_{j, -})|_{j=1,\cdots, i-1}$ and $(w^{m+1}_{j, +}-w^m_{j, +})|_{j=i+1,\cdots, n}$.
Let $\mu_j(z, t)=w^{m+1}_{j, -}-w^m_{j, -}$, then $\mu_j(z, t)$ satisfies that
\begin{equation}\label{eq:3.67}
\begin{split}
\begin{cases}
\partial_t \mu_j+(\lambda_j(w^m_{-})-\sigma^m(t))\partial_z \mu_j=\Big(\lambda_j(w^{m-1}_{-})-\lambda_j(w^m_{-})+\sigma^m(t)-\sigma^{m-1}(t)\Big)\partial_zw^m_{j, -}\\[5pt]
\quad\displaystyle-\sum_{k\neq i, j}p_{jk}(w^m_{-})\Big\{\partial_t(w^m_{k, -}-w^{m-1}_{k, -})+(\lambda_j(w^m_{-})-\sigma^m(t))\partial_z(w^m_{k, -}-w^{m-1}_{k, -})\\[5pt]
\qquad+\Big(\lambda_j(w^m_{-})-\lambda_j(w^{m-1}_{-})
-\sigma^m(t)+\sigma^{m-1}(t)\Big)\partial_zw^{m-1}_{k, -}\Big\}\\[6pt]
\quad-\displaystyle\sum_{k\neq i, j}\Big(p_{jk}(w^m_{-})-p_{jk}(w^{m-1}_{-})\Big)
(\partial_tw^{m-1}_{k, -}+(\lambda_j(w^{m-1}_{-})-\sigma^{m-1}(t))\partial_zw^{m-1}_{k, -}),\\[5pt]
\mu_j(z, T_{\epsilon})=0.
\end{cases}
\end{split}
\end{equation}
Suppose that the backward $j$-th characteristics $\xi^{m+1}_j=\xi^{m+1}_j(z, t, s)$ of
\eqref{eq:3.67} through the point $(z, t)$ intersects
with $z$-axis before meeting $t$-axis.
By integrating $\eqref{eq:3.67}_1$ along the characteristics and making direct computations, one has
\begin{equation*}
\begin{split}
|\mu_j|\leq& \sum_{k\neq i, j}|p_{jk}(w^m_{-})|w^m_{k, -}-w^{m-1}_{k, -}|+C_M\epsilon(t-T_\epsilon)
|w^m_{\pm}-w^{m-1}_{\pm}|+C_M\epsilon(t-T_{\epsilon})^{\frac13}\sum_{k\neq i, j}|w^m_{k, -}-w^{m-1}_{k, -}|\\
\leq&C_M\epsilon\sum^n_{l=1}|w^m_{l, \pm}-w^{m-1}_{l, \pm}|, \qquad j=1, \cdots, i-1.
\end{split}
\end{equation*}
Otherwise, if $\xi^{m+1}_j=\xi^{m+1}_j(z, t, s)$ intersects $t$-axis at the point $(0, s)$ with $s\geq T_{\epsilon}$, then
\[
|\mu_j(z, t)|\leq |w^{m+1}_{j, -}(0-, s)-w^m_{j, -}(0-, s)|+C_M\epsilon\sum^n_{l=1}\|w^m_{l,\pm}-w^{m-1}_{l, \pm}\|_{L^{\infty}
(\tilde{\Omega}_{\pm})}.
\]
Note that
\[
|w^{m+1}_{j, -}(0-, s)-w^m_{j, -}(0-, s)|\leq |w^{m+1}_{j, +}(0+, s)-w^m_{j, +}(0+, s)|+|[w^{m+1}_j-w^m_j]|
\]
and
\[
\begin{split}
&|[w^{m+1}_j-w^m_j]|\\
=&\Big|\mathcal{F}_j(w^{m+1}_{1, +}, \cdots, w^{m+1}_{i, +}, w^{m+1}_{i, -}, \cdots, w^{m+1}_{n, -})
[w^{m+1}_i]^3-\mathcal{F}_j(w^{m}_{1, +}, \cdots, w^{m}_{i, +}, w^{m}_{i, -}, \cdots, w^{m}_{n, -})
[w^{m}_i]^3\Big|\\
=&\Big|\mathcal{F}_j(w^{m+1}_{1, +}, \cdots, w^{m+1}_{n, -})[w^{m+1}_i-w^m_i]\Big([w^{m+1}_i]^2+[w^m_i][w^{m+1}_i]
+[w^m_i]^2\Big)\Big|\\
&+C_M|[w^m_i]^3|\Big(\sum_{1\leq j\leq i}|w^{m+1}_{j, +}-w^m_{j, +}|+\sum_{i\leq j\leq n}|w^{m+1}_{j, -}-w^m_{j, -}|\Big).
\end{split}
\]
In addition,
\[
\begin{split}
|[w^{m+1}_i-w^m_i]|\leq& |w^{m+1}_{i, +}(0+, t)-w^m_{i, +}(0+, t)|+|w^{m+1}_{i, -}(0-, t)-w^m_{i,-}(0-, t)|\\
 \leq& \Big(2\ln\frac32+C_M\sqrt{t-T_\epsilon}\Big)|w^m_{i, \pm}-w^{m-1}_{i, \pm}|+C_M\sum_{k\neq i}|w^m_{k, \pm}
-w^{m-1}_{k, \pm}|,\\[5pt]
|[w^m_i]|=&|w^m_{i, +}-w^m_{i, -}|\leq |w^m_{i, +}(0+, t)-w^0_{i, +}(0+, t)|+|w^0_{i, +}(0+, t)-w^0_{i, -}(0-, t)|\\
&+|w^0_{i, -}(0-, t)-
w^m_{i, -}(0-, t)|\leq C_M\epsilon(t-T_{\epsilon})^{\frac12},\\
\end{split}
\]
then
\[
\begin{split}
|[w^{m+1}_j-w^m_j]|\leq &C_M\epsilon^2(t-T_\epsilon)\Big(\big(2\ln\frac32
+C_M\sqrt{t-T_\epsilon}\big)|w^m_{i, \pm}-w^{m-1}_{i, \pm}|
+C_M\sum_{k\neq i}\Big|w^m_{k, \pm}-w^{m-1}_{k, \pm}\Big|\Big)\\
&+C_M\epsilon^3(t-T_\epsilon)^{\frac32}\sum^n_{l=1}
|w^m_{l, \pm}-w^{m-1}_{l, \pm}|.
\end{split}
\]
On the other hand, as shown in \eqref{eq:d-512}, one has
\[
\begin{split}
&|w^{m+1}_{j, +}(0+, s)-w^m_{j, +}(0+, s)|\leq C_M\epsilon\sum^n_{l=1}|w^m_{l, \pm}-w^{m-1}_{l, \pm}|, \quad 1\leq j\leq i-1.
\end{split}
\]
Hence, it holds that for $1\leq j\leq i-1$,
\begin{equation}\label{YH-14}
\begin{split}
|w^{m+1}_{j, -}(z, t)-w^m_{j, -}(z, t)|
\leq C_M\epsilon\sum^n_{l=1}\|w^m_{l, \pm}-w^{m-1}_{l, \pm}\|_{L^{\infty}(\tilde{\Omega}_{\pm})}.
\end{split}
\end{equation}
Similarly, one can treat the estimate of $w^{m+1}_{j, +}(z, t)-w^m_{j, +}(z, t)$ for $i+1\leq j\leq n$.

In conclusion, we obtain that
\[
\begin{split}
&\|w^{m+1}_{i, \pm}-w^m_{i, \pm}\|_{L^{\infty}(\tilde{\Omega}_{\pm})}
\leq
\Big(2\ln\frac32+C_M\sqrt{t-T_\epsilon}\Big)\|w^m_{i, \pm}-w^{m-1}_{i, \pm}\|_{L^{\infty}(\tilde{\Omega}_{\pm})}+C_M
\sum_{k\neq i}\|w^{m}_{k, \pm}-w^{m-1}_{k, \pm}\|_{L^{\infty}(\tilde{\Omega}_{\pm})},\\
&\sum_{j\neq i}\|w^{m+1}_{j, \pm}(z, t)-w^m_{j, \pm}(z, t)\|_{L^{\infty}(\tilde{\Omega}_{\pm})}\leq C_M\epsilon\sum^n_{k=1}\|w^m_{k, \pm}
-w^{m-1}_{k, \pm}\|_{L^{\infty}(\tilde{\Omega}_{\pm})}.
\end{split}
\]
If $\epsilon>0$ is small and $t-T_{\ve}\le\delta_0$ with $\delta_0$ being suitably small holds such that
\[
2\ln\frac32+C_M\sqrt{t-T_\epsilon}+C_M(C_M+1)\epsilon<1-\epsilon,\,\quad (C_M+1)^2\epsilon<1,
\]
then it holds that
\begin{equation}\label{eq:5.15d}
\begin{split}
&\|w^{m+1}_{i, \pm}-w^m_{i, \pm}\|_{L^{\infty}(\tilde{\Omega}_{\pm,\delta_0})}+\sum_{j\neq i}(C_M+1)\|w^{m+1}_{j, \pm}-w^m_{j, \pm}\|_{L^{\infty}(\tilde{\Omega}_{\pm,\delta_0})}\\
\leq&\Big(2\ln\frac32+C_M\sqrt{t-T_\epsilon}+C_M(C_M+1)\epsilon\Big)\|w^m_{i, \pm}-w^{m-1}_{i, \pm}\|_{L^{\infty}(\tilde{\Omega}_{\pm,\delta_0})}\\
&\qquad+\Big(C_M+C_M(C_M+1)\epsilon\Big)\sum_{k\neq i}\|w^m_{k, \pm}-w^{m-1}_{k, \pm}\|_{L^{\infty}(\tilde{\Omega}_{\pm,\delta_0})}\\
\leq&(1-\epsilon)\Big(\|w^m_{i, \pm}-w^{m-1}_{i, \pm}\|_{L^{\infty}(\tilde{\Omega}_{\pm,\delta_0})}
+\sum_{k\neq i}(C_M+1)\|w^m_{k, \pm}-w^{m-1}_{k, \pm}\|_{L^{\infty}(\tilde{\Omega}_{\pm,\delta_0})}\Big).
\end{split}
\end{equation}

Therefore, the proof of Lemma \ref{l:3.6} is completed.
\end{proof}

{\bf Proof of Theorem \ref{T:2.14}.} By Lemma \ref{l:3.6}, we know that there exist $\sigma(t)\in C[T_\epsilon, T_\epsilon+\delta_0]$ and
$w_{\pm}(z, t)\in C(\tilde \Omega_{\pm,\delta_0})$ such that $\sigma^m(t)$ converges to  $\sigma(t)$ uniformly in
$[T_\epsilon, T_\epsilon+\delta_0]$
and $w^m_{\pm}(z, t)$  converges to   $w_{\pm}(z, t)$ uniformly in $\tilde \Omega_{\pm,\delta_0}$, respectively.
In addition,
we can similarly show that $\p_{t,z} w^m_{\pm}(z, t)$  converges to   $\p_{t,z}  w_{\pm}(z, t)$ uniformly in any closed subset of $\tilde \Omega_{\pm,\delta_0}$.
By Lemma \ref{l:3.6} and Lemma \ref{l:3.5},  $\p_{t,z}  w^m_{\pm}(z, t)$ are equicontinuous on  $z$ for any fixed $t\in (T_{\ve}, T_{\ve}+\delta_0)$
in $\tilde \Omega_{\pm,\delta_0}$ respectively,
which means that $w_{\pm}(0\pm, t)$ exist for $t\in (T_{\ve}, T_{\ve}+\delta_0)$ and $(\phi(t), w_{\pm}(z, t))$
satisfies \eqref{eq:3.30}.
Therefore, Theorem \ref{T:2.14} is proved by Lemma \ref{l:3.3} and Lemma \ref{l:3.5}
as well as the entropy condition \eqref{eq:3.28}.

{\bf Proof of Theorem \ref{DY-1}.} By Theorem \ref{T:2.14} and Lemma \ref{Y-4}, the results in Theorem \ref{DY-1}
can be obtained directly.

\section{Applications of Theorem \ref{DY-1}}

In this section, some applications of Theorem \ref{DY-1} are given.
Firstly, let us consider the initial value problem of 2-D supersonic steady full compressible Euler equations
\begin{equation}\label{FFF-1}
\begin{cases}
\partial_1(\rho u_1)+\partial_2(\rho u_2)=0,\\[5pt]
\partial_1(\rho u_1^2+P)+\partial_2(\rho u_1u_2)=0,\\[5pt]
\partial_1(\rho u_1 u_2)+\partial_2(\rho u_2^2+P)=0,\\[5pt]
\partial_1((\rho e+\f12 \rho |u|^2+P)u_1)+
\partial_2((\rho e+\f12 \rho |u|^2+P)u_2)=0,\\[5pt]
\rho(0, x_2)=\bar\rho+\ve \rho_0(x_2), u_1(0, x_2)=q_0+\ve u_1^0(x_2),
u_2(0, x_2)=\ve u_2^0(x_2),\\[5pt]
S(0,x_2)=\bar S+\ve S_0(x_2),
\end{cases}
\end{equation}
where $x=(x_1, x_2)\in\mathbb{R}^2$, $(\partial_{x_1}, \partial_{x_2})=(\partial_1, \partial_2)$,
$\ve>0$ is sufficiently small, $u=(u_1,u_2)^{\top}$,
$\rho$, $P$, $e$ and  $S$ are the velocity, density, pressure, internal energy and specific entropy, respectively.
The pressure function $P=P(\rho, S)$ and the
internal energy function $e=e(\rho,S)$ are smooth in their arguments, in particular, $\p_{\rho}P(\rho, S)>0$
and $\p_{S}e(\rho, S)>0$ for $\rho>0$. One sometimes writes the state equations as
$\rho=\rho(P,S)$ and $e=e(P,S)$.
In addition, $\bar\rho, q_0$ and $\bar S$ are constants with $q_0>\bar c=
c(\rho,S)|_{(\rho,S)=(\bar\rho,\bar S)}$ and $c(\rho,S)=\sqrt{\p_{\rho}P(\rho, S)}$,
and $(\rho_0(x_2), u_1^0(x_2), u_2^0(x_2), S_0(x_2))\in C_0^{\infty}(\Bbb R)$. Note that \eqref{FFF-1}
is symmetric hyperbolic with respect to the supersonic $x_1-$direction and the unknown functions
$(P,u_1,u_2,S)^T$ (see \cite{CF}). It follows from a direct computation that the system in \eqref{FFF-1} has four real eigenvalues
$$\lambda_1=\ds\f{u_1u_2-c(\rho,S)\sqrt{u_1^2+u_2^2-c^2(\rho,S)}}{u_1^2-c^2(\rho, S)}<\lambda_{2, 3}=\f{u_2}{u_1}
<\lambda_4=\ds\f{u_1u_2+c(\rho,S)\sqrt{u_1^2+u_2^2-c^2(\rho,S)}}{u_1^2-c^2(\rho, S)}$$
and is genuinely nonlinear with respect to $\lambda_1, \lambda_4$.

Secondly, let us consider the Cauchy problem of the 1-D MHD equations under Lagrangian coordinate
\begin{equation}\label{FFF-2}
\begin{cases}
\partial_tv-\partial_xu=0,\\[5pt]
\partial_t u+\partial_xP+\p_x(H_y^2+H_z^2)=0,\\[5pt]
\partial_tH_y+\ds\f{H_y}{v}\p_xu=0,\\[5pt]
\partial_tH_z+\ds\f{H_z}{v}\p_xu=0,\\[5pt]
\partial_tS=0,\\[5pt]
v(x, 0)=\bar v+\ve v_0(x), u(x, 0)=\ve u_0(x),
H_y(x, 0)=\ve H_y^0(x),\\[5pt]
H_z(x, 0)=\ve H_z^0(x), S(x, 0)=\bar S+\ve S_0(x),
\end{cases}
\end{equation}
where $v, u, H_y, H_z$ and $S$ stand for the specific volume, velocity, components of magnetic field in
$y-$direction and $z-$direction, and energy respectively. The equation of state is
$P=P(v,S)=Av^{-\gamma}e^{\f{S}{c_v}}$ with $A, c_v$ and $\gamma>1$ being positive constants. In addition, $\bar v>0$ and $\bar S$ are constants,
$(v_0(x), u_0(x), H_y^0(x),H_z^0(x),S_0(x))\in C_0^{\infty}(\Bbb R)$. Note that \eqref{FFF-2}
comes from the 1-D mode of MHD transverse flows in some process of geophysics or astrophysics (see \cite{Rammaha} or \cite{Dong}).
By direct computations, it is known that \eqref{FFF-2} has five real eigenvalues
$$\lambda_1=-\sqrt{-\p_{v}P+\f{2(H_y^2+H_z^2)}{v}}<\lambda_{2, 3, 4}=0<\lambda_5=\sqrt{-\p_{v}P+\f{2(H_y^2+H_z^2)}{v}}$$
and is genuinely nonlinear with respect to $\lambda_1, \lambda_5$.

Thirdly, under the planar symmetry, the elastic wave $u(x, t)=(u_1(x, t), u_2(x, t),u_3(x, t))^{\top}$
satisfies (see \cite{Ag} for the physical background)
\begin{equation}\label{FFF-3}
\begin{cases}
\partial_t^2u_1-c_1^2\partial_x^2u_1=\sigma_0\p_x((\p_xu_1)^2)
+\sigma_1\p_x((\p_xu_2)^2)+\sigma_1\p_x((\p_xu_3)^2),\\[5pt]
\partial_t^2u_2-c_2^2\partial_x^2u_2=2\sigma_1\p_x(\p_xu_1\p_xu_2),\\[5pt]
\partial_t^2u_3-c_2^2\partial_x^2u_3=2\sigma_1\p_x(\p_xu_1\p_xu_3),\\[5pt]
\end{cases}
\end{equation}
where $c_1>c_2>0$ and $\sigma_0\sigma_1\not=0$. Set $v=(v_1,v_2,v_3,v_4,v_5,v_6)^T=(\p_xu_1,\p_xu_2,\p_xu_3,\p_tu_1,\p_tu_2,\p_tu_3)^T$.
Then the system \eqref{FFF-3} can be rewritten by
\begin{align}\label{FFF-4}
\partial_tv+\p_xf(v)=0
\end{align}
with $f(v)=-(v_4, v_5, v_6, c_1^2v_1+\sigma_0v_1^2+\sigma_1v_2^2+\sigma_1v_3^2, c_2^2v_2+2\sigma_1v_1v_2, c_2^2v_3+2\sigma_1v_1v_3)^T$.
At this time, the corresponding $6\times6$ matrix $F(v)$ in \eqref{Y-1} is
\[
\left(
\begin{matrix}
0&0&0&-1&0&0\\[5pt]
0&0&0&0&-1&0\\[5pt]
0&0&0&0&0&-1\\[5pt]
-c_1^2-2\sigma_0v_1&-2\sigma_1v_2&-2\sigma_1v_3&0&0&0\\[5pt]
-2\sigma_1v_2&-c_2^2-2\sigma_1v_1&0&0&0&0\\[5pt]
-2\sigma_1v_3&0&-c_2^2-2\sigma_1v_1&0&0&0\\[5pt]
\end{matrix}
\right).
\]
When \eqref{FFF-4} is imposed the following initial data
\begin{align}\label{FFF-5}
v(x, 0)=(\ve v_1^0(x), q_1+\ve v_2^0(x), q_2+\ve v_3^0(x),\ve v_4^0(x),\ve v_5^0(x),\ve v_6^0(x))
\end{align}
with $(q_1, q_2)\not=0$, $q_1^2+q_2^2<\ds\f{c_1^2c_2^2}{4\sigma_1^2}$ and $(v_1^0(x), v_2^0(x), v_3^0(x), v_4^0(x), v_5^0(x), v_6^0(x))\in C_0^{\infty}(\Bbb R)$,
it is known that the $6\times 6$ matrix $F(v)|_{v=\bar v=(0,q_1,q_2,0,0,0)}$ has six distinct real eigenvalues
\begin{equation*}
\begin{split}
&\lambda_1=-\sqrt{\f{c_1^2+c_2^2+\sqrt{(c_2^2-c_1^2)^2+16\sigma_1^2(q_1^2+q_2^2)}}{2}}
<\lambda_2=-c_2\\
&<\lambda_3=-\sqrt{\f{c_1^2+c_2^2-\sqrt{(c_2^2-c_1^2)^2+16\sigma_1^2(q_1^2+q_2^2)}}{2}}\\
&<\lambda_4=-\lambda_3
<\lambda_5=-\lambda_2<\lambda_6=-\lambda_1,
\end{split}
\end{equation*}
and \eqref{FFF-4} is genuinely nonlinear with respect to all the eigenvalues $\lambda_i$ ($1\le i\le 6$)
for small perturbations of $\bar v$.

Lastly, the equations of 3-D ideal compressible magnetohydrodynamics (MHD) (see \cite{Ag} or \cite{YSL}) are
\begin{equation}\label{FFF-6-0}
\begin{cases}
\ds\p_t\rho+div(\rho u)=0,\\
\ds\p_t(\rho u)+div(\rho u\otimes u-H\otimes H)+\nabla(P+\f12|H|^2)=0,\\
\ds\p_tH-curl(u\times H)=0,\\
\ds div H=0,\\
\ds \p_t(\rho S)+div(\rho uS)=0,\\
\end{cases}
\end{equation}
where $(x,t)=(x_1,x_2,x_3,t)$, $\rho$ is the fluid density,
$u=(u_1,u_2,u_3)^{\top}$ is the fluid velocity, $H=(H_1,H_2,H_3)^{\top}$
is the magnetic field, $S$ is the entropy
and $P$ is the pressure satisfying the state equation $P=P(\rho,S)=A\rho^{\gamma}e^{\f{S}{c_v}}$
with $A, c_v$ and $\gamma>1$ being positive constants.
Let $(\rho, u, H, S)(x,t)=(\rho, u, H, S)(x_1,t)$ and $H_1=\bar H_1>0$
is a constant. Then \eqref{FFF-6-0} becomes the $7\times 7$ 1-D conservation law
\begin{equation}\label{FFF-6-10}
\begin{cases}
\ds\p_t\rho+\p_1(\rho u_1)=0,\\
\ds\p_t(\rho u_1)+\p_1(\rho u_1^2)+\p_1(P+\f12|H_2|^2+\f12|H_3|^2)=0,\\
\ds\p_t(\rho u_2)+\p_1(\rho u_1 u_2-\bar H_1H_2)=0,\\
\ds\p_t(\rho u_3)+\p_1(\rho u_1 u_3-\bar H_1H_3)=0,\\
\ds\p_tH_2+\p_1(u_1H_2-\bar H_1u_2)=0,\\
\ds\p_tH_3+\p_1(u_1H_3-\bar H_1u_3)=0,\\
\ds \p_t(\rho S)+\p_1(\rho u_1S)=0.\\
\end{cases}
\end{equation}
The initial data of \eqref{FFF-6-10} is imposed by
\begin{equation}\label{FFF-7}
\begin{split}
&u_1(x, 0)=\ve u_1^0(x), u_2(x, 0)=\ve u_2^0(x), u_3(x, 0)=\ve u_3^0(x), \rho(x, 0)=\bar\rho+\ve \rho_0(x),\\
&H_2(x, 0)=\bar H_2+\ve H_2^0(x), H_3(x, 0)=\bar H_3+\ve H_3^0(x), S(x, 0)=\bar S+\ve S_0(x)
\end{split}
\end{equation}
with the constants $\bar\rho>0$, $\bar H_2\bar H_3\not=0$ and $(u_1^0(x), u_2^0(x), u_3^0(x), \rho_0(x), H_2^0(x), H_3^0(x), S_0(x))\in C_0^{\infty}(\Bbb R)$. Set $v=(v_1,v_2,v_3,v_4,v_5,v_6,v_7)^T=(u_1,u_2,u_3,\rho,H_2,H_3,S)^T$. Then it follows from \eqref{FFF-6-10}
that
\begin{align}\label{FFF-8}
\partial_tv+A(v)\p_xv=0,
\end{align}
where $A(v)$ is a $7\times 7$ matrix, $A(v)|_{v=\bar v=(0,0,0,\bar\rho,\bar H_2,\bar H_3,\bar S)}$ has seven real distinct eigenvalues
\begin{equation*}
\begin{split}
&\ds\lambda_1=-\bigg\{\f{\mu_0}{2\bar\rho}({\bar H}_1^2+{\bar H}_2^2+{\bar H}_3^2)+\f{{\bar c}^2}{2}
+\f12\sqrt{(\f{\mu_0}{\bar\rho}({\bar H}_1^2+{\bar H}_2^2+{\bar H}_3^2)+{\bar c}^2)^2-\f{4\mu_0}{\bar\rho}{\bar H}_1^2{\bar c}^2}
\bigg\}^{\f12}\\
&<\lambda_2=-\sqrt{\f{\mu_0}{\bar\rho}}{\bar H}_1\\
&<\lambda_3=-\bigg\{\f{\mu_0}{2\bar\rho}({\bar H}_1^2+{\bar H}_2^2+{\bar H}_3^2)+\f{{\bar c}^2}{2}
-\f12\sqrt{(\f{\mu_0}{\bar\rho}({\bar H}_1^2+{\bar H}_2^2+{\bar H}_3^2)+{\bar c}^2)^2-\f{4\mu_0}{\bar\rho}{\bar H}_1^2{\bar c}^2}
\bigg\}^{\f12}\\
&<\lambda_4=0<\lambda_5=-\lambda_3<\lambda_6=-\lambda_2<\lambda_7=-\lambda_1
\end{split}
\end{equation*}
and \eqref{FFF-8} is genuinely nonlinear with respect to all the eigenvalues except $\lambda_4$
for small perturbations of $\bar v$.

Based on the analyses above, in terms of Theorem \ref{DY-1} and Remark \ref{CYY-5}, we can have the following conclusions.

\begin{theorem}\label{CYY-1}
Under the corresponding generic nondegenerate conditions \eqref{Eq:2.1}, around the resulting geometric blowup points,
problems \eqref{FFF-1}, \eqref{FFF-2},  \eqref{FFF-3} with \eqref{FFF-5} and \eqref{FFF-6-10} with \eqref{FFF-7}
admit  weak entropy solutions with 1-shock or 4-shock, 1-shock or 5-shock,
$i-$shock ($1\le i\le 6$) and $j-$shock ($1\le j\le 7$ but $j\not=4$), respectively.
Moreover, the analogous estimates in \eqref{CYY-2} and \eqref{CYY-3} hold.
\end{theorem}

\section{Appendix}

In this appendix, we prove the estimates \eqref{eq:3.53} and \eqref{eq:3.54}.
Note that for $T_{\epsilon}\leq s\leq t\le T_{\epsilon}+1$ and by the notation in \eqref{HC-1}, one has
\begin{equation}\label{YHC-01}
\begin{split}
&\xi^{m+1}_i(z, t; s)-z\\
=&\varphi(y, s)-\varphi(y, t)-(\phi^m(s)-\phi^m(t))\\
=&\partial_t\varphi(y_\epsilon, T_{\epsilon})(s-t)-\lambda_i(w^m(y_\epsilon, T_\epsilon))(s-t)+\partial^2_{ty}\varphi(y_\epsilon, T_\epsilon)(s-t)(y-y_\epsilon)\\
&+O(1)\Big((y-y_\epsilon)^2(s-t)+(s-T_\epsilon)^2-(t-T_\epsilon)^2\Big)\\
=&\partial^2_{ty}\varphi(y_\epsilon, T_\epsilon)(s-t)(y-y_\epsilon)+O(1)\Big((y-y_\epsilon)^2(s-t)
+(s-T_\epsilon)^2-(t-T_\epsilon)^2\Big),
\end{split}
\end{equation}
where $\xi^{m+1}_i(z, t; s)=\varphi(y, s)-\phi^m(s)$, $z=\varphi(y, t)-\phi^m(t)$,
$\partial_y\varphi(y_\epsilon, T_\epsilon)=\partial^2_y\varphi(y_\epsilon, T_\epsilon)=0$, $\phi^m(t)$ is the approximate shock wave curve, and
\[
\begin{split}
&y=y^{\epsilon}_+(x, t), \;\;|y^{\epsilon}_+(x, t)-y_{\epsilon}|\sim d^{\frac{1}{6}}_{\epsilon}\sim (t-T_{\epsilon})^{\frac{1}{2}},\\[3pt]
&\phi^m(t)=x_\epsilon+\lambda_i(w^m(x_\epsilon, T_\epsilon))(t-T_\epsilon)+O(1)(t-T_\epsilon)^2.
\end{split}
\]
It follows from the entropy condition \eqref{eq:3.28} that
\[
\frac{\text{d}\xi^{m+1}_i}{\text{d}s}=\lambda_i(w^m_+(\xi^{m+1}_i, s))-\sigma^m(s)<0,
\]
then for $T_{\epsilon}\leq s\leq t$, it holds that  $\xi^{m+1}_i(z, t; s)-z>0$,
along with $\partial^2_{ty}\varphi(y_\epsilon, T_\epsilon)<0$, one has
 \[
\quad  y^\epsilon_+(x, t)-y_\epsilon=k_0(t-T_\epsilon)^{\frac12}+O(1)(t-T_\epsilon),
 \]
 for some positive constant $k_0$.

In addition,
\[
\begin{split}
\xi^{m+1}_i(z, t; s)-z\geq & -\partial^2_{ty}\varphi(y_\epsilon, T_\epsilon)(t-s)\sqrt{t-T_\epsilon}-C(t-s)(t-T_\epsilon)\\
\geq & -\frac{1}{2}\partial^2_{ty}\varphi(y_\epsilon, T_\epsilon)(t-s)\sqrt{t-T_\epsilon}.
\end{split}
\]
Meanwhile, we  can analogously obtain
\begin{equation}\label{YCH-02}
\begin{split}
&\xi^{m+1}_i(z, t; s)+z\\
=&\partial_t\varphi(y_\epsilon, T_{\epsilon})(s+t-2T_\epsilon)-\lambda_i(w^m(y_\epsilon, T_\epsilon))(s+t-2T_\epsilon)+\partial^2_{ty}\varphi(y_\epsilon, T_\epsilon)(s+t-2T_\epsilon)(y-y_\epsilon)\\
&+O(1)\Big((s-T_\epsilon)^2+(t-T_\epsilon)^2\Big)\\
=&\partial^2_{ty}\varphi(y_\epsilon, T_\epsilon)(s-t)(y-y_\epsilon)+2\partial^2_{ty}\varphi(y_\epsilon, T_\epsilon)(t-T_\epsilon)(y-y_\epsilon)+O(1)\Big((s-T_\epsilon)^2+(t-T_\epsilon)^2\Big).
\end{split}
\end{equation}
Collecting \eqref{YHC-01} and \eqref{YCH-02} yields
\[
\begin{split}
(\xi^{m+1}_i(z, t; s))^2=&(\partial^2_{ty}\varphi(y_\epsilon, T_\epsilon))^2(s-t)^2(y-y_\epsilon)^2+z^2+2(\partial^2_{ty}\varphi(y_\epsilon, T_\epsilon))^2(t-T_\epsilon)^2(y-y_\epsilon)^2\\
&-2\partial^2_{ty}\varphi(y_\epsilon, T_\epsilon)(t-T_\epsilon)(y-y_\epsilon)z+O(1)\Big((t-T_\epsilon)^4+(s-T_\epsilon)^4\Big)\\
\geq &\frac{1}{20}z^2+\frac{4}{5}k^2_0(\partial^2_{ty}\varphi(y_\epsilon, T_\epsilon))^2(t-T_\epsilon)^3
+(\partial^2_{ty}\varphi(y_\epsilon, T_\epsilon))^2k^2_0(t-T_\epsilon)(s-t)^2,
\end{split}
\]
where we have used the inequality
\[
-2\partial^2_{ty}\varphi(y_\epsilon, T_\epsilon)(t-T_\epsilon)(y-y_\epsilon)z\geq -\frac{19}{20}z^2-\frac{20}{19}\Big(\partial^2_{ty}\varphi(y_\epsilon, T_\epsilon)(t-T_\epsilon)(y-y_\epsilon)\Big)^2.
\]
From \eqref{eq:3.2}, one has
\[
6h\partial_t h\partial_y h-A'(t)\partial_y h+(3h^2-A)\partial^2_{yt}h(y, t)=\partial^2_{yt}\varphi(y, t)
\]
and
\[
6(\partial_y h)^3+18h\partial_y h\partial^2_y h+(3h^2(y, t)-A(t))\partial^3_y h(y, t)=\partial^3_y\varphi(y, t).
\]
This, together with $h(y_\epsilon, T_\epsilon)=A(T_\epsilon)=0$ and $\partial_y h(y_\epsilon, T_\epsilon)=A'(T_\epsilon)=1$, it holds that
\[
\partial^2_{yt}\varphi(y_\epsilon, T_\epsilon)=-1, \quad \partial^3_y\varphi(y_\epsilon, T_\epsilon)=6.
\]
Therefore, it yields that for $T_{\epsilon}\leq s\leq t\leq T_\epsilon+1$, there is a positive constant $C<1$, independent of the approximate solution $(w^m, \sigma^m)$, such that
\begin{equation}\label{YHC}
(s-T_\epsilon)^3+(\xi^{m+1}_i(z, t; s))^2\geq C((t-T_\epsilon)^3+z^2).
\end{equation}

Next, we  prove the estimate \eqref{eq:3.54}. Note that
\begin{equation}\label{YHC-5}
\begin{split}
\int^t_{T_{\epsilon}}\partial_{\eta}\lambda_i(w^m_+)(\eta, s)|_{\eta=\xi^{m+1}_i(z, t; s)}\text{d}s=&
\int^t_{T_\epsilon}\frac{\partial^2\xi^{m+1}_i}{\partial s\partial z}\frac{\partial z}{\partial \xi^{m+1}_i}\text{d}s
=\int^t_{T_\epsilon}\partial_s(\ln
|\partial_z\xi^{m+1}_i(z, t; s)|)\text{d}s\\
=&-ln|\partial_z\xi^{m+1}_i(z, t; T_\epsilon)|.
\end{split}
\end{equation}
Due to
\begin{equation}\label{eq:6.6}
h^3(y, t)-A(t)h(y, t)+B(t)=z+\phi^m(t)=\varphi(y, t),
\end{equation}
then one can obtain
\[
\partial_y z=(3h^2(y, t)-A(t))\partial_y h(y, t).
 \]
In addition, from formula \eqref{d1} of the real root to the cubic algebraic equation \eqref{eq:6.6} on $h$,
it is known that $\varphi(y, t)-B(t)$ and $h(y, t)$ have the same sign, then $h^2(y, t)-A(t)>0$ and further
\begin{equation}\label{Eq:6.6}
\begin{split}
|3h^2(y, t)-A(t)|=&|3(h^2(y, t)-A(t))+2A(t)|=3|h^2(y, t)-A(t)|+2A(t)\\
=&3|h^{-1}(y, t)||\varphi(y, t)-B(t)|+2A(t)\\
\geq&2A(t)=2(t-T_\epsilon)+C(t-T_\epsilon)^2,
\end{split}
\end{equation}
On the other hand,
\[
h^3(y, T_\epsilon)-A(T_\epsilon)h(y, T_\epsilon)+B(T_\epsilon)=\xi^{m+1}_i(z, t; T_\epsilon)+\phi^m(T_\epsilon)=\varphi(y, T_\epsilon),
\]
then we have
\begin{equation}\label{eq:6.7}
\begin{split}
\partial_y\xi^{m+1}_i(z, t; T_\epsilon)=&3h^2(y, T_\epsilon)\partial_y h(y, T_\epsilon)=\partial_y\varphi(y, T_\epsilon)\\
=&\partial_y\varphi(y_\epsilon, T_\epsilon)+\partial^2_y\varphi(y_\epsilon, T_\epsilon)(y-y_\epsilon)
+\frac{1}{2}\partial^3_y\varphi(y_\epsilon,
T_\epsilon)(y-y_\epsilon)^2+O(1)(y-y_\epsilon)^3\\
=&3(y-y_\epsilon)^2+C(t-T_\epsilon)^{\frac32}.
\end{split}
\end{equation}
Since
\begin{equation}\label{YHC-6}
\begin{split}
\partial_z\xi^{m+1}_i(z, t; T_\epsilon)=&\partial_y\xi^{m+1}_i(z, t; T_\epsilon)
(\partial_y z)^{-1}|_{t=T_\epsilon}\\[3pt]
=&\frac{3h^2(y, T_\epsilon)\partial_y h(y, T_\epsilon)}{(3h^2(y, t)-A(t))\partial_y h(y, t)}=\frac{\partial_y\varphi(y, T_\epsilon)}{\partial_y\varphi(y,t)}
\end{split}
\end{equation}
and
\[
\begin{split}
&\partial^2_{y}\varphi(y, T_\epsilon)\partial_y\varphi(y, t)-\partial_y\varphi(y, T_\epsilon)\partial^2_y\varphi(y, t)\\
=&\Big(\partial^3_y\varphi(y_\epsilon, T_\epsilon)(y-y_\epsilon)+O(1)(y-y_\epsilon)^2\Big)\Big(\partial^2_{yt}\varphi(y_\epsilon, T_\epsilon)(t-T_\epsilon)
+\frac{1}{2}\partial^3_y\varphi(y_\epsilon, T_\epsilon)(y-y_\epsilon)^2\\
&+O(1)(y-y_\epsilon)(t-T_\epsilon)\Big)-\Big(\frac{1}{2}\partial^3_y\varphi(y_\epsilon, T_\epsilon)(y-y_\epsilon)^2+O(1)(y-y_\epsilon)^3\Big)\Big(\partial^3_y\varphi(y_\epsilon, T_\epsilon)(y-y_\epsilon)\\
&+O(1)(t-T_\epsilon)\Big)\\
=&-6(y-y_\epsilon)(t-T_\epsilon)+O(1)(t-T_\epsilon)^2<0\quad\text{for small $(t-T_\epsilon)$},
\end{split}
\]
then  one can derive that $\partial_z\xi^{m+1}_i(z, t; T_\epsilon)$ is decreasing with respect to $y$.

Note that
\[
y-y_\epsilon=k_0(t-T_\epsilon)^{\frac12}+O(1)(t-T_\epsilon),\quad k_0>0.
\]
\begin{itemize}
\item If $0<k_0\leq 1$,
then it follows from  \eqref{Eq:6.6}-\eqref{YHC-6} that
\[
|\partial_z\xi^{m+1}_i(z, t; T_\epsilon)|\leq \frac{3}{2}+C_M\sqrt{t-T_\epsilon}.
\]

\item If $k_0>1$, then we can choose $y_*<y$ with
\[
y_*=y_\epsilon+(t-T_\epsilon)^{\frac{1}{2}}+O(1)(t-T_\epsilon).
\]
Since $\partial_z\xi^{m+1}(z, t; T_\epsilon)$ is decreasing with respect to $y$, then
\[
\partial_z\xi^{m+1}_i(z, t; T_\epsilon)=\frac{3h^2(y, T_\epsilon)\partial_y h(y, T_\epsilon)}{(3h^2(y, t)-A(t))\partial_y h(y, t)}\leq
\frac{3h^2(y_*, T_\epsilon)\partial_y h(y_*, T_\epsilon)}{(3h^2(y_*, t)-A(t))\partial_y h(y_*, t)}.
\]

In addition, it follows from \eqref{YHC-6} and
\[
\begin{split}
\partial_y\varphi(y, t)=&\partial_y\varphi(y_\epsilon, T_\epsilon)+\partial^2_{yt}\varphi(y_\epsilon, T_\epsilon)(t-T_\epsilon)
+\partial^2_{y}\varphi(y_\epsilon, T_\epsilon)(y-y_\epsilon)\\[5pt]
&+\frac{1}{2}\partial^3_y\varphi(y_\epsilon, T_\epsilon)(y-y_\epsilon)^2+O(1)(t-T_\epsilon)^2\\[5pt]
=&(3k^2_0-1)(t-T_\epsilon)+O(1)(t-T_\epsilon)^2>0,\\[3pt]
\partial_y\varphi(y, T_\epsilon)=&\partial_y\varphi(y_\epsilon, T_\epsilon)+\partial^2_y\varphi(y_\epsilon, T_\epsilon)(y-y_\epsilon)^2
+\frac{1}{2}\partial^3_y\varphi(y_\epsilon, T_\epsilon)(y-y_\epsilon)^2+O(1)(y-y_\epsilon)^3\\
=&3(y-y_\epsilon)^2+O(1)(y-y_\epsilon)^3>0
\end{split}
\]
that $\partial_z\xi^{m+1}_i(z, t; T_\epsilon)>0$ holds.
Therefore, it yields
\[
|\partial_z\xi^{m+1}_i(z, t; T_\epsilon)|
\leq \frac{3h^2(y_*, T_\epsilon)\partial_y h(y_*, T_\epsilon)}{(3h^2(y_*, t)-A(t))\partial_y h(y_*, t)}
\leq  \frac{3}{2}+C_M\sqrt{t-T_\epsilon}.
\]
\end{itemize}

In conclusion, together with \eqref{YHC-5}, the proof of \eqref{eq:3.54} is completed.

\vskip 0.3 true cm

{\bf Acknowledgements}. Yin Huicheng wishes to express his deep gratitude to Professor Xin Zhouping, Chinese
University of Hong Kong, and Professor Chen Shuxing, Fudan University, Shanghai, for their constant interests
in this problem and many fruitful discussions in the past. In addition, the authors would like to thank the
editor and the referees very much for their invaluable suggestions
and comments that lead to the essential improvement of our paper.

\vskip 0.5 true cm






\end{document}